\documentclass[11pt,twoside]{article}
\usepackage{mathrsfs}
\usepackage{amsfonts}
\usepackage{bbm}
  \usepackage{url}
  \usepackage{amssymb,amsmath}
     \usepackage{graphicx}
   \usepackage[a4paper, margin=2cm]{geometry}

  \input xypic
  \xyoption{all}

 \usepackage{tikz}
\usetikzlibrary{arrows}

 \usepackage{times}

  \usepackage{amsmath, amsthm, amsfonts}

\numberwithin{equation}{section}
\newtheorem{theorem}{Theorem}[section]

\newtheorem{lemma}[theorem]{Lemma}

\newtheorem{remark}[theorem]{Remark}

\def \fs {\mathbf{s}}
\def \ft {\mathbf{t}}
\def \E{{\mathcal E}}

\def \inv{^{-1}}
\def  \supp{\mbox{supp}}

\def \fc {\mathfrak{c}}
\def \half{\frac{1}{2}}

\def \p{\partial}
\def  \wkr {W^{k,2,\alpha}_{r,u_{(r)}} }

\def  \cwk {\mathcal{W}^{k,2,\alpha}_{u} }
\def  \wka {W^{k,2,\alpha}_{u} }

\def  \lka {L_{u}^{k-1,2,\alpha} }
\def \lkar {L^{k-1,2,\alpha}_{r,u_{(r)}} }
\def \ker {\mbox{ker}}

\def \E{{\mathcal E}}
\def \v {\vskip 0.1in}
\def \n {\noindent}

\def \im {\mbox{im}}
\def \LP {P_{\mathbf r}}
\def \LH {H_{\mathbf r}}

\begin{document}

  \begin{center}
 {\LARGE \bf A Finite Rank Bundle over $J$-Holomorphic map Moduli Spaces}
\v
  {\large An-Min Li and Li Sheng}\footnote{partially  supported by a NSFC grant}
   \footnote{anminliscu@126.com, lshengscu@gmail.com}

{Department of Mathematics, Sichuan University
        Chengdu, PRC}
        \end{center}

\v\v
\begin{abstract}
We study a finite rank bundle $\mathbf{F}$ over a neighborhood of $J$-Holomorphic map Moduli Spaces, prove the exponential decay of the derivative of the gluing maps for $\mathbf{F}$ with respect to the gluing parameter.
\end{abstract}

\section{\bf Introduction and Preliminary}\label{introduction}

In \cite{S}, \cite{R3} and \cite{LR} the authors introduced a finite rank bundle $\mathbf{F}$ over a neighborhood of $J$-Holomorphic map Moduli Spaces. This bundle plays an important role in the study of the Gromov-Witten theory and the relative Gromov-Witten theory. In this paper we study some local analysis properties of this bundle.
\v
Let $(M,\omega,J)$ be a closed $C^{\infty}$ symplectic manifold of dimension $2m$ with $\omega$-tame almost complex structure $J$, let $(\Sigma,j,\mathbf{y})$ be a smooth Riemann surface of genus $g$ with $n$ marked points with $n>2-2g$.
We fix a local coordinate system $\psi: U\rightarrow \mathbf{A}$ for the Teichm\"uller space $\mathbf{T}_{g,n}$,
where $U\subset \mathbf{T}_{g,n}$ is a open set. Let $a_o=(j_o, \mathbf{y}_o)\in \mathbf{A}$, $u: \Sigma\to M$ be a $(j_o,J)$-holomorphic map. Then $\mathbf{F}$ can be viewed locally as a bundle over $\mathbf{A}\times W^{k,2}(\Sigma, u^*TM)$, denoted by $\widetilde{\mathbf{F}}$. In \S\ref{smoothness} we study the smoothness of $\widetilde{\mathbf{F}}$.
\v
In \S\ref{pregluing}
and \S\ref{diff gluing parameters} we study the gluing theory for $\widetilde{\mathbf{F}}$. Let $(\Sigma, j,{\bf y}, q)$ be a marked nodal Riemann surface with one nodal point $q$. We write $\Sigma=\Sigma_{1}\wedge\Sigma_{2}$. Let $u=(u_1,u_2)$, where $u_i:\Sigma_i\to M$ is a $(j_{oi},J)$-holomorphic map. We glue $\Sigma$ and $u$ at $q$ with gluing parameter $(r,\tau):=(r)$ to get $\Sigma_{(r)}$ and pregluing map $u_{(r)}:\Sigma_{(r)}\to M$. We have a gluing map from $\widetilde{\mathbf{F}}\mid_u$ to $\widetilde{\mathbf{F}}\mid_{u_{(r)}}$. We prove the exponential decay of the derivatives of the gluing maps with respect to the gluing parameter.

\subsection{\bf Metrics on $\Sigma$}\label{metric on surfaces} Let $(\Sigma,j,\mathbf{y})$ be a smooth Riemann surface of genus $g$ with $n$ marked points. In this paper we assume that $n>2-2g$, and $(g,n)\ne (1,1), (2,0)$. It is well-known that there is a unique complete hyperboloc metric $\mathbf{g}_0$ in $\Sigma\setminus \{{\bf y}\}$ of constant curvature $-1$ of finite volume, in the given conformal class $j$ ( see \cite{Wolp-1}).
Let $\mathbb H=\{\zeta=\lambda+\sqrt{-1}\mu|\mu>0\}$ be the half upper plane with the Poincare metric
$$
\mathbf{g}_0(\zeta)=\frac{1}{(Im(\zeta))^2}d\zeta d\bar\zeta.
$$
 Let
$$
\mathbb{D}=\frac{ \{\zeta \in \mathbb H|  Im (\zeta) \geq 1\}}{\zeta \sim \zeta + 1}
$$
be a cylinder, and $\mathbf{g}_0$ induces a metric on $\mathbb{D}$, which is still denoted  by $\mathbf{g}_0$. Let $z=e^{2\pi i\zeta}$, through which we identify $\mathbb{D}$ with $D(e^{-2\pi}):=\{z||z|< e^{-2\pi}\}$.
An important result is that for any punctured point $y_i$ there exists a
neighborhood $O_i$ of $y_i$ in $\Sigma$ such that
$$
(O_i\setminus \{y_i\},\mathbf{g}_0)\cong (D(e^{-2\pi})\setminus \{0\},\mathbf{g}_0),
$$
moreover, all $O_i$'s are disjoint with each other. Then we can view $D_{y_i}(e^{-2\pi})$ as a neighborhood of $y_i$ in $\Sigma$ and $z$ is a local complex coordinate on $D_{y_i}(e^{-2\pi})$ with $z(y_i)=0$. For any $c>0$ denote $$\mathbf{D}(c)=\bigcup D_{y_i}(c),\;\;\;\Sigma(c)=\Sigma\setminus \mathbf{D}(c).$$ Let $\mathbf{g}'=dzd\bar{z}$ be the standard Euclidean metric on each $D_{y_i}(e^{-2\pi})$. We fix a smooth cut-off function $\chi(|z|)$ to glue $\mathbf{g}_0$ and $\mathbf{g}'$, we get a smooth metric $\mathbf{g}$  in the given conformal class $j$ on $\Sigma$ such that
\[
\mathbf{g}=\left\{
\begin{array}{ll}
\mathbf{g}_0 \;\;\;\;\; on \;\;\Sigma\setminus \mathbf{D}(e^{-2\pi}),    \\  \\
\mathbf{g}'\;\;\;on\; \mathbf{D}(\frac{1}{2}e^{-2\pi})\;\;.
\end{array}
\right.
\]
\v
\v
 Let $\mathbf{g}^{c}=ds^2+d\theta^2$ be the  cylinder metric on each $D_{y_i}^{*}(e^{-2\pi})$, where $z=e^{s+2\pi\sqrt{-1}\theta}$.
We also define another metric $\mathbf{g}^\diamond$ on $\Sigma$ as above by glue $\mathbf{g}_0$ and $\mathbf{g}^c$, such that
\[
\mathbf{g}^{\diamond}=\left\{
\begin{array}{ll}
\mathbf{g}_0 \;\;\;\;\; on \;\;\Sigma\setminus \mathbf{D}(e^{-2\pi}),    \\  \\
\mathbf{g}^{c}\;\;\;on\; \mathbf{D}(\frac{1}{2}e^{-2\pi})\;\;.
\end{array}
\right.
\]

\v
The metric $\mathbf{g}$ (resp. $\mathbf{g}^{\diamond}$) can be generalized to marked nodal surfaces in a natural way. Let $(\Sigma, j,\mathbf{y})$ be a
marked nodal surfaces with $\mathfrak{e}$ nodal points $\mathbf{p}=(p_{1},\cdots,p_{\mathfrak{e}})$. Let $\sigma:\tilde{\Sigma}=\sum_{\nu=1}^{\mathfrak{r}}\Sigma_{\nu}\to \Sigma$ be the normalization. For every node $p_i$ we have a pair $\{\mathbf{a}_i, \mathbf{b}_i\}$. We view $\mathbf{a}_i$, $\mathbf{b}_i$ as marked points on $\tilde{\Sigma}$ and define the metric $\mathbf{g}_{\nu}$ (resp. $\mathbf{g}^{\diamond}_{\nu}$) for each $\Sigma_{\nu}$. Then we define
$$\mathbf{g}:=\bigoplus_{1}^{\nu} \mathbf{g}_{\nu},\;\;\;\;\;\mathbf{g}^{\diamond}:=\bigoplus_{1}^{\nu} \mathbf{g}^{\diamond}_{\nu}.$$

\v\v
\subsection{Teichm\"uller space}\label{sub_sect_Teich}

Denote by $\mathcal{J}(\Sigma)\subset End(T\Sigma)$ the manifold of all $C^{\infty}$ complex structures on $\Sigma$. Denote by $Diff^+(\Sigma)$ the group of orientation preserving $C^{\infty}$ diffeomorphisms of $\Sigma$, by $Diff^+_{0}(\Sigma)$ the identity component of $Diff^+(\Sigma).$ $Diff^+(\Sigma)$ acts on $\mathcal{J}(\Sigma)\times (\Sigma^{n} \setminus \Delta)$ by

$$\phi(j,\mathbf{y})=\left((d\phi_x)^{-1}J_{\phi(x)}d\phi_x, \varphi^{-1}\mathbf{y}\right)$$
for all $\phi\in Diff^+(\Sigma)$, $x\in \Sigma$, where $\Delta\subset \Sigma^{n}$ denotes the fat diagonal.
Put
$$\mathbf{P}:=\mathcal J(\Sigma)\times (\Sigma^{n} \setminus \Delta).$$
The orbit spaces are
\begin{align*}
\mathcal{M}_{g,n}=\left(\mathcal J(\Sigma)\times (\Sigma^{n} \setminus \Delta)\right)/ Diff^+(\Sigma),\;\;\;\mathbf{T}_{g,n}=\left(\mathcal J(\Sigma)\times (\Sigma^{n} \setminus \Delta)\right)/ Diff^+_{0}(\Sigma).
\end{align*}
$\mathcal{M}_{g,n}$ is called the Deligne-Mumford space, $\mathbf{T}_{g,n}$ is called the Teichm\"uller space.

\v

\v
Consider the principal fiber bundle
$
Diff^+_0(\Sigma) \to \mathbf{P}   \to \mathbf{T}_{g,n}
$
and the associated fiber bundle
$$
\pi_{\mathbf{T}}:   \mathcal{Q}:= \mathbf{P}\times_{Diff^+_0(\Sigma)}\Sigma\to \mathbf{T}_{g,n}.
$$
The following result is well-known ( cf \cite{RS} ):
\begin{lemma}\label{slice} Suppose that $n+ 2g \geq 3$. Then for any $\gamma_o=[(j_o,{\bf y}_o)]\in \mathbf{T}_{g,n}$, and any $(j_o,{\bf y}_o)\in \mathbf{P}$ with $\pi_{\mathbf{T}}(j_o,{\bf y}_o)=\gamma_o$ there is an open neighborhood $\mathbf{A}$ of zero in $\mathbb{C}^{3g-3+n}$ and a local holomorphic slice
$\iota=(\iota_{0},\cdots,\iota_{n}):\mathbf{A}\to \mathbf{P} $ such that
\begin{equation}\label{slice-1}
\iota_0(o)=j_o,\qquad \iota_i(o)=y_{io},\qquad i=1,\dots,n,
\end{equation}
and the map
$$
\mathbf{A}\times Diff_{0}(\Sigma) \to \mathbf{P}:(a ,\phi)\mapsto (\phi^*\iota_{0}(a),\phi^{-1}(\iota_1(a)),
\cdots,\phi^{-1}(\iota_{n}(a))
$$
is a diffeomorphism onto a neighborhood of the orbit of $(j_o,{\bf y}_o)$.
\end{lemma}
\v\n

From the local slice we have a local coordinate chart on $U$ and a local trivialization on $\pi_{\mathbf{T}}^{-1}(U)$:
\begin{equation}\label{local coordinates}
\psi: U\rightarrow \mathbf{A},\;\;\;\Psi:\pi_{\mathbf{T}}^{-1}(U)\rightarrow \mathbf{A}\times \Sigma,\end{equation}
 where $U\subset \mathbf{T}_{g,n}$ is a open set. We call $(\psi, \Psi)$ in \eqref{local coordinates} a local coordinate system for $\mathcal{Q}$. Suppose that we have two local coordinate systems
\begin{equation}\label{local coordinates-1}
(\psi, \Psi): (O, \pi_{\mathbf{T}}^{-1}(O))\rightarrow (\mathbf{A}, \mathbf{A}\times \Sigma),\end{equation}
\begin{equation}\label{local coordinates-2}
(\psi', \Psi'): (O', \pi_{\mathbf{T}}^{-1}(O'))\rightarrow (\mathbf{A}', \mathbf{A}'\times \Sigma).\end{equation}
Suppose that $O\bigcap O'\neq \emptyset.$ Let $W$ be a open set with $W\subset O\bigcap O'$.
Denote $V=\psi(W)$ and $V'=\psi'(W)$. Then ( see \cite{RS})
\begin{lemma}\label{local coordinates-3}
$\psi'\circ \psi^{-1}|_{V}:V\to V'$ and $\Psi'\circ \Psi^{-1}|_{V}: {V}\times \Sigma\to {V'}\times \Sigma$ are holomorphic.
\end{lemma}

\subsection{\bf $J$-holomorphic maps }\label{J-holo}

Let $(M,\omega,J)$ be a closed $C^{\infty}$ symplectic manifold of dimension $2m$ with $\omega$-tame almost complex structure $J$, where $\omega$ is a symplectic form. Then there is a Riemannian metric
\begin{equation}\label{definition_of_metrics}
G_J(v,w):=<v,w>_J:=\frac{1}{2}\left(\omega(v,Jw)+\omega(w,Jv)\right)
\end{equation}
for any $v, w\in TM$. Following \cite{MS} we choose the complex linear connection
 $$
 \widetilde {\nabla}_{X}Y=\nabla_{X}Y-\tfrac{1}{2}J\left(\nabla_{X}J\right)Y $$
 induced by the Levi-Civita connection $\nabla$ of the metric $G_{J}.$
\v
Let $(\Sigma, j, {\bf y})$ be a marked nodal Riemann surface of genus $g$ with $n$ marked points. Let $\sigma:\tilde{\Sigma}=\sum_{\nu=1}^{\iota}\Sigma_{\nu}\to \Sigma$ be the normalization.
Let $u:\Sigma\longrightarrow M$ be a smooth map. Here and later we say a map ( or section ) is smooth we mean that it is a continuous map such that, restricting to every $\Sigma_{\nu}$, it is smooth. The map $u$ is called a $(j,J)$-holomorphic map if, restricting to each $\Sigma_{\nu}$,  $du\circ j=J\circ du$. Alternatively
\begin{equation}\label{holo}
\bar\p_{j,J}(u):=\half\left(du + J(u) du\circ j\right)=0.
\end{equation}
Given $A\in H_2(M,\mathbb{Z})$. Let $u:\Sigma\to M$ be $(j_o,J)$-holomorphic map with $u([\Sigma])=A$.
Set $b_o=(\mathbf{s}_o, u)$, $\mathbf{s}_o=(j_o,\mathbf{y}_o)$. Let
$\mathbf{A}=\mathbf{A}_1\times \mathbf{A}_2\times...\times \mathbf{A}_\iota$ be a local coordinate system of complex structures on $\Sigma$ such that $\mathbf{s}_o\in \mathbf{A}$. Denote by $j_{\mathbf{s}}$ the complex structure corresponding to $\mathbf{s}=(j, \mathbf{y})\in \mathbf{A}$. Let $\alpha$ be a small
constant such that $0<\alpha <1$. For any
section $h\in C^{\infty}(\Sigma;u^{\ast}TM)$ and section $\eta \in
C^\infty(\Sigma, u^{*}TM\otimes \wedge^{0,1}_jT^{*}\Sigma)$ and given integer $k>4$ we define the norms
$\|h\|_{j_i,k,2,\alpha}$ and $\|\eta\|_{j_i,k-1,2,\alpha}$
( see \cite{LS-1} ).
Denote by $W^{k,2,\alpha}(\Sigma;u^{\ast}TM)$ and
$W^{k-1,2,\alpha}(\Sigma, u^{*}TM\otimes\wedge^{0,1}_jT^{*}\Sigma)$ the complete spaces with respect to the norms $\|h\|_{j_i,k,2,\alpha}$ and $\|\eta\|_{j_i,k-1,2,\alpha}$ respectively. We can also define  $\mathcal{W}^{k,2,\alpha}(\Sigma;u^{\ast}TM)$  as in \cite{LS-1}.
Let
 $$
\widetilde{\mathcal B}=\{u\in W^{k,2,\alpha}(\Sigma,M)\mid u_{*}([\Sigma])=A\}.
 $$
For fixed $\mathbf{s}_o$, restricting to each $\Sigma_{\nu}$, $\widetilde{\mathcal B}$ is an infinite dimensional Banach manifold.  Let $\delta>0$, $\rho>0$ be two small numbers. Denote
$$\widetilde{\mathbf{O}}_{b_{o}}(\delta,\rho):=\{(\mathbf{s},v)\in \mathbf{A}\times \widetilde{\mathcal{B}} \;|\; d_{\mathbf{A}}(\mathbf{s}_o,\mathbf{s})<\delta, \|h\|_{j_\mathbf{s},k,2,\alpha}<\rho \},
$$
where $v=\exp_{u}h$, $d_{\mathbf{A}}$ is the distance function induced by the Weil-Petersson metric on the Deligne-Mumford space $\overline{\mathcal{M}}_{g,n}$.

\v
\section{\bf A finite rank bundle and weighted norms}\label{finite rank orbi-bundle}

\subsection{\bf A finite rank bundle }\label{s_intro_2}

We slightly
deform $\omega$ to get a rational class $[\omega^*]$. By taking multiple, we can
assume that $[\omega^*]$ is an integral class on $M$. Therefore, it is the Chern class of a complex line bundle $L$ over $M$. Let $i$ be the complex structure on $L$. We choose a Hermition metric $G^L$ and the associate unitary connection $\nabla^{L}$ on $L$.
\v
Let $(\Sigma, j, {\bf y})$ be a marked nodal Riemann surface of genus $g$ with $n$ marked points.
Let $u:\Sigma\longrightarrow M$ be a $\mathcal{W}^{k,2,\alpha}$ map. We have a complex line bundle $u^*L$ over $\Sigma$ with complex structure $u^*i$ and unitary connection $u^*\nabla^{L}$. Put $b=(\mathbf{s},u)$, $\mathbf{s}=(j, {\bf y})$.
The unitary connection $u^{*}\nabla^{L}$
splits into $ u^{*}\nabla^{L}:=u^{*}\nabla^{L,(1,0)}\oplus u^{*}\nabla^{L,(0,1)}$.
We can define the spaces $\mathcal{W}^{k,2,\alpha}(\Sigma,u^{*}L)$ and
$W^{k-1,2,\alpha}(\Sigma,u^{*}L\otimes\wedge_{j}^{(0,1)}T^{\star}\Sigma)$
as in \cite{LS-1} (see also section \S\ref{s_intro_2}).

Denote
$$D^L:=u^{*}\nabla^{L,(0,1)}:\mathcal{W}^{k,2,\alpha}(\Sigma,u^{*}L)\to W^{k-1,2,\alpha}(\Sigma,u^{*}L\otimes\wedge_{j}^{(0,1)}T^{\star}\Sigma).$$
One can check that
$$
D^{L}(f\xi)=\bar{\p}_{j}f \otimes \xi +  f\cdot D^{L}\xi.
$$
$D^{L}$ determines a holomorphic structure on $u^*L$, for
which $D^{L}$ is an associated Cauchy-Riemann operator (see \cite{HLS,IS}).
Then $u^*L$ is a holomorphic line bundle.
\v
Let $\Sigma$ be a smooth Riemann surface. Let $\{V\}$ be a covering of $\Sigma$ such that each $V\subset \Sigma$ is a trivializing open set of $u^*L$. $D^{L}$ becomes in each $V$
\begin{equation}\label{nonvanishing solution}
D^{L}f=\frac{\p}{\p \bar{z}}f + a_{V}f.
\end{equation}
Consider the PDE
\begin{equation}\label{nonvanishing solution-1}
\frac{\p}{\p \bar{z}}f + a_{V}f=0.
\end{equation}
We can find a nonvanishing solution $e_{V}$ of \eqref{nonvanishing solution-1}. ( see \cite{HLS,IS}). Then $\{e_{V}\}$ define a holomorphic structure on $u^*L$ such that $D^{L}$ is $\bar{\p}_j$.
\v
Now let $\Sigma$ be nodal Riemann surface. For every smooth component $\Sigma_\nu$ we have a holomorphic structure on $u^*L$ over $\Sigma_\nu$. Suppose that $p$ is a node of $\Sigma_1$ and $\Sigma_2$. We choose nonvanishing solutions $e_{V_i}$ of \eqref{nonvanishing solution-1}, where $V_i\subset \Sigma_i$. Since \eqref{nonvanishing solution-1} is a linear equation, we can choose $e_{V_i}$ such that $e_{V_1}(p)=e_{V_2}(p)$. Then we have a holomorphic structure on $u^*L$ over $\Sigma$.
\v
Let $\lambda_{(\Sigma, j)}$ be
the dualizing sheaf of meromorphic 1-form with at worst simple pole at the nodal points and for each
nodal point p, say $\Sigma_1$ and $\Sigma_2$ intersects at p,
$$Res_p(\lambda_{(\Sigma_1, j_1)}) + Res_p(\lambda_{(\Sigma_2, j_2)})=0.$$
Let $\Pi:\overline{\mathscr{C}}_{g}\to \overline{\mathcal{M}}_{g}$ be the universal curve. Let $\lambda$ be the relative dualizing sheaf over $\overline{\mathscr{C}}_{g}$, the restriction of $\lambda$ to $(\Sigma, j)$  is $\lambda_{(\Sigma, j)}$.
\v
Set $\Lambda_{(\Sigma, j)}:=\lambda_{(\Sigma, j)}\left(\sum_{i=1}^{n} y_i\right)$.
$\Lambda\mid_{\mathscr{C}_{g,n}}$ is a line bundle over  $\mathscr{C}_{g,n}$. Let $(\psi, \Psi): (O, \pi_{\mathbf{T}}^{-1}(O))\rightarrow (\mathbf{A}, \mathbf{A}\times \Sigma)$ be a local coordinate systems, where $O\subset \mathbf{T}_{g,n}$ is an open set.
 $\Lambda$ induces a line bundle over  $ \mathbf{A}\times \Sigma,$ denoted by $\widetilde{\Lambda}$.
 Then
$\widetilde{\mathbf{L}}\mid_{b}:=\mathscr{P}^{*}\widetilde{\Lambda}\otimes u^*L $ is a holomorphic line bundle over $\Sigma$, where $\mathscr{P}$ denote the forgetful map.
  We have a Cauchy-Riemann operator $  \bar{\p}_b.$
Then $H^0(\Sigma, \widetilde{\mathbf{L}}\mid_{b})$ is the $ker \bar{\p}_b$. Here the $\bar{\p}$-operator depends on the complex structure $j$ on $\Sigma$ and the bundle $u^*L$, so we denote it by $\bar{\p}_b$.
\v
If $\Sigma_{\nu}$
is not a ghost component, there exist a constant $\hbar_{o}>0$ such that
$$\int_{u(\Sigma_{\nu})}\omega^*> \hbar_{o} .$$ Therefore, $c_1(u^*L)(\Sigma_{\nu})>0$. For ghost component $\Sigma_{\nu}$, $\lambda_{\Sigma_{\nu}}\left(\sum_{i=1}^{n} y_i\right)$ is positive. So
for any $b=(\mathbf{s},v)\in \widetilde{\mathbf{O}}_{b_{o}}(\delta,\rho)$
by taking the higher power of $\widetilde{\mathbf{L}}\mid_{b}$, if necessary, we can assume that
$\widetilde{\mathbf{L}}\mid_{b}$ is very ample. Hence, $H^1(\Sigma, \widetilde{\mathbf{L}}\mid_{b})= 0$. Therefore,
$H^0(\Sigma, \widetilde{\mathbf{L}}\mid_{b})$ is of constant rank ( independent of $b\in \widetilde{\mathbf{O}}_{b_{o}}(\delta,\rho)$). We have a finite rank bundle $\widetilde{\mathbf{F}}$ over $\widetilde{\mathbf{O}}_{b_{o}}(\delta,\rho)$,
whose fiber at $b=(j,{\bf y},v)\in\widetilde{\mathbf{O}}_{b_{o}}(\delta,\rho)$
is $H^0(\Sigma, \widetilde{\mathbf{L}}\mid_{b})$.
\v
\begin{remark}\label{finite bundle over moduli space} In \cite{S}, \cite{R3} and \cite{LR} the authors constructed a finite rank bundle $\widetilde{\mathbf{F}}$ over a neighborhood of $J$-Holomorphic map Moduli Spaces. In this paper we study the local analysis properties, so we only give the local construction here.
\end{remark}

\v
\subsection{\bf Weighted norms}\label{s_intro_2}
\v
Let $(V,z)$ be a local coordinate system on $\Sigma$ around a nodal point ( or a marked point) $q$ with  $z(q)=0$ . Let $b=(\mathbf{s},u)\in \widetilde{\mathbf{O}}_{b_o}(\delta_{o},\rho_{o})$ and $e$ be a local holomorphic section of $u^{*}L|_{V}$ with $\|e\|_{G^L}(q)\neq 0$ for $q\in V$.
Then for any $\phi\in   \widetilde{\mathbf F}|_{b}$ we can write
\begin{equation}\label{eqn_phi_loc_hol}
\phi|_{V} =f \left(\frac{dz}{z}\otimes e \right)^{p},\;\;\mbox{ where } f\in \mathcal{O}(V), \;p\in \mathbb{Z}.
\end{equation}
In terms of the holomorphic cylindrical coordinates $(s,t)$ defined by $z=e^{s+2\pi\sqrt{-1}t}$ we re-written \eqref{eqn_phi_loc_hol} as
$$
\phi(s,t)|_{V}=f(s,t) \left((ds+2\pi\sqrt{-1}dt)\otimes e \right)^{p},
$$
where $f(z)\in \mathcal O(V)$. It is easy to see that $|f(s,t)-f(-\infty, t)|$ uniformly exponentially converges to 0 with respect to $t\in S^1$ as
$|s|\to \infty$.
\v
The metrics $G^L$ and $\mathbf{g}^{\diamond}$ together induce a metric $\mathbb{G}$ on $\widetilde{\mathbf L}$. We define weighted norms for $C^{\infty}_{c}(\Sigma,\widetilde{\mathbf L}|_{b})$ and
$C^{\infty}_{c}(\Sigma, \widetilde{\mathbf L}|_{b}\otimes\wedge^{0,1}_{j}T^{*}\Sigma).$
Fix a positive function $W$ on $\Sigma$ which has order equal
to $e^{\alpha |s|} $ on each end of $\Sigma_i$, where $\alpha$ is a small
constant such that $0<\alpha  <1$. For any $\zeta\in C^{\infty}_{c}(\Sigma,\widetilde{\mathbf L}|_{b})$ and
any section $\eta \in
C^{\infty}_{c}(\Sigma, \widetilde{\mathbf L}|_{b}\otimes\wedge^{0,1}_{j}T^{*}\Sigma)$
we define the norms
\begin{equation}
\|\zeta\|_{j,k,2,\alpha}=
 \left(\int_{\Sigma}e^{2\alpha|s|} \sum_{i=0}^{k} |\nabla^i \zeta|^2
dvol_{\Sigma}\right)^{1/2},
\end{equation}
\begin{equation}
\|\eta\|_{j,k-1,2,\alpha}= \left(\int_{\Sigma}e^{2\alpha|s|}\sum_{i=0}^{k-1} |\nabla^i \eta|^2
dvol_{\Sigma}\right)^{1/2}.
\end{equation}
Here all norms and
covariant derivatives are taken with respect to the Hermition metric $\mathbb{G}$ on $\widetilde{\mathbf{L}}$ and
the metric $\mathbf{g}^\diamond$ on $(\Sigma, j, {\bf y})$, $dvol_{\Sigma}$ denotes the volume form with respect to $\mathbf{g}^\diamond$. Denote by $W^{k,2,\alpha}(\Sigma;\widetilde{\mathbf L}|_{b})$ and
$W^{k-1,2,\alpha}(\Sigma, \widetilde{\mathbf L}|_{b}\otimes \wedge^{0,1}_{j}T^{*}\Sigma)$ the complete spaces with respect to the norms respectively.
\v
We choose $R_0$ so large that  $u_{i}(\{|s_i|\geq \frac{r}{2}\})$ lie in $O_{u_{i}(q)}$ for any $r>R_0$. In this coordinate system we identify $T_xM$ with $T_{u_{i}(q)}M$ for all $x\in O_{u_{i}(q)}$. With respect to the base $\left(e\otimes\frac{dz}{z} \right)^{p}$ for $\widetilde{\mathbf L}|_{b}$ we have a local trivialization.
Any $\zeta_0\in \widetilde{\mathbf L}|_{b}(q)$
may be considered as a vector field in the coordinate neighborhood.
We fix a smooth cutoff function $\varrho$:
\[
\varrho(s)=\left\{
\begin{array}{ll}
1, &\mbox{ if }\ |s|\geq \bar{d} \\
0, &\mbox{ if }\ |s|\leq \frac{\bar{d}}{2}
\end{array}
\right.
\]
where $\bar{d}$ is a large positive number. Put
$$\hat{\zeta}_0=\varrho \zeta_0.$$
Then for $\bar{d}$ large enough $\hat{\zeta}_0$ is a section in $C^{\infty}(\Sigma; \widetilde{\mathbf L}|_{b_{o}})$
supported in the tube $\{(s,t)||s|\geq \frac{\bar{d}}{2}, t \in {S} ^1\}$.
Denote
$${\mathcal W}^{k,2,\alpha}(\Sigma;\widetilde{\mathbf L}|_{b})=\left\{\zeta+\hat{\zeta}_0 | \zeta \in
W^{k,2,\alpha}(\Sigma; \widetilde{\mathbf L}|_{b}),\zeta_0 \in \widetilde{\mathbf L}|_{b}(q)\right\}.$$
We define weighted Sobolev  norm  on ${\mathcal W}^{k,2,\alpha}$ by
$$\| \zeta+\hat{\zeta}_{0}\|_{\mathcal{W},j,k,2,\alpha}=
\|\zeta\|_{j,k,2,\alpha} + |\zeta_{0}| ,$$
 where $|\zeta_{0}|=[\mathbb{G}(\zeta_{0},\zeta_{0})_{u(q)}]^{\frac{1}{2}}$.
\v
Let $b=(\mathbf{s},u)$. We define a Cauchy-Riemann operator
$$D^{\widetilde{\mathbf L}}|_{b}: {\mathcal W}^{k,2,\alpha}(\Sigma,\widetilde{\mathbf L}|_{b})\to W^{k-1,2,\alpha}(\Sigma, \widetilde{\mathbf L}|_{b}\otimes\wedge^{0,1}_{j_\mathbf{s}}T^{*}\Sigma)\;\;\;by$$
\begin{equation}\label{eqn_D_L}
D^{\widetilde{\mathbf L}}|_{b}(f(\mathbf{k}\otimes e)^{p})=(\bar{\p}f)(\mathbf{k}\otimes e)^{p}+(p
 f)(\mathbf{k}\otimes D^{L}e)\otimes(\mathbf{k}\otimes e)^{p-1},
\end{equation}
where $\mathbf{k}$ is a local frame field of $\widetilde{\Lambda}$, $f(\mathbf{k}\otimes e)^{p}\in {\mathcal W}^{k,2,\alpha}(\Sigma, \widetilde{\mathbf L}|_{b}).$
\v
With respect to the holomorphic structure $\{e_{V}\}$ we have $D^{\widetilde{\mathbf L}}|_{b}=\bar{\p}_b$. The linearized operator of $D^{\widetilde{\mathbf L}}|_{b}$ is also $\bar{\p}_b$.
\v\v
\section{Smoothness of $\widetilde{\mathbf F}$ on top strata}\label{smoothness}
\v

Let $(\Sigma,j,\mathbf{y})$ be a smooth Riemann surface of genus $g$ with $n$ marked points.
Let $b_o=(a_o, u_o)=(j_o,\mathbf{y}_o,u_o)$, $b=(a,u)$, $u=\exp_{u_o}h$, $b\in \widetilde{\mathbf{O}}_{b_o}(\delta_{o},\rho_{o})$.
\v
\subsection{Smoothness of $\mathscr{F}(h,\xi)$}

First we recall a fact about the exponential map on a compact Riemannian manifold $M$ (see \cite{MS}, Page 362, Remark 10.5.5). There
are two smooth families of endomorphisms
$$
E_{i}(p,\xi) : T_{p}M\rightarrow T_{\exp _{p}\xi} M ,\;\;\; i=1,2,
$$
that are characterized by the following property.
Let $\gamma :\mathbb R\to M $ be any smooth path in $M$ and $v(t)\in T_{\gamma(t)}M$ be any smooth vector field along this path then the derivative of the path $t\to \exp_{\gamma(t)}(v(t))$ is given by the formula
$$
\frac{d}{dt} \exp_{\gamma}(v)=E_{1}(\gamma,v)\dot\gamma+E_{2}(\gamma, v)\tilde \nabla_{t}v,
$$
where $\dot \gamma =\frac{d\gamma}{dt}.$
We have
$$
E_{1}(p,0)=E_{2}(p,0)=Id:T_{p}M\to T_{p}M,\;\;\;\forall p\in M,
$$
and $E_{i}(p, \xi)$ are uniformly invertible for sufficiently small $\xi.$ Since $M$ is compact, there exists a constant $ \epsilon$ such that for any $p\in M$ and $\xi\in T_{p}M$ with $|\xi|_{T_{p}M}\leq \epsilon$, $E_{i}(p, \xi)$ are  uniformly invertible.
\v
Given $x\in M$ and $\zeta\in T_xM$ we define two linear maps
\begin{align*}E_x(\zeta): T_xM\to T_{\exp_x(\zeta)}M,\;\;\; \;\;\;\; \Psi_{x}(\zeta):T_{x}M\times T_{x}M \to L_{\exp_{x}(\zeta)}\\
E_x(\zeta)\zeta':=\frac{d}{dt}\exp_x(\zeta+t\zeta')\mid_{t=0},\;\;\;\;\;\;
\Psi_{x}(\zeta;\zeta',\eta)=\nabla^{L}_{t}(\Phi^{L}_{x,\exp_{x}(\zeta+t\zeta')} \eta)|_{t=0}.
\end{align*}
Choose a local coordinate system $x_1,...,x_{2m}$ on $M$, denote $\frac{\p}{\p x_i}=\p_{x_{i}}$. Let $\xi$ be a smooth section of the bundle $L$ over a neighborhood $U_{o}$ of $u_o(\Sigma)$. Let $u=\exp_{u_o}h$.
Let $\Phi^L_{u_o,u}$ be the parallel transport with  respect to the connection $\nabla^L$, along the
geodesics $s\rightarrow \exp_{u_o}(sh)$.
$\Phi^L_{u_o,u} $ induce two isomorphisms
$$\mathcal{W}^{k,2,\alpha}(\Sigma,u_o^*L)\to \mathcal{W}^{k,2,\alpha}(\Sigma,u^*L),\;\;\;\; \mathcal{W}^{k-1,2,\alpha}(\Sigma,u_o^*L\otimes\wedge_j^{0,1}T\Sigma)\to  \mathcal{W}^{k-1,2,\alpha}(\Sigma,u^*L\otimes\wedge_j^{0,1}T\Sigma),
$$
still denote them by $\Phi^L_{u_o,u}.$ Denote $u_{t}=\exp_{u_{o}}(h+th').$ We calculate  $(u_{t}^{*}\nabla^{L})_t\left.\left(\Phi^L_{u_o,u_{t}}u_{o}^{*}\xi\right)\right|_{t=0}$.
By definition, for any $p\in \Sigma$,
$$(u_{t}^{*}\nabla^{L})_t\left.\left(\Phi^L_{u_o,u_{t}}u_{o}^{*}\xi\right)(p)\right|_{t=0}
=\nabla^{L}_t\left.\left(\Phi^L_{u_o,u_{t}}\xi\right)\circ u_{t}\right|_{t=0}(p)=\Psi_{u_{o}}(h;h',\xi)(u(p)).$$
Since $L$ and $\nabla^{L}$ are smooth on $M$, $\Psi_{u_{o}}(h;h',\xi)=\nabla^{L}_{E_{u_{o}}(h)h'}(\Phi_{u_{o},u}\xi)$ and $\Psi_{u_{o}}(0;h',\xi)=0,$ there is a constant $C>0$ independent of $p$ such that
\begin{equation}\label{eqn_est_Psi}
|\Psi_{u_{o}}(h;h',\xi)|_{u(p)}\leq C|h(p)||h'(p)|(|\xi|+|\nabla \xi|)|_{u(p)},
\end{equation}
when $\|h\|_{k,2}\leq \epsilon.$ If no danger of confusion we denote $(u_{t}^{*}\nabla^{L})_{t}$ by $\nabla_{t}^{L}.$ Then we have
$$\|\nabla^{L}_t\left.\left(\Phi^L_{u_o,u_{t}}u_{o}^{*}\xi\right)\right|_{t=0}\|_{C^{0}}\leq C\|h\|_{C^{0}}\|h'\|_{C^{0}}\|\xi\|_{C^{1}(U_{o})}$$
for some constant $C>0$.  By the Sobolev embedding Theorem we have
\begin{equation}\label{eqn_est_Psi_k}
\left\|\nabla^{L}_t\left.\left(\Phi^L_{u_o,u_{t}}u_{o}^{*}\xi\right)\right|_{t=0}\right\|_{k,2}\leq C\|h\|_{k,2}\|h'\|_{k,2}
\end{equation}
where $C>0$ is a constant depending on $\|\xi\|_{C^{k+1}(U_{o})}$, the Sobolev constant and the metric of $M.$
Then the operator
$$\nabla^{L}_{E_{u}(h)\cdot}(\Phi^L_{u_o,u})\xi: W^{k,2}(\Sigma,u_o^*TM)\to \mathcal{W}^{k-1,2,\alpha}(\Sigma,u^*L)$$
is a bounded linear operator. For any $l\in Z^{+},$ denote $\ft=(t_{1},\cdots,t_{l}),$ $u_{\ft}=\exp_{u_{0}}(h+\sum_{i=1}^{l}t_{l}h_{l} )$
 and
$$
T^{l}(h;h_{1},\cdots,h_{l})\xi=\nabla^{L}_{t_{1}}\cdots\nabla^{L}_{t_{l}}\left.\left(\Phi^L_{u_o,u_{\ft}}u_{o}^{*}\xi\right)\right|_{\ft=\mathbf 0}.
$$
A direct calculation gives us
\begin{equation}\label{eqn_T_claim}
|T^{l}(h;h_{1},\cdots,h_{l})\xi|\leq C\Pi_{i=1}^{l}|h_{i}(p)|.
\end{equation}
 By the same way as above we can show that
$$T^{l}(h;\cdot\cdot\cdot) \cdot : W^{k,2}(\Sigma,u_o^*TM)\times \cdots \times  W^{k,2}(\Sigma,u_o^*TM)\times  \mathcal{W}^{k,2,\alpha}(\Sigma,u^*L)\to  \mathcal{W}^{k,2,\alpha}(\Sigma,u^*L)$$ is a bounded linear operator with respect to $h_{1},\cdots,h_{l}$.

\v
Now we calculate
$\nabla^{L}_t(D^L_{u_{t}}\Phi_{u_{o},u}u_{o}^{*}\xi)(p)|_{t=0}$. Let $\p_{ v}$ be a section of $T\Sigma$, denote $(u^*\nabla^L)_ {\p_{v}}=\nabla^{L}_{\p_{v}}.$  
Then, by the definition of curvature and $[\p_{t},\p_{ v}]=0,$ we have
$$\nabla^L_t\nabla^L_{\p_{ v}} \Phi_{u_{o},u_{t}}u_{o}^*\xi
\mid_{t=0}(p)=\nabla^L_{\p_{ v}}\nabla^L_t \Phi_{u_{o},u_{t}}u_{o}^*\xi
\mid_{t=0}(p) +R(\p_{t},\p_{ v}) \Phi_{u_{o},u_{t}}u_{o}^*\xi
\mid_{t=0}(p).
$$
By \eqref{eqn_est_Psi_k} we get
$$
\|\nabla^L_{\p_{ v}}\nabla^L_t \Phi_{u_{o},u_{t}}u_{o}^*\xi
\mid_{t=0}\|_{k-1,2}\leq C\|\nabla^L_t \Phi_{u_{o},u_{t}}u_{o}^*\xi
\mid_{t=0}\|_{k,2}\leq C \|h'\|_{k,2}.
$$
Since curvature $R$ is a tensor,   we have
$$\|\nabla^L_t\nabla^L_{\p_{ v}} \Phi_{u_{o},u_{t}}u_{o}^*\xi
 \mid_{t=0}(p)\|_{k-1,2}\leq C\|h'\|_{k,2}
 .$$
Let $u_{\ft}$ be as above. One can check that
$$\|\nabla^{L}_{t_{1}}\cdots\nabla^{L}_{t_{l}}D^L_{u(\ft)}\left.\left(\Phi^L_{u_o,u_{\ft}}u_{o}^{*}\xi\right)\right|_{\ft=\mathbf 0}
\|_{k-1,2}\leq C\Pi_{i=1}^{l}\|h_{l}\|_{k,2}.
$$
Define $$\widetilde T^{l}(h;\cdot\cdot\cdot)\nabla^L_{\p_{v}}\cdot : W^{k,2}(\Sigma,u_o^*TM)\times \cdots \times W^{k,2}(\Sigma,u_o^*TM)\times  \mathcal{W}^{k,2,\alpha}(\Sigma,u^*L)\to  \mathcal{W}^{k-1,2,\alpha}(\Sigma,u^*L \otimes\wedge^{0,1}_{j_o}T^{*}\Sigma)$$ by
$$
\widetilde{T}^{l}(h;h_{1},\cdots,h_{l})\nabla^L_{\p_{ v}}\xi(p)=\nabla^{L}_{t_{1}}\cdots\nabla^{L}_{t_{l}}\left.
\left(\nabla^L_{\p_{ v}}\Phi^L_{u_o,u_{\ft}}u_{o}^{*}\xi\right)(p)\right|_{\ft=\mathbf 0}.
$$
We can show that
$\widetilde{T}^{l}(h;\cdot\cdot\cdot)\nabla^L_{\p_{ v}}\cdot$ is a bounded linear operator with respect to $h_{1},\cdots,h_{l}$.
Define
$$
\mathscr {F}:  W^{k,2}(\Sigma,u_{o}^{\star}TM)
\times {\mathcal W}^{k,2,\alpha}(\Sigma,u_o^*L)\to W^{k-1,2,\alpha}(\Sigma,u_o^*L\otimes \wedge_{j_o}^{0,1}T^*\Sigma)
\;\;\;by$$
$$\mathscr{F}(h,\zeta)= (\Phi^L_{u_o,u})^{-1}D^L_{u}\Phi^L_{u_o,u}\zeta.$$
\begin{lemma}\label{smoothness-1 }
$\mathscr{F}(h,\zeta)$ is a smooth map.
\end{lemma}
\v\n
{\bf Proof.} Note that $L$ has finite rank and for any fixed $h$, $\mathscr{F}(h,\zeta)$ is a linear map. The key point is to prove the smoothness of $\mathscr{F}(h,\zeta)$ with respect to $h$. Since both
$T^{l}(h;\cdot\cdot\cdot)$ and $\widetilde{T}^{\ell}(h;\cdot\cdot\cdot)$ are bounded linear operators for any $l, \ell \in \mathbb Z^+$ , the smoothness of $\mathscr{F}(h,\zeta)$ follows.  $\Box$
\v\v
$\Phi^L_{u_o,u}$
induces two isomorphisms
$$
\mathbf{\Phi}^{L}_{j_a,u_o,u}:{\mathcal W}^{k,2,\alpha}(\Sigma,\widetilde{\mathbf L}|_{j_a,u_o}) \to {\mathcal W}^{k,2,\alpha}(\Sigma,\widetilde{\mathbf L}|_{j_a,u})\;\;and $$
$$\mathbf{\Phi}^{L}_{j_a,u_o,u}:
\mathcal{W}^{k-1,2,\alpha}(\Sigma, \widetilde{\mathbf L}|_{u_o}\otimes\wedge^{0,1}_{j_a}T^{*}\Sigma)\to \mathcal{W}^{k-1,2,\alpha}(\Sigma,\widetilde{\mathbf L}|_{u}\otimes\wedge^{0,1}_{j_a}T^{*}\Sigma)\;\;\;by
$$
$$\mathbf{\Phi}^L_{j_a,u_o,u}( \mathbf{k}_a\otimes\mathbf{e}_{u_o})^{p}=\left(\Phi^L_{u_o,u}(\mathbf{k}_a\otimes\mathbf{e}_{u_o})\right)^{p}.$$
Choose a connection $\nabla^{\Lambda}$ for the bundle $\Lambda$. Denote by $\tilde{\Lambda}$ and $\nabla^{\widetilde{\Lambda}}$ the expressions of $\Lambda$ and $\nabla^{\Lambda}$ in the local coordinate $\mathbf{s}\in \mathbf{A}$. $\widetilde{\Lambda}$ is a line bundle over $\mathbf{A}\times \Sigma$.  Let $\Phi^{\widetilde{\Lambda}}_{a_o,a}$ be the
parallel transport with  respect to the connection $\nabla^{\widetilde{\Lambda}}$, along the line $a_o+t(a-a_o)$.
We have two maps
$$\Psi^{\widetilde{\Lambda}}_{j_{o},j_a}: \widetilde{\mathbf L}|_{j_o,u}   \to \widetilde{\mathbf L}|_{j_a,u}  ,\;\;\;\;\Psi^{\widetilde{\Lambda}}_{j_a,j_{o}}:\widetilde{\mathbf L}|_{j_a,u}  \to \widetilde{\mathbf L}|_{j_o,u}. $$
For any $j_a\in \mathcal{J}(\Sigma)$ near $j_o$
we can write $j_a=(I + H)j_o(I + H)^{-1}$ where $H\in T_{j_o}\mathcal{J}(\Sigma)$.
We define two maps
$$\Psi^{\widetilde{\Lambda}}_{j_{o},j_a}: \widetilde{\mathbf L}|_{j_o,u}\otimes \wedge^{0,1}_{j_{o}}T^{*}\Sigma  \to \widetilde{\mathbf L}|_{j_a,u}\otimes \wedge^{0,1}_{j_a}T^{*}\Sigma\;\;\;and $$  $$\Psi^{\widetilde{\Lambda}}_{j_a,j_{o}}:\widetilde{\mathbf L}|_{j_a,u}\otimes \wedge^{0,1}_{j_a}T^{*}\Sigma  \to \widetilde{\mathbf L}|_{j_o,u}\otimes \wedge^{0,1}_{j_{o}}T^{*}\Sigma$$ by
\begin{align*}
 \Psi^{\widetilde{\Lambda}}_{j_{o},j_a}(\eta)= \frac{1}{2}(\Phi^{\widetilde{\Lambda}}_{a_o,a}\eta-\Phi^{\widetilde{\Lambda}}_{a_o,a}\eta \cdot j_{o} j_a),\;\;\;\; \Psi^{\widetilde{\Lambda}}_{j_a,j_{o}}(\varpi)= \frac{1}{2}(\Phi^{\widetilde{\Lambda}}_{a,a_o}\varpi-
 \Phi^{\widetilde{\Lambda}}_{a,a_o}\varpi \cdot j_aj_{o}).
\end{align*}
Note that $$u^*i(\mathbf{k}\otimes\mathbf{e}_{u})=\mathbf{k}\otimes u^*i(\mathbf{e}_{u}),$$ $$\Phi^{\widetilde{\Lambda}}_{a_o,a}(\mathbf{k}\otimes\mathbf{e}_{u})=
\Phi^{\widetilde{\Lambda}}_{a_o,a}(\mathbf{k})\otimes\mathbf{e}_{u}.$$
We have $u^*i\circ\Phi^{\widetilde{\Lambda}}_{a_o,a}=\Phi^{\widetilde{\Lambda}}_{a_o,a}u^*i$.
Since $u^*i\eta=-\eta j_{o}$ and $u^*i\varpi=-\varpi j_a$ for any
$$\eta\in \widetilde{\mathbf L}|_{j_o,u}\otimes \wedge^{0,1}_{j_{o}}T^{*}\Sigma,\;\;\;\varpi\in \widetilde{\mathbf L}|_{j_a,u}\otimes \wedge^{0,1}_{j_{a}}T^{*}\Sigma.$$
One can check that $u^{*}i\Psi^{\widetilde{\Lambda}}_{j_{o},j_a}(\eta) =-\Psi^{\widetilde{\Lambda}}_{j_{o},j_a}(\eta) j_a$ and $u^{*}i \Psi^{\widetilde{\Lambda}}_{j_a,j_{o}}(\varpi) =-\Psi^{\widetilde{\Lambda}}_{j_a,j_{o}}(\varpi) j_{o}.$ Then $\Psi^{\widetilde{\Lambda}}_{j_{o},j_a}$ and $\Psi^{\widetilde{\Lambda}}_{j_a,j_{o}}$ are well defined.
The proof of the following lemma is similar to the proof of Lemma 7.3 in \cite{LS-1}, we omit it here.
\begin{lemma}\label{isomorphism} Both
$\Psi^{\widetilde{\Lambda}}_{j_a,j_o}$ and $\Psi^{\widetilde{\Lambda}}_{j_o,j_a}$ are isomorphisms when $|H|$ small enough.
\end{lemma}
\v\n
Set
\begin{equation}
P^{\widetilde{\mathbf{L}}}_{b,b_o}=\Psi^{\widetilde{\Lambda}}_{j_a,j_o}\circ {\Phi}^{\widetilde{\mathbf L}}_{j_a,u,u_{o}}.
\end{equation}
We consider the map
$$
\mathcal {F}: \mathbf{A}\times W^{k,2}(\Sigma,u_{o}^{\star}TM)
\times {\mathcal W}^{k,2,\alpha}(\Sigma,\widetilde{\mathbf L}|_{b_o})\to W^{k-1,2,\alpha}(\Sigma,\wedge^{0,1}T\Sigma\otimes \widetilde{\mathbf L}|_{b_o})\;\;\mbox{defined\;\;by }
$$ $$
\mathcal {F}(a,h,\xi)= P^{\widetilde{\mathbf{L}}}_{b,b_{o}}
\circ D^{\widetilde{\mathbf L}}_{b}\circ (P^{\widetilde{\mathbf{L}}}_{b,b_o})^{-1}\xi.
$$

\begin{lemma}\label{lem_smooth_F}
 The following hold.
\begin{description}
\item[(1).] $\frac{d}{d\lambda } \mathcal{F}(a,0, \lambda\xi) |_{\lambda=0}=D_{\xi}\mathcal {F}|_{u_{o}}(\xi)=\Psi^{\widetilde{\Lambda}}_{j_a,j_o}\circ D^{\widetilde{\mathbf L}|_{j_a,u_{o}}}\circ (\Psi^{\widetilde{\Lambda}}_{j_a,j_o})^{-1}(\xi).$
\item[(2).] $\mathcal {F}$ is smooth functional of $(a,h,\xi)$.
\end{description}
\end{lemma}
\v\n
{\bf Proof.} {\bf (1)} is obtained by a direct calculation. Sine $\widetilde{\Lambda}$ is a smooth finite rank bundle over $\mathbf{A}$, by Lemma \ref{smoothness-1 } we obtain  {\bf (2)}.  $\Box$
\begin{lemma}\label{smoothness of bundle-1} In the local coordinate system $\mathbf{A}$ the bundle $\widetilde{\mathbf{F}}$ is smooth. Furthermore, for any base $\{e_{\alpha}\}$ of the fiber at $b_o$ we can get a smooth frame fields $\{e_{\alpha}(a,h)\}$ for the bundle $\widetilde{\mathbf{F}}$ over $\widetilde{\mathbf{O}}_{b_o}(\delta_{o},\rho_{o})$ .
\end{lemma}
\n
{\bf Proof.} Note that $D_{\xi}\mathcal {F}|_{b_{o}}=D^{\widetilde{\mathbf L}}|_{b_{o}}$. It is a Fredholm operator with $\mbox{coker} D^{\widetilde{\mathbf L}}|_{b_{o}}=0$ ( because of $H^1(\Sigma, \widetilde{\mathbf{L}}\mid_{b})= 0$ ). There is a right inverse $Q^{\widetilde{\mathbf{L}}}_{b_{o}}$ of $D^{\widetilde{\mathbf L}}|_{b_o}$.
Now we view $a$ and $h$ as parameters. It is easy to check that the conditions of the implicit function theorem (Theorem \ref{details_implicit_function_theorem}, Theorem \ref{smooth_implicit_function_theorem}) hold. Then there
there exist $\delta_o>0$, $\rho_o>0$ and a small neighborhood $O$ of $0 \in \ker\;D^{\widetilde{\mathbf L}}|_{b_o}$ and a unique smooth map
$$f^{\widetilde{\mathbf{L}}}: \widetilde{\mathbf{O}}_{b_o}(\delta_{o},\rho_{o})\times O\rightarrow W^{k-1,2,\alpha}(\Sigma,\wedge^{0,1}T\Sigma\otimes \widetilde{\mathbf L}|_{b_{o}})$$ such that for any $\zeta\in O$ and any $b\in \widetilde{\mathbf{O}}_{b_o}(\delta_{o},\rho_{o})$
$$ D^{\widetilde{\mathbf L}}|_{b}\circ(P^{\widetilde{\mathbf{L}}}_{b,b_{o}})^{-1}\left(\zeta + Q^{\widetilde{\mathbf{L}}}_{b_{o}}\circ f^{\widetilde{\mathbf{L}}}_{a,h}(\zeta)\right)=0.$$
We get the smoothness of $\widetilde{\mathbf{F}}$ in $\widetilde{\mathbf{O}}_{b_o}(\delta_{o},\rho_{o})$. Furthermore, choosing a base $\{e_{\alpha}\}$ of the fiber at $b_o$ we get a smooth frame fields $\{e_{\alpha}(a,h)\}$ by Theorem \ref{smooth_implicit_function_theorem}.
We complete the proof. \;\;$\Box$
\v
Let $(\Sigma,j,{\bf y})$ be a smooth Riemann surface of genus $g$ with $n$ marked points, $u:\Sigma\to M$ be a $C^{\infty}$ map. Denote $b=(j,{\bf y},u)$. For any $\varphi\in Diff^+(\Sigma)$ denote
	$$ b'=(j',{\bf y}',u') =\varphi\cdot (j,{\bf y},u)=(\varphi^*j, \varphi^{-1}{\bf y}, \varphi^{*}u).$$
Then
\begin{equation}\label{eqn_iso_Lo}
(u')^{*}i=\varphi^{*}(u^{*}i), \;\;(u')^{*}\nabla^{L}=\varphi^{*}(u^{*}\nabla^L).
\end{equation}
Let $\varphi\in Diff^+(\Sigma)$. For any section $\xi\in L$ we have
\begin{equation}
\label{eqn_iso_L}
(\varphi\cdot (u^*\xi))^p=(\varphi^*\circ u^*\xi)^p=((u\circ \varphi)^*\xi)^p=
((u')^*\xi)^p
\end{equation}
 and for any  $f(z) (dz)^{p}$
\begin{equation}\label{eqn_iso_grp}
\varphi\cdot f(z) (dz)^p=f(\varphi^{-1}(w))[d(\varphi^{-1}(w))]^p ,
\end{equation}
 where
$w=\varphi(z).$
We have the following lemma
\begin{lemma}\label{equi_diff} 	
   $(\varphi\cdot\widetilde{\mathbf{L}})|_{b'}=\varphi^{*}(\widetilde{\mathbf{L}}|_{b})$, $ D^{\widetilde{\mathbf L}}|_{b'}(\varphi^{*}\xi)=\varphi^{*}(D^{\widetilde{\mathbf L}}|_{b}(\xi))$ for any $\xi\in \widetilde{\mathbf L}|_{b}$.
 \end{lemma}
\v\n
{\bf Proof.} The first inequality follows from \eqref{eqn_iso_L} and \eqref{eqn_iso_grp}.
For any $f(\mathbf{k}\otimes \mathbf{e}_{u})^{p}\in  \widetilde {\mathbf L}|_{b}$, we have
$$
D^{\widetilde {\mathbf L}}(\mathbf{k}\otimes \mathbf{e}_{u})^{p}=\bar{\p}(f)\cdot(\mathbf{k}\otimes \mathbf{e}_{u})^{p}+pf\cdot\mathbf{k} \otimes D^{L}(\mathbf{e}_{u})\otimes(\mathbf{k}\otimes \mathbf{e}_{u})^{p-1}.
$$
By $\varphi\cdot f(z,\bar z)= f(\varphi^{-1}(w),\overline{\varphi^{-1}(w)}),$ we get
$$\bar{\p}(\varphi (f(z,\bar z)))= \frac{\p f}{\p \bar z}(\varphi^{-1}(w),\overline{\varphi^{-1}(w)}) \overline{\frac{\p\varphi^{-1}(w)}{\p w}}d\bar w = \frac{\p f}{\p \bar z}(\varphi^{-1}(w),\overline{\varphi^{-1}(w)}) d\overline{\varphi^{-1}(w)}   $$
Similar \eqref{eqn_iso_grp} we have
$$\varphi\cdot( \bar{\p}(f(z,\bar z)))=\varphi \cdot\left(\frac{\p f}{\p \bar z}(z,\bar z) d\bar z\right)= \frac{\p f}{\p \bar z}(\varphi^{-1}(w),\overline{\varphi^{-1}(w)}) d\overline{\varphi^{-1}(w)} $$
It follows that $\varphi\cdot(\bar{\p }f)=\bar{\p}(\varphi\cdot f).$
Since $D^{L}|_{b}=(u^{*}\nabla^{L})^{0,1}$ and $\varphi$ is holomorphic, by \eqref{eqn_iso_Lo} we have
$$D^{L}|_{b'}=((u')^{*}\nabla^{L})^{0,1}=(\varphi^{*}(u^{*}\nabla^{L}))^{0,1}=\varphi^{*}(u^{*}\nabla^{L})^{0,1}=\varphi^*D^{L}|_{b}.$$
Then the second inequality follows from the first inequality.  $\Box$

\v
\begin{remark}\label{isotropy group}
Let $G_{b_o}$ be the isotropy group at $b_o$. By Lemma \ref{equi_diff}, $D^{\widetilde{\mathbf L}}$ is $G_{b_o}$-equivariant and $G_{b_o}$ acts on $\mbox{ker} D^{\widetilde{\mathbf L}}|_{b_{o}}$. We may choose a $G_{b_o}$-equivariant right inverse $Q^{\widetilde{\mathbf{L}}}_{b_{o}}$. In fact, let $\hat Q^{\widetilde{\mathbf{L}}}_{b_o}$ be a right inverse of $D^{\widetilde{\mathbf L}}|_{b_{o}}$,  we define
$$Q^{\widetilde{\mathbf{L}}}_{b_o}(\eta)=\frac{1}{|G_{b_{o}}|}   \sum_{\varphi\in G_{b_{o}}}  \varphi^{-1}\cdot  \hat Q^{\widetilde{\mathbf{L}}}_{b_o}( \varphi \cdot \eta).$$
Then, for any $\varphi'\in G_{b_{o}} ,$ we have
 $$Q^{\widetilde{\mathbf{L}}}_{b_o}(\varphi'\cdot\eta)=\frac{1}{|G_{b_{o}}|}   \sum_{\varphi\in G_{b_{o}}}  \varphi^{-1}\cdot  \hat Q^{\widetilde{\mathbf{L}}}_{b_o}( \varphi \cdot \varphi'\cdot \eta)=$$$$
 \frac{1}{|G_{b_{o}}|}   \sum_{\varphi\in G_{b_{o}}} \varphi'\cdot(\varphi')^{-1}  \varphi^{-1}\cdot  \hat Q^{\widetilde{\mathbf{L}}}_{b_o}( \varphi \cdot \varphi'\cdot \eta)=\varphi'\cdot Q^{\widetilde{\mathbf{L}}}_{b_o}(\eta).$$
By uniqueness, it follows that $f^{\widetilde{\mathbf{L}}}$ is $G_{b_o}$-equivariant. So we have a $G_{b_o}$-equivariant version of Lemma \ref{smoothness of bundle-1}. In particular,
for any base $\{e_{\alpha}\}$ of the fiber at $b_o$ we can get a smooth $G_{b_o}$-equivariant frame fields $\{e_{\alpha}(a,h)\}$ for the bundle $\widetilde{\mathbf{F}}$ over $\widetilde{\mathbf{O}}_{b_o}(\delta_{o},\rho_{o})$.
\end{remark}

\begin{remark}\label{smoothness in coord.}
Note that what Lemma \ref{smoothness of bundle-1} claim is the smoothness in a local coordinate system
$(\psi, \Psi): (O, \pi_{\mathbf{T}}^{-1}(O))\rightarrow (\mathbf{A}, \mathbf{A}\times \Sigma)$. If we choose another local coordinate system $(\psi', \Psi'): (O', \pi_{\mathbf{T}}^{-1}(O'))\rightarrow (\mathbf{A}', \mathbf{A}'\times \Sigma)$ we have
$u'=u\circ d\vartheta_{a}^{-1}$
where $\vartheta_{a}=\Psi'\circ \Psi^{-1}|_{a}\in Diff^+(\Sigma)$ is a family of diffeomorphisms. If $u$ is only $W^{k,2}$ map, the coordinate transformation is not smooth. Nevertheless the Lemma \ref{smoothness of bundle-1} is still very useful in the study of the smoothness of top strata of virtual neighborhood, as we have a PDE here, we can use the standard elliptic estimates to get the smoothness of $u$ ( see \cite{LS-1}).
\end{remark}

\v

\v
\section{\bf Gluing}\label{pregluing}

\subsection{\bf Pregluing for maps}

Let $(\Sigma, j,{\bf y}, q)$ be a marked nodal Riemann surface of genus $g$ with $n$ marked points ${\bf y}=(y_1,...,y_n)$ and one nodal point $q$.  We write the marked nodal Riemann surface as
$$
\left(\Sigma=\Sigma_{1}\wedge\Sigma_{2},j=(j_{1},j_{2}),{\bf y}=({\bf y}_{1},{\bf y}_{2}), q=(q_1,q_2)\right),$$
where $(\Sigma_{i},j_{i},{\bf y}_{i}, q_i)$ are smooth Riemann surfaces, $(j_i, \mathbf{y}_i)\in \mathbf{A}_i$, $i=1,2$. We say that $q_1,q_2$ are paired to form $q$. Assume that $(\Sigma_{i},j_{i},{\bf y}_{i},q_{i})$ is stable, i.e., $n_{i}+2g_{i}+1\geq 3$, $i=1,2$. We choose metric $\mathbf{g}_i$ on each $\Sigma_i$ as in \S\ref{metric on surfaces}.
Let $z_i$ be the cusp coordinates around $q_i$, $z_i(q_i)=0$, $i=1,2$.
 Let
$$
z_1=e^{-s_1 - 2\pi \sqrt{-1} t_1},\;\;\;z_2=e^{s_2 +2\pi\sqrt{-1}  t_2}.
$$
$(s_i, t_i)$ are called the cusp holomorphic cylindrical coordinates near $q_{i}$.  In terms of the cusp holomorphic cylindrical coordinates we write
$$\stackrel{\circ}{\Sigma}_{1}:=\Sigma_1\setminus\{q_1\}\cong\Sigma_{10}\cup\{[0,\infty)\times S^1\},\;\;\;\stackrel{\circ}{\Sigma}_{2}:=\Sigma_2\setminus\{q_2\}\cong\Sigma_{20}\cup\{(-\infty,0]\times S^1\}.$$
Here $\Sigma_{i0}\subset \Sigma_i$, $i=1,2$, are compact surfaces with boundary. Put $\stackrel{\circ}{\Sigma}= \Sigma\setminus\{q_1,q_2\} =\stackrel{\circ}{\Sigma}_{1}\cup\stackrel{\circ}{\Sigma}_{2}$. We introduce the notations
$$
\Sigma_i(R_0)=\Sigma_{i0}\cup \{(s_i,t_i)|\;|s_i| \leq R_0\},\;\; \;\;\;\;\;\Sigma(R_0)=\Sigma_1(R_0)\cup \Sigma_2(R_0).$$
For any gluing parameter $(r,\tau)$ with $r\geq R_{0}$ and $\tau\in S^1$ we construct a surface $\Sigma_{(r)}$
with the gluing formulas:
\begin{equation}\label{gluing surface}
s_1=s_2 + 2r,\;\;\;t_1=t_2 + \tau.
\end{equation}
where we use $(r)$ to denote gluing parameters.
\v
Let $b_o=(a_o, u)$, $a_o=(\Sigma, j,{\bf y}, q)$, $u=(u_1,u_2)$, where $u_{i}:\Sigma_i\rightarrow M$ are are $(j_i,J)$-holomorphic maps with $u_{1}(q)=u_{2}(q).$ We will use the cusp holomorphic cylinder coordinates to describe the construction of $u_{(r)}:\Sigma_{(r)}\to M$.
We choose local normal coordinates $(x^{1},\cdots,x^{2m})$ in a neighborhood  $O_{u(q)}$ of $u(q)$ and choose $R_0$ so large that $u(\{|s_i|\geq \frac{r}{2}\})$ lie in $O_{u(q)}$ for any $r>R_0$. We glue the map $(u_1,u_2)$ to get a pregluing maps $u_{(r)}$ as follows. Set\\
\[
u_{(r)}=\left\{
\begin{array}{ll}
u_1 \;\;\;\;\; on \;\;\Sigma_{10}\bigcup\{(s_1,t_1)|0\leq s_1 \leq
\frac{r}{2}, t_1 \in S^1 \}    \\  \\
u_1(q)=u_2(q) \;\;
on \;\{(s_1,t_1)| \frac{3r}{4}\leq s_1 \leq
\frac{5r}{4}, t_1 \in S^1 \}  \\   \\
u_2 \;\;\;\;\; on \;\;\Sigma_{20}\bigcup\{(s_2,t_2)|0\geq s_2
\geq - \frac{r}{2}, t_2 \in S^1 \}     \\
\end{array}.
\right.
\]
To define the map $u_{(r)}$ in the remaining part we fix a smooth cutoff
function $\beta : {\mathbb{R}}\rightarrow [0,1]$ such that
\begin{equation}\label{def_beta}
\beta (s)=\left\{
\begin{array}{ll}
1 & if\;\; s \geq 1 \\
0 & if\;\; s \leq 0
\end{array}
\right.
\end{equation}
and $\sqrt{1-\beta^2}$ is a smooth function,  $0\leq \beta^{\prime}(s)\leq 4$ and $\beta^2(\frac{1}{2})=\frac{1}{2}.$
 We define\\
$$u_{(r)}= u_1(q)+ \left(\beta\left(3-\frac{4s_1}{r}\right)(u_1(s_1,t_1)-u_1(q)) +\beta\left(\frac{4s_1}{r}-5\right)(u_2(s_1-2r,t_1-\tau)- u_2(q))\right).$$

\vskip 0.1in
\noindent

\subsection{\bf Pregluing for $\widetilde{\mathbf{F}}$}\label{s_intro_3}

Let $b_o=(a_o,u)$, $u=(u_1,u_2)$, $u_i:\Sigma_i\to M$ are $(j_i,J)$-holomorphic maps, $i=1,2$. We choose $\{e_V\}$ as in \S\ref{s_intro_2} such that $\widetilde{\mathbf L}|_{b_{o}}$ is a holomorphic lie bundle. Then $D^{\widetilde{\mathbf L}}|_{b_{o}}=\bar{\p}_{j,u}$. Recall that with respect to the base $\left(\frac{dz}{z} \otimes e\right)^{p}$ for $\widetilde{\mathbf L}|_{b}$ we have a local trivialization.

Denote
$$\beta_{1;R}(s_1)=\beta\left(\frac{1}{2}+\frac{r-s_1}{R}\right),\;\;\;
\beta_{2;R}(s_{2})=\sqrt{1-\beta^2\left(\frac{1}{2}-\frac{s_{2}+r}{R}\right)}, $$
where $\beta$ is the cut-off function defined in \eqref{def_beta}.
Then we have
\begin{equation}\label{beta_rel.}\beta_{2;R}^2(s_1- 2r)=1-\beta^2\left( \frac{1}{2}-\frac{s_{1}-r}{R}\right)=1-\beta_{1;R}^2(s_1).
\end{equation}
For any $\eta \in
C^{\infty}(\Sigma_{(r)}; \widetilde{\mathbf L}|_{b_{(r)}}\otimes \wedge_{j}^{0,1}T\Sigma_{(r)})$,
let
$$
\eta_{i}(p) =\left\{
\begin{array}{ll}
\eta  & if\;\; p\in \Sigma_{i0}\cup\{|s_{i}|\leq r-1\}\\
\beta_{i;2}(s_{i})\eta(s_{i},t_{i}) & if\;\; p\in \{r-1\leq |s_{i}|\leq r+1\} \\
0 & otherwise.
\end{array}
\right..
$$
If no danger of confusion  we will simply write $\eta_{i}=\beta_{i;2}\eta.$ Then $\eta_{i}$  can be considered as
a section over $\Sigma_i$. Define
\begin{equation}
\|\eta\|_{r,k-1,2,\alpha}=\|\eta_{1} \|_{\Sigma_1,j_1,k-1,2,\alpha} +
\|\eta_{2} \|_{\Sigma_2,j_2,k-1,2,\alpha}.
\end{equation}
\v\n
We now define a norm $\|\cdot\|_{r,k,2,\alpha}$ on
$C^{\infty}(\Sigma_{(r)}; \widetilde{\mathbf L}|_{b_{(r)}}).$ For any section
$\zeta\in C^{\infty}(\Sigma_{(r)}; \widetilde{\mathbf L}|_{b_{(r)}})$ denote
$$\zeta_0=\int_{ {S}^1}\zeta(r,t)dt,$$
$$\zeta_1(s_1,t_1) = (\zeta-\hat \zeta_{0})(s_1,t_1)\cdot\beta_{1;2}(s_1),\;\;\;\zeta_2(s_2,t_2)= (\zeta-\hat \zeta_{0})(s_2,t_2)\cdot\beta_{2;2}(s_{2}).$$
We define
\begin{equation}
\|\zeta\|_{r,k,2,\alpha}=\|\zeta_1\|_{\Sigma_1,j_1,k,2,\alpha} +
\|\zeta_2\|_{\Sigma_2,j_2,k,2,\alpha}+|\zeta_{0}|.
\end{equation}
Denote the resulting completed spaces by $W^{k-1,2,\alpha}(\Sigma_{(r)}; \widetilde{\mathbf L}|_{b_{(r)}}\otimes \wedge_{j_o}^{0,1}T\Sigma_{(r)})$ and $W^{k,2,\alpha}(\Sigma_{(r)}; \widetilde{\mathbf L}|_{b_{(r)}})$  respectively.
\v
In terms of the cusp holomorphic cylinder coordinates we may write
$$D^{\widetilde{\mathbf L}}|_{b_{(r)}}=\bar{\p}_{j_o} + E^{\widetilde{\mathbf{L}}}_{b_{(r)}},$$
where $\bar{\p}_{j_o}=\frac{1}{2}\left(\frac{\p}{\p s} + \sqrt{-1}\frac{\p}{\p t}\right),\;\frac{\p}{\p t}=j_{o}\frac{\p}{\p s}$ and
\begin{equation}\label{linearized operator}
E^{\widetilde{\mathbf{L}}}_{b_{(r)}}=\frac{p}{2}\left(\sum \frac{\p u_{(r)}^j}{\p s} + \sqrt{-1}\sum\frac{\p u_{(r)}^j}{\p t} \right)\frac{(\nabla^{L}_{\p_{x_j}}\mathbf e_{u_{(r)}},\mathbf  e_{u_{(r)}})}{(\mathbf e_{u_{(r)}}, \mathbf e_{u_{(r)}})}.
\end{equation}
In fact, for any $f(\mathbf{k}\otimes \mathbf{e}_{u_{(r)}})^{p}\in  \widetilde {\mathbf L}|_{b_{(r)}},$ by \eqref{eqn_D_L} we have
\begin{equation}\label{eqn_D_L_s}
D^{\widetilde{\mathbf L}}|_{b_{(r)}}(f(\mathbf{k}\otimes \mathbf{e}_{u_{(r)}})^{p})\left(\frac{\p}{\p s}\right)= \bar{\p}_{j_{o}} (f) (\mathbf{k}\otimes \mathbf{e}_{u_{(r)}})^{p}+ pf (\mathbf{k}\otimes D^{L}\mathbf{e}_{u_{(r)}})\left(\frac{\p}{\p s}\right)(\mathbf{k}\otimes \mathbf{e}_{u_{(r)}})^{p-1}.
\end{equation}
On the other hand, using $D^{L}=\frac{1}{2}\left(\nabla^{L}+u_{(r)}^*i\cdot\nabla^{L}\cdot j_{o}\right)$, we obtain that
\begin{align}
D^{L}\mathbf{e}_{u_{(r)}}\left(\frac{\p}{\p s}\right)&=\frac{1}{2}\left(\nabla^{L}+u_{(r)}^*i\nabla^{L}\cdot j_{o}\right)(\mathbf{e}_{u_{(r)}})\left(\frac{\p}{\p s}\right)\nonumber\\
&=\frac{1}{2}\left(\nabla^{L}(\mathbf{e}_{u_{(r)}})\left(\frac{\p}{\p s}\right)+u_{(r)}^*i\nabla^{L}(\mathbf{e}_{u_{(r)}})\left(\frac{\p}{\p t}\right) \right)\nonumber\\
&=\frac{1}{2}\nabla^{L}_{\p_{x_{j}}}(\mathbf{e}_{u_{(r)}}) \left(\left(\frac{\p u_{(r)}^{j}}{\p s}\right)+\sqrt{-1}  \left(\frac{\p u_{(r)}^{j}}{\p t}\right) \right)\nonumber\\\label{eqn_D_L_1}
&=\frac{1}{2}\frac{(\nabla^{L}_{\p {x_{j}}}(\mathbf{e}_{u_{(r)}}),\mathbf{e}_{u_{(r)}})}{(\mathbf{e}_{u_{(r)}},\mathbf{e}_{u_{(r)}})} \mathbf{e}_{u_{(r)}} \left(\left(\frac{\p u_{(r)}^{j}}{\p s}\right)+\sqrt{-1}  \left(\frac{\p u_{(r)}^{j}}{\p t}\right) \right)
\end{align}
where we used the fact that  $L$ is a line bundle.  Substituting \eqref{eqn_D_L_1} into \eqref{eqn_D_L_s} we get \eqref{linearized operator}.
\v
Note that $u_i(s_i,t_i)$  exponentially converges to $0$, with higher-order derivatives, as $s_i\to \infty $, $i=1,2$. We have
\begin{equation}\label{gluing surface}
E^{\widetilde{\mathbf{L}}}_{b_{(r)}}\mid_{|s_i|\leq \tfrac{r}{2}}=0,\;\;\;\;\;\;\sum\limits_{p+q=d} \left|\frac{\p^{d}E^{\widetilde{\mathbf{L}}}_{b_{(r)}}}{\partial s_i^ p \partial  t_i^q}\right|_{\tfrac{r}{2}\leq |s_i|\leq \tfrac{3r}{2}}\rightarrow 0,
\end{equation}
for $i=1,2,\forall d\geq0$, exponentially and
uniformly in $t_i$ as $r\rightarrow \infty $.
\v
For any $b=(a,v)$ with $v=\exp_{u_{(r)}}(h_{r})$, denote $\mathbf{e}_{v}=P^{\widetilde{\mathbf L}}_{b_{(r)},b}\mathbf e_{u_{(r)}}$. We have
$$(P_{b_{(r)},b}^{\widetilde{\mathbf L}})^{-1}\circ D^{\widetilde{\mathbf L}}|_{b}\circ P^{\widetilde{\mathbf L}}_{b_{(r)},b}= \bar{\p}_{j_{a}}+E^{\widetilde{\mathbf L}}_{b},$$
where $ \bar{\p}_{j_{a}}=\frac{\p}{\p s} + \sqrt{-1}j_{a}\frac{\p}{\p s}$ and
\begin{equation}\label{defn_F_b}
E^{\widetilde{\mathbf L}}_{b}=\frac{p}{2} \sum \left(\frac{\p v^j}{\p s} + \sqrt{-1}\left(j_{a}\frac{\p }{\p s}\right)(v^j)\right)\frac{((P^{\widetilde{\mathbf L}}_{b_{(r)},b})^{-1}\nabla^{L}_{\eta_j}\mathbf{e}_{v},\mathbf  e_{u_{(r)}})}{(\mathbf e_{u_{(r)}}, \mathbf e_{u_{(r)}})}.
\end{equation}
It is easy to check that
 \begin{equation}\label{eqn_diff_D}
 \|D^{\widetilde{\mathbf L}}|_{b_{(r)}}-(P_{b_{(r)},b}^{\widetilde{\mathbf L}})^{-1}\circ D^{\widetilde{\mathbf L}}|_{b}\circ P^{\widetilde{\mathbf L}}_{b_{(r)},b}\|\leq C(|a-a_{o}|+\|h\|_{k,2,\alpha,r}).
 \end{equation}
Given
$\eta\in W^{k-1,2,\alpha}(\Sigma_{(r)}; \widetilde{\mathbf L}|_{b_{(r)}}\otimes \wedge_{j_o}^{0,1}T\Sigma_{(r)})$
denote
\begin{equation*}
\left(\eta_1(s_1,t_1),\eta_2(s_2,t_2)\right) = \left(\beta_{1;2}(s_1)\eta(s_1,t_1), \beta_{2;2}(s_2)\eta(s_2,t_2)\right),\end{equation*}
\begin{equation*}
Q^{\widetilde{\mathbf L}}_{b_o}(\eta_1,\eta_2)=\zeta=(\zeta_1,\zeta_2),\;\;\zeta_{i}\in W^{k,2,\alpha}(\Sigma;\widetilde{\mathbf L}|_{b_o}).\end{equation*}
where $Q^{\widetilde{\mathbf L}}_{b_o}$ is a right inverse of $D^{\widetilde{\mathbf L}}|_{b_{o}}$. Define
\begin{equation*}
\left(Q^{\widetilde{\mathbf L}}_{b_{(r)}}\right)'(\eta):= \zeta_{(r)}=(\beta_{1;r}(s_1)\zeta_{1}(s_1,t_1)+\beta_{2;r}(s_{1}-2r)\zeta_{2}(s_{1}-2r,t_1-\tau)).
\end{equation*}

\begin{lemma}\label{DQ'} For any $\eta\in W^{k-1,2,\alpha}(\Sigma_{(r)}; \widetilde{\mathbf L}|_{b_{(r)}}\otimes \wedge_{j_o}^{0,1}T\Sigma_{(r)})$
we have
\begin{equation}\label{eqn_DQ'}
D^{\widetilde{\mathbf L}}|_{b_{(r)}}\circ  \left(Q^{\widetilde{\mathbf L}}_{b_{(r)}}\right)'(\eta)- \eta=\sum (\bar{\partial}\beta_{i;r}) \zeta_{i}+\sum  \beta_{i;r}E^{\widetilde{\mathbf{L}}}_{b_{(r)}}\zeta_{i}\end{equation}
$$+(\sum \beta_{i;r}\beta_{i;2}-1)\eta
.$$\end{lemma}

\n{\bf Proof:} It is obvious that
\begin{equation}\label{app_DS_right_inverse}
D^{\widetilde{\mathbf L}}|_{b_{(r)}}\circ  \left(Q^{\widetilde{\mathbf L}}_{b_{(r)}}\right)'(\eta)
=\eta\;\;\;\;\;\;for \;\;|s_{i}|\leq \tfrac{r}{2}.\end{equation}
It suffices to calculate the left hand side in the annulus $\{\frac{r}{2}\leq |s_i|\leq \frac{3r}{2}\}.$ By choosing $r$ large enough we may assume that $\{\frac{r}{2}\leq |s_i|\leq \frac{3r}{2}\}\subset \Sigma\setminus \Sigma(R_0)$.
 Note that in this annulus
 $$D^{\widetilde{\mathbf L}}|_{b_{o}}=\bar{\p}_{j_o,u},\;\;D^{\widetilde{\mathbf L}}|_{b_{o}}\zeta_i= D^{\widetilde{\mathbf L}}|_{u_i}\zeta_i=\eta_{i},\;\;\beta_{1;r} D^{\widetilde{\mathbf L}}_{u_{1}}\zeta_{1}+\beta_{2;r} D^{\widetilde{\mathbf L}}_{u_{2}}\zeta_{2}= \sum_{i=1}^{2} \beta_{i;r}\beta_{i;2} \eta.$$
By a direct calculation we get \eqref{eqn_DQ'}.
$\Box$
\begin{lemma}\label{aright_inverse_after_gluing}
$D^{\widetilde{\mathbf L}}|_{b_{(r)}}$ is surjective for $r$ large enough. Moreover, there is a right inverse $ Q^{\widetilde{\mathbf L}}_{b_{(\mathbf{r})}}$
such that
\begin{equation}
\label{right_estimate}
\left\| Q^{\widetilde{\mathbf L}}_{b_{(r)}}\right\|\leq  \mathsf{C}
\end{equation}
 for some constant $ \mathsf{C}>0$ independent of $ r $.
\end{lemma}
\n{\bf Proof:} We first show that
\begin{eqnarray}
\label{approximate_right_inverse_estimate_1}
\left\|\left(Q^{\widetilde{\mathbf L}}_{b_{(r)}}\right)'\right\|\leq C, \\
\label{approximate_right_inverse_estimate_2}
\left\|D^{\widetilde{\mathbf L}}|_{b_{(r)}}\circ \left(Q^{\widetilde{\mathbf L}}_{b_{(r)}}\right)'-Id\right\|\leq \frac{2}{3}
\end{eqnarray}
for some constant $C>0$ independent of $ r $. Since $0\leq \beta_{i;r}\leq1 $ we have
   \begin{align}\label{h_r_o}
   |(\zeta_{(r)})_{0}| \leq &  e^{-\alpha r}\max_{t\in S^1} | e^{\alpha r}\zeta_{(r)}(r,t)| \leq  e^{-\alpha r}\max_{t_i\in S^1} \sum | e^{\alpha r}\zeta_{i}(r,t_i)|   \\ \nonumber
 \leq  & C e^{-\alpha r} \sum_{i=1,2}\|e^{\alpha |s_{i}|}\zeta_{i}(s_{i},t_{i})|_{r-1\leq  s_{1}\leq r+1}\|_{k,2}    \leq C e^{-\alpha r} \sum \|\zeta_{i}\|_{k,2,\alpha},
   \end{align}
where we used  the  Sobolev embedding theorem in the third  inequality.
By $\|Q^{\widetilde{\mathbf L}}_{b_o}\|\leq C$  and the definition of $\|\cdot\|_{k,2,\alpha,r}$ we have
\begin{align*}
\|\zeta_{(r)}\|_{k,2,\alpha,r} &=\sum\|\beta_{i;2} (\zeta_{(r)}-(\hat{ \zeta}_{(r)})_{0})\|_{k,2,\alpha}+|(\zeta_{(r)})_{0}| \\
&\leq \sum\|\beta_{i;2} \zeta_{(r)}\|_{k,2,\alpha}+C\sum \|\zeta_{i}\|_{k,2,\alpha}   \\
 &\leq
2(C+1)  \|(\zeta_{1},\zeta_{2})\|_{k,2,\alpha} \leq C \|(\eta_1,\eta_{2})\|_{k-1,2,\alpha}\leq C\|\eta\|_{k-1,2,\alpha,r},
\end{align*}
where we used \eqref{h_r_o} in the second inequality. Then \eqref{approximate_right_inverse_estimate_1} follows.
\v
We prove \eqref{approximate_right_inverse_estimate_2}. It follows from \eqref{eqn_DQ'} that
\begin{align}
\left\|D^{\widetilde{\mathbf L}}|_{b_{(r)}}\circ \left(Q^{\widetilde{\mathbf L}}_{b_{(r)}}\right)' \eta-\eta\right\|_{k-1,2,\alpha,r}
&\leq \frac{1}{2}\|\eta\|_{k-1,2,\alpha,r}+\frac{C}{r}\sum \|\zeta_i\|_{k,2,\alpha}  \\
&  \leq   \left(\frac{C}{r}+\frac{1}{2}\right)\|\eta\|_{k-1,2,\alpha,r}. \nonumber
\end{align}
where we used $ \frac{1}{2} \leq \sum \beta_{i;r}\beta_{i;2}\leq \sqrt{2},\;\left.E^{\widetilde{\mathbf{L}}}_{b_{(r)}}\right|_{\Sigma(r/2)}=0,\; \sum_{i=1}^{2}|E^{\widetilde{\mathbf{L}}}_{b_{(r)}}|\leq Ce^{-\fc\frac{r}{2}}\mbox{ in }\;\left\{\tfrac{r}{2}\leq s_1\leq \tfrac{3r}{2 }\right\}$ in the first inequality,
and used $\|Q^{\widetilde{\mathbf L}}_{b_o}\|\leq C$
in the last inequality. Then \eqref{approximate_right_inverse_estimate_2} follows when  $r $ large enough. \v
The estimate \eqref{approximate_right_inverse_estimate_2} implies that $D^{\widetilde{\mathbf L}}|_{b_{(r)}}\circ \left(Q^{\widetilde{\mathbf L}}_{b_{(r)}}\right)'$
is invertible, and a right inverse $ Q^{\widetilde{\mathbf L}}_{b_{(r)}}$  of $D^{\widetilde{\mathbf L}}|_{b_{(r)}}$ is given by
 \begin{equation}
 \label{express_right_inverse-1}
 Q^{\widetilde{\mathbf L}}_{b_{(r)}}= \left(Q^{\widetilde{\mathbf L}}_{b_{(r)}}\right)'\left[ D^{\widetilde{\mathbf L}}|_{b_{(r)}}\circ  \left(Q^{\widetilde{\mathbf L}}_{b_{(r)}}\right)'\right]\inv.
 \end{equation}
Then the Lemma follows.  $\Box$
\v
\v
For any $\zeta+\hat \zeta_{0}\in \ker D^{\widetilde{\mathbf L}}|_{b_{o}},$
we set
\begin{equation}
 \zeta_{(r)}=\beta_{1;r}(s_1)\zeta_{1}(s_1,t_1)+\beta_{2;r}(s_{1}-2r)\zeta_{2}(s_{1}-2r,t_{1}-\tau)+\hat \zeta_{0},
\end{equation}
Define $I^{\widetilde{\mathbf L}}_{(r)}:\ker D^{\widetilde{\mathbf L}}|_{b_{o}}
\to \ker D^{\widetilde{\mathbf L}}|_{b_{(r)}}$ by
\begin{equation}\label{ker-iso}
I^{\widetilde{\mathbf L}}_{(r)}(\zeta+\hat \zeta_{0})=\zeta_{(r)}- Q^{\widetilde{\mathbf L}}_{b_{(r)}}\circ D^{\widetilde{\mathbf L}}|_{b_{(r)}}(\zeta_{(r)}).
\end{equation}
\begin{lemma}\label{lem_est_I_r}
$I^{\widetilde{\mathbf{L}}}_{(r)}: \ker D^{\widetilde{\mathbf L}}|_{b_{o}}\longrightarrow \ker D^{\widetilde{\mathbf L}}|_{b_{(r)}}$
is an isomorphism for $r$ large enough, and
$$\|I^{\widetilde{\mathbf L}}_{(r)}\|\leq \mathsf C$$
for some constant $\mathsf C>0$ independent of $r$.
\end{lemma}
\begin{proof} Let $\zeta+\hat \zeta_{0}\in \ker D^{\widetilde{\mathbf L}}|_{b_{o}}$ with $I^{\widetilde{\mathbf{L}}}_{(r)}(\zeta+\hat{\zeta}_{0})=0$.
By \eqref{ker-iso}  and \eqref{right_estimate}, we have
$$\|\zeta_{(r)}\|_{k,2,\alpha,r} =\left\|I^{\widetilde{\mathbf L}}_{(r)}(\zeta+\hat{\zeta}_{0}) -  \zeta_{(r)}\right\|_{k,2,\alpha,r}\leq C\|D^{\widetilde{\mathbf L}}|_{b_{(r)}}
\left(\zeta_{(r)}\right)\|$$
for some constant $C>0.$ A direct culculation gives us
\begin{equation}\label{eqn_Dh_r}
D^{\widetilde{\mathbf L}}|_{b_{(r)}}(\zeta_{(r)})
=\sum_{i=1}^2 \beta_{i;r}\bar{\partial}_{j_{o}}(\zeta+\hat \zeta_{0})
+\sum_{i=1}^2 (\bar{\partial}\beta_{i;r}) \zeta_{i} +
\sum\beta_{i;r}E^{\widetilde{\mathbf{L}}}_{b_{(r)}}\zeta_{i}+E^{\widetilde{\mathbf{L}}}_{b_{(r)}}\hat \zeta_{0}.
\end{equation}
Since $E^{\widetilde{\mathbf{L}}}_{b_{(r)}}|_{\Sigma(r/2)}=E^{\widetilde{\mathbf{L}}}_{b_{o}}|_{\Sigma(r/2)},$ by $D^{\widetilde{\mathbf L}}|_{b_{o}}(\zeta+\hat \zeta_{0})=\bar{\partial}_{j_{o}}(\zeta+\hat \zeta_{0})=0,$ we have
\begin{equation}\label{delta_I}
\|\zeta_{(r)}\|_{k,2,\alpha,r}\leq \frac{C}{r}
(\|\zeta\|_{k,2,\alpha} + |\zeta_0|)\end{equation} for some constant $C>0$.
\v
Let $ \epsilon'\in (0,1)$ be a constant. By Lemma \ref{tube_ker_c-1}  we can choose $R$ large enough such that
$$\|\zeta|_{|s_i|\geq 2R}\|_{k,2,\alpha}\leq \epsilon'(\|\zeta\|_{k,2,\alpha}+|\zeta_0|).$$
Therefore
\begin{equation}\label{lower_bound_h_r}
\|\zeta_{(r)}\|_{k,2,\alpha,r}\geq \|\zeta|_{|s_i|\leq 2R_1}\|_{k,2,\alpha} + |\zeta_0|
\geq (1 - \epsilon')(\|\zeta\|_{k,2,\alpha} + |\zeta_0|),
\end{equation}
for $r>4R$.  Then \eqref{delta_I} and \eqref{lower_bound_h_r} give us $\zeta=0$ and $\zeta_{0}=0$. Hence $I^{\widetilde{\mathbf{L}}}_{(r)}$ is injective.
\vskip 0.1in
\noindent
Since $H^0(\Sigma, \widetilde{\mathbf{L}}\mid_{b})$ and $H^0(\Sigma, \widetilde{\mathbf{L}}\mid_{b_{(r)}})$ have the same dimension, the Lemma follows.
\end{proof}
\v

\subsection{ Equivariant Gluing}\label{Equivariant Gluing}
Let $b_o$ be as in \S\ref{s_intro_2} and \S\ref{s_intro_3}.
Assume that $(\Sigma_i,\mathbf y_i, q)$ is stable. Let $G_{b_o}=(G_{b_{o1}}, G_{b_{o2}})$ be the isotropy group at $b_o$, thus,
$$G_{b_{o}}=\{\phi=(\phi_1,\phi_2)|\;\phi_i\in Diff^+(\Sigma_i),\; \phi_i^{*}(j_i,\mathbf{y}_i, q,
u_{i})=(j_i,\mathbf{y}_i,q,
u_i)\}.$$
 Obviously,  $G_{b_o}$ is a subgroup of $G_{\mathbf{a}_{o}}$.
\v\n
It is easy to check that  the operator  $D^{\widetilde{\mathbf{L}}}|_{b_o}$ is $G_{b_{o}}$-equivariant. Then we may choose a $G_{b_{o}}$-equivariant right inverse $Q^{\widetilde{\mathbf{L}}}_{b_o}$. $G_{b_{o}}$ acts on $ker D^{\widetilde{\mathbf{L}}}|_{b_o}$ in a natural way. Put
$$ker\; D^{\mathbf{L}}|_{b_o}=ker D^{\widetilde{\mathbf{L}}}|_{b_o}/G_{b_{o}}.$$

Note that we used the cusp holomorphic cylinder coordinates $(s_i, t_i)$ on $\Sigma_i$ near $q$ to do gluing in \S\ref{s_intro_2} and \S\ref{s_intro_3}.  Since the cut-off function $\beta (s)$
depends only on $s$, $G_{b_o}$ acts on $\widetilde{O}_{b_{o}}(\delta,\rho).$
\v
Denote $a_{(r)}=(\Sigma_{(r)},j,y)$ and $b_{(r)}=(a_{(r)},u_{(r)}).$
Denote by $G_{b_{(r)}}$ the isotropy group at $b_{(r)}$. $G_{b_{(r)}}$ acts on $ker\;D^{\widetilde{\mathbf{L}}}\mid_{b_{(r)}}$ in a natural way. Put
$$ker\; D^{\mathbf{L}}|_{b_{(r)}}=ker\; D^{\widetilde{\mathbf{L}}}|_{b_{(r)}}/G_{b_{(r)}}.$$

It is easy to see that $G_{b_{(r)}}$ is a subgroup of $G_{b_{o}}$ and can be seen as rotation in the gluing part. Then the gluing map is a $\frac{|G_{b_{o}}|}{|G_{b_{(r)}}|}$-multiple covering map.
 Since $\beta_{1;r}$ is independent of $\tau$,  $Q'_{b_{(r)}}$ is  $G_{b_{(r)})}$-equivariant.
 By the definition of $Q^{\widetilde{\mathbf{L}}}_{b_{(r)}}$ and the $G_{b_{(r)}}$-equivarance of $D^{\widetilde{\mathbf{L}}}|_{b_{(r)}}$, we conclude that  $Q^{\widetilde{\mathbf{L}}}_{b_{(r)}}$ is  $G_{b_{(r)}}$-equivariant. By the uniqueness, $f^{\widetilde{\mathbf{L}}}$ is $G_{b_{(r)}}$-equivariant. Then we have

\begin{lemma}\label{equi_isomorphism}
{\bf (1)} $I^{\widetilde{\mathbf{L}}}_{(r)}: \ker D^{\widetilde{\mathbf L}}|_{b_{o}}\longrightarrow \ker D^{\widetilde{\mathbf L}}|_{b_{(r)}}$ is a $\frac{|G_{b_{o}}|}{|G_{b_{(r)}}|}$-multiple covering map.
\v
{\bf (2)} $I^{\widetilde{\mathbf{L}}}_{(r)}$ induces a isomorphism $I^{\mathbf{L}}_{(r)}: \ker D^{\mathbf L}|_{b_{o}}\longrightarrow \ker D^{\mathbf L}|_{b_{(r)}}$.
\end{lemma}

\subsection{\bf Pregluing several nodes }\label{s_intro_3}

The above estimates can be generalized to gluing several nodes. Let $(\Sigma, j, {\bf y})$ be a marked nodal Riemann surface of genus $g$ with $n$ marked points. Suppose that $\Sigma$ has $\mathfrak{e}$ nodal points $\mathbf{q}=(q_{1},\cdots,q_{\mathfrak{e}})$ and $\iota$ smooth components. For each node $q_i$ we can glue $\Sigma$ and $u$ at $q_i$ with gluing parameters $(\mathbf{r})=((r_1,\tau_1),...,(r_\mathfrak{e}, \tau_\mathfrak{e}))$ to get $\Sigma_{(\mathbf{r})}$ and $u_{(\mathbf{r})}$.
The operators  $D^{\widetilde{\mathbf{L}}}_{b_o}$ and $D^{\widetilde{\mathbf{L}}}_{b_{(\mathbf{r})}}$ are $G_{b_{o}}$-equivariant and $G_{b_{(\mathbf{r})}}$-equivariant respectively. We may choose a $G_{b_{o}}$-equivariant right inverse $Q^{\widetilde{\mathbf{L}}}_{b_o}$ and $G_{b_{(\mathbf{r})}}$-equivariant right inverse $Q^{\widetilde{\mathbf{L}}}_{b_{(\mathbf{r})}}$. $G_{b_{o}}$ ( resp.$G_{b_{(\mathbf{r})}}$ ) acts on $ker D^{\widetilde{\mathbf{L}}}_{b_o}$ ( resp. $ker D^{\widetilde{\mathbf{L}}}_{b_{(\mathbf{r})}}$ )  in a natural way. Put
$$ker\; D^{\mathbf{L}}|_{b_o}=ker D^{\widetilde{\mathbf{L}}}|_{b_o}/G_{b_{o}},\;\;
ker\; D^{\mathbf{L}}|_{b_{(\mathbf{r})}}=ker D^{\widetilde{\mathbf{L}}}|_{b_{(\mathbf{r})}}/G_{b_{(\mathbf{r})}}.
$$
By the same methods as in \S \ref{s_intro_2}, \S\ref{s_intro_3} and \S\ref{Equivariant Gluing}
we can prove

\begin{lemma}\label{lem_est_I_r} {\bf (1)}
$I^{\widetilde{\mathbf L}}_{(\mathbf{r})}: \ker D^{\widetilde{\mathbf L}}|_{b_{o}}\longrightarrow \ker D^{\widetilde{\mathbf L}}|_{b_{(\mathbf{r})}}$
is a $\frac{|G_{b_{o}}|}{|G_{b_{(\mathbf{r})}}|}$-multiple covering map for $r_i$, $1\leq i \leq \mathfrak{e}$, large enough, and
$\|I^{\widetilde{\mathbf L}}_{(\mathbf{r})}\|\leq \mathsf C$
for some constant $\mathsf C>0$ independent of $(\mathbf{r})$.
\v
{\bf (2)} $I^{\widetilde{\mathbf{L}}}_{(\mathbf{r})}$ induces a isomorphism $I^{\mathbf{L}}_{(\mathbf{r})}: \ker D^{\mathbf L}|_{b_{o}}\longrightarrow \ker D^{\mathbf L}|_{b_{(\mathbf{r})}}$.
\end{lemma}
\v\v

For fixed $(\mathbf{r})$ we consider the family of maps:
$$
\mathcal {F}_{(\mathbf{r})}: \mathbf{A}\times   W^{k,2,\alpha}(\Sigma_{(\mathbf{r})},u^{\star}_{(\mathbf{r})}TM)
\times {\mathcal W}^{k,2,\alpha}(\Sigma_{(\mathbf{r})},\widetilde{\mathbf L}|_{b_{(\mathbf{r})}})\to W^{k-1,2,\alpha}(\Sigma_{(\mathbf{r})},\wedge^{0,1}T\Sigma_{(\mathbf{r})}\otimes \widetilde{\mathbf L}|_{b_{(\mathbf{r})}})$$
defined by
\begin{equation}\label{def_F_r}
\mathcal {F}_{(\mathbf{r})}(\mathbf{s},h,\xi)= P^{\widetilde{\mathbf{L}}}_{b,b_{(\mathbf{r})}}
\circ D^{\widetilde{\mathbf L}}|_{b}\circ (P^{\widetilde{\mathbf{L}}}_{b,b_{(\mathbf{r})}})^{-1}\xi,
\end{equation}
where $b=((\mathbf r),\mathbf{s},v_{\mathbf r})$ and $v_{\mathbf r}=\exp_{u_{(\mathbf r)}}h$.
By implicit function theorem (Theorem \ref{details_implicit_function_theorem}, Theorem \ref{smooth_implicit_function_theorem}),
there exist $\delta>0$, $\rho>0$ and a small neighborhood $\widetilde{O}_{(\mathbf{r})}$ of $0 \in \ker\;D^{\widetilde{\mathbf L}}|_{u_{(\mathbf{r})}}$ and a unique smooth map
$$f^{\widetilde{\mathbf L}}_{(\mathbf{r})}: \widetilde{\mathbf{O}}_{b_{(\mathbf{r})}}(\delta,\rho)
\times \widetilde{O}_{(\mathbf{r})}\rightarrow W^{k-1,2,\alpha}(\Sigma_{(\mathbf{r})},\wedge^{0,1}T\Sigma_{(\mathbf{r})}\otimes \widetilde{\mathbf L}|_{b_{(\mathbf{r})}})$$ such that for any $(b,\zeta)\in \widetilde{\mathbf{O}}_{b_{(\mathbf{r})}}(\delta,\rho)
\times \widetilde{O}_{(\mathbf{r})}$
\begin{equation}\label{glu_solu}
D^{\widetilde{\mathbf L}}|_{b}\circ(P^{\widetilde{\mathbf{L}}}_{b,b_{(\mathbf{r})}})^{-1}\left(\zeta + Q^{\widetilde{\mathbf L}}_{b_{(\mathbf{r})}}\circ f^{\widetilde{\mathbf L}}_{\mathbf{s},h,(\mathbf{r})}(\zeta)\right)=0.
\end{equation}
Together with Lemma \ref{lem_est_I_r} and $I^{\mathbf L}_{(\mathbf{r})}$ we have gluing map
$$Glu^{\mathbf L}_{(\mathbf r)}:\mathbf{F}\mid_{[b_o]}\to \mathbf{F}\mid_{[b]}\;\;\;for\; any \;b\in \;\mathbf{O}_{[b_{(\mathbf{r})}]}(\delta,\rho)$$ defined\;by
$$Glu^{\mathbf L}_{(\mathbf r)}([\zeta]):=\left[(P^{\widetilde{\mathbf{L}}}_{b,b_{(\mathbf{r})}})^{-1}\left(I^{\widetilde{\mathbf L}}_{(\mathbf{r})}\zeta + Q^{\widetilde{\mathbf L}}_{b_{(\mathbf{r})}}\circ f^{\widetilde{\mathbf L}}_{\mathbf{s},h,(\mathbf{r})}I^{\widetilde{\mathbf L}}_{(\mathbf{r})}\zeta \right)\right],\;\;\;\;\forall [\zeta]\in \mathbf{F}\mid_{[b_o]}.$$
\v
Given a frame $e_{\alpha}(z)$ on $\widetilde{\mathbf{F}}\mid_{b_o}$, $1\leq \alpha\leq rank\; \widetilde{\mathbf{F}},$ as Remark \ref{isotropy group} we have a $G_{b_o}$-equivariant frame field
$$e_{\alpha}((\mathbf{r}),\mathbf{s}, h)(z)=(P^{\widetilde{\mathbf{L}}}_{b,b_{(\mathbf{r})}})^{-1}
\left(I^{\widetilde{\mathbf L}}_{(\mathbf{r})}e_{\alpha}  + Q^{\widetilde{\mathbf L}}_{b_{(\mathbf{r})}}\circ f^{\widetilde{\mathbf L}}_{\mathbf{s},h,(\mathbf{r})}I^{\widetilde{\mathbf L}}_{(\mathbf{r})}e_{\alpha} \right)(z) $$
over $D_{R_{0}}^{*}(0)\times \widetilde{\mathbf{O}}_{b_o}(\delta_{o},\rho_{o})$,
where $z$ is the coordinate on $\Sigma$, and
$$D_{R_{0}}^{*}(0):=\bigoplus_{i=1}^{\mathfrak{e}}\left\{(r, \tau)\mid R_0< r<\infty, \;\tau\in S^1\right\}.$$
For any fixed $(\mathbf{r})$, $e_{\alpha}$ is smooth with respect to $\mathbf{s},h$ over $\widetilde{\mathbf{O}}_{b_o}(\delta_{o},\rho_{o})$.
\v

\subsection{Gluing $J$-holomorphic maps}\label{J-gluing map}

We recall some results in \cite{LS-1}.
Let $b_o=(a_{o}, u)$, and $u$ be a $(j_{o},J)$-holomorphic map. The domain $\Sigma $ of elements of $\mathcal{M}^{\Gamma}$ are marked nodal Riemann surfaces. Suppose that $\Sigma$
has nodes $p_{1},\cdots,p_{\mathfrak{e}}$ and  marked points $y_{1},\cdots,y_{n}$. We choose local coordinate system $\mathbf{A}$ and define a pregluing map $u_{(\mathbf{r})}: \Sigma_{(\mathbf{r})}\to M$ as in \S\ref{s_intro_3}. Set
 $$t_{i}=e^{-2r_{i}-2\pi\tau_{i}}, \;\;\;|\mathbf{r}|=min\{r_1,...,r_{\mathfrak{e}}\},\;\;\;b_{(\mathbf r)}:=(a_{o},\mathbf{(r)}, u_{(\mathbf r)}).$$
\v
Let $K$ be a $N$-dimensional linear space. Let
$$ \mathfrak{i}: K\times \mathbf{A}\times W^{k,2,\alpha}\left(\Sigma(R_0),(u\mid_{\Sigma(R_0)})^*TM \right)$$$$\to
W^{k-1,2,\alpha}\left(\Sigma(R_0),(u\mid_{\Sigma(R_0)})^*TM\otimes \wedge^{0,1}_{j_\mathbf{s}}T^{*}\Sigma(R_0)\right)
$$
be a smooth map such that
 $D_{v} + d\mathfrak{i}_{(\kappa,\mathbf{s},v\mid_{\Sigma(R_0)})}$
is surjective for any $(\kappa,b)\in K\times  O_{b_o}(\mathrm{R},\delta,\rho)$, where $b=(\mathbf{s},(\mathbf r),v)$, $v=\exp_{u_{(\mathbf{r})}}h$ and $\mathbf{O}_{b_o}(\mathrm{R},\delta,\rho)=\cup_{|\mathbf r|\geq \mathrm{R}}\mathbf{O}_{b_{(\mathbf r)}}(\delta,\rho)$.
\v\n
Define a thickned Fredholm system
$(K\times \mathbf{O}_{b_{o}}(\mathrm{R},\delta,\rho), K\times \E|_{\mathbf{O}_{b_{o}}(\mathrm{R},\delta,\rho)}, \mathcal{S})$ with
 \begin{equation}\label{regu.equ}
\mathcal{S}(\kappa,b)=\bar{\partial}_{j_\mathbf{s},J}v +  \mathfrak{i}(\kappa, b).
\end{equation}
The following lemma is proved in
\cite{LS-1}.
\begin{lemma}\label{isomor of ker}
For $|\mathbf{r}|>R_0$ there is an isomorphism
$I_{(\mathbf r)}: \ker D \mathcal{S}_{( \kappa_{o}, b_{o})}\longrightarrow \ker D \mathcal{S}_{( \kappa_{o},b_{(\mathbf r)})}.$
\end{lemma}
\v

For fixed $(\mathbf{r})$ we consider the family of maps:
\begin{align*}& \mathcal{F}_{(\mathbf{r})}: K   \times \mathbf{A} \times W^{k,2,\alpha}\left(\Sigma_{(\mathbf r)},u_{(\mathbf r)}^*TM \right)\to W^{k-1,2,\alpha}\left(\Sigma_{(\mathbf r)},(u_{(\mathbf r)}^*TM\otimes \wedge^{0,1}_{j_\mathbf{s}}T^{*}\Sigma_{(\mathbf r)}\right),\;\;\;\;\;\\
& \mathcal{F}_{(\mathbf{r})}(\kappa,\mathbf{s},h)=\Psi_{j_{\mathbf{s}},j_{\mathbf{s}_{o}}}\Phi_{u_{(\mathbf{r})}}(h)^{-1}
\left(\bar{\partial}_{j_{\mathbf{s}},J}v
+ \mathfrak{i}(\kappa,b)\right),
\end{align*}
where $b=(\mathbf{s},(\mathbf{r})  , v),\;v=\exp_{u_{(\mathbf{r})}}h$  and
\begin{equation*}
 \Psi_{j_{\mathbf{s}},j_{\mathbf{s}_{o}}}\Phi_{u_{(\mathbf{r})}}(h)^{-1}:W_{\mathbf{r}}^{k-1,2,\alpha}(\Sigma_{(\mathbf{r})}, v^{*}TM\otimes \wedge^{0,1}_{j_{\mathbf{s}}}T^{*}\Sigma_{(\mathbf{r})})\to W_{\mathbf{r}}^{k-1,2,\alpha}(\Sigma_{(\mathbf{r})}, u_{(\mathbf{r})}^{*}TM\otimes \wedge^{0,1}_{j_{\mathbf{s}}}T^{*}\Sigma_{(\mathbf{r})}).
\end{equation*}
By implicit function theorem (Theorem \ref{details_implicit_function_theorem}, Theorem \ref{smooth_implicit_function_theorem}),
there exist $\delta>0$, $\rho>0,R>0$, a small neighborhood $O_{(\mathbf{r})}$ of $0 \in \ker\;D\mathcal S|_{b_{(\mathbf{r})}}$ and a unique smooth map
$$f_{(\mathbf{r})}: \mathbf{A}
\times O_{(\mathbf{r})}\rightarrow W^{k-1,2,\alpha}(\Sigma_{(\mathbf{r})},u_{(\mathbf{r})}^{*}TM\otimes \wedge^{0,1}T\Sigma_{(\mathbf{r})})$$ such that for any $(\kappa,\mathbf{s},h)\in \mathbf{A}
\times O_{(\mathbf{r})}$ and $|\mathbf r|>R,$
\begin{equation}\label{glu_solu_J_hol}
\mathcal S(\kappa,b)=0.
\end{equation}

Let $(s_{l}^{i},t_{l}^{i}),l=1,2$ be the cylinder coordinates near the node $q_{i}$. Set $$V_{i}:=\cup_{l=1}^{2}\left\{\left.\left(s_{l}^{i},t_{l}^{i}\right)\in\Sigma\;\right|\;\tfrac{r_{i}}{2} \leq |s_{l}^{i}|\leq \tfrac{3r_{i}}{2} \right\}.$$
Let $\pi: K   \times W_{\mathbf r}^{k,2,\alpha}\left(\Sigma_{(\mathbf r)},u_{(\mathbf r)}^*TM \right)\to W_{\mathbf r}^{k,2,\alpha}\left(\Sigma_{(\mathbf r)},u_{(\mathbf r)}^*TM \right)$ (resp. $\pi:   K   \times \mathcal W^{k,2,\alpha}\left(\Sigma,u^*TM \right)\to \mathcal W^{k,2,\alpha}\left(\Sigma,u^*TM \right)$) be the projection.
 Denote $$Glu_{\mathbf{s},(\mathbf{r})}(\kappa,\xi)=I_{(\mathbf{r})}(\kappa,\xi)+Q_{(\kappa_{o},b_{(\mathbf{r})})}\circ f_{\mathbf{s},(\mathbf{r})} \circ I_{\mathbf{r}}(\kappa,\xi),$$
$$Glu^{*}_{\mathbf{s},(\mathbf{r})}(\kappa,\xi)=I^{*}_{(\mathbf{r})}(\kappa,\xi)+Q^{*}_{(\kappa_{o},b_{(\mathbf{r})})}\circ f_{\mathbf{s},(\mathbf{r})} \circ I_{\mathbf{r}}(\kappa,\xi).$$
In \cite{LS-1} we proved
\begin{theorem}\label{coordinate_decay-2}
	There exists positive  constants  $\mathsf C, \mathsf{d}, R_{0}$ such that for any $(\kappa,\xi)\in \ker D \mathcal{S}_{(\kappa_{o},b_{o})}$ with $\|(\kappa,\xi)\|< \mathsf{d}$, and any $X_{i}\in \{\frac{\p}{\p r_{i}},\frac{\p}{\p \tau_{i}}\},i=1,\cdots,\mathfrak{e}$, the following estimate hold
	$$  \left\| X_{i}\left(Glu^{*}_{\mathbf{s},(\mathbf{r})}(\kappa,\xi) \right) \right\|_{k-2,2,\alpha}
	\leq  \mathsf{C} e^{-(\fc-5\alpha)\tfrac{r_{i}}{4} },$$
	$$ \left\|X_{i}X_{j}\left(Glu^{*}_{\mathbf{s},(\mathbf{r})}(\kappa,\xi) \right) \right\|_{k-2,2,\alpha}+\left\|\left.X_{i}\left(Glu^{*}_{\mathbf{s},(\mathbf{r})}(\kappa,\xi) \right)\right|_{V_{j}} \right\|_{k-2,2,\alpha}
	\leq  \mathsf{C} e^{-(\fc-5\alpha)\tfrac{r_{i}+r_{j}}{4} },$$
	$1\leq i\neq j\leq\mathfrak{e},$
	for any $\mathbf{s}\in \bigotimes_{i=1}^{\mathfrak{e}} O_i$.
 \end{theorem}

\section{Smoothness of $Glu^{\widetilde{\mathbf L}}_{(\mathbf r)}(e_{\alpha})\mid_{\Sigma(R_0)}$}

We have shown in \S \ref{s_intro_3} that for any fixed $(\mathbf{r})$, $Glu^{\widetilde{\mathbf L}}_{(\mathbf r)}$ is smooth with respect to $\mathbf{s},h$ over $\widetilde{\mathbf{O}}_{b_o}(\delta_{o},\rho_{o})$. In this section we discuss the smoothness with respect to $(\mathbf{r}),\mathbf{s},h$. To this end we need to fix a Riemann surface $\Sigma_{(\mathbf{R_o})}$.
\v
We first consider gluing one node case.
Let $\alpha_{(r)}:[0,2r]\to [0,2R_{0}]$ be a smooth increasing  function satisfying
$$
  \alpha_{(r)}(s)=\left\{
\begin{array}{ll}
s\;\;\;\;  &  if\; s\in [0,\frac{R_{0}}{2}-1] \\
 \frac{R_{0}}{2}+\frac{R_{0}}{2r-R_{0}}(s-R_{0}/2) \;\;\;\;\;\; & if\;s\in [R_{0}/2,2r-R_{0}/2]   \\
 s-2r+2R_{0} & if\; s\in [2r-\frac{R_{0}}{2}+1,2r]
\end{array}
\right.
$$
Set $\alpha_{(r)}:[-2r,0]\to [-2R_{0},0]$ by $\alpha_{(r)}(s)=-\alpha_{(r)}(-s).$
We can define a map $\varphi_{(r)}:\Sigma_{(r)}\to \Sigma_{(R_{0})}$ as follows:
$$\varphi_{(r)}=\left\{
\begin{array}{ll}
p, & p\in \Sigma(R_{0}/4).\\
(\alpha_{(r)}(s_{1}),t_{1})   \;\;\;\;\;\;\;& (s_{1},t_{1})\in \Sigma_{(r)}\setminus \Sigma(R_{0}/4).
\end{array}
\right.
$$
Obviously, $\varphi_{(r)}^{-1}(\mathbf{y})=\mathbf{y}.$ For any $s_{1}\in [0,2r]$ and $s_{2}\in [-2r,0]$, we have
\begin{equation}
s_{1}=s_{2}+2r \Longleftrightarrow \alpha_{(r)}(s_{1})=\alpha_{(r)}(s_{2})+2R_{0}.
\end{equation}
Then we obtain a family of Riemann surfaces $\left(\Sigma_{(R_{0})},(\varphi_{(r)}^{-1})^{*}j_{r} ,\varphi_{(r)}^{-1}(\mathbf{y})\right)$.
\v

Denote $u^{\circ}_{(r)}:=u_{(r)}\circ \varphi_{r}^{-1}.$  The $\varphi_{(r)}^{-1}$ induce an isomorphism
$$
(\varphi_{(r)}^{-1})^{*}:W^{k,2,\alpha}\left(\Sigma_{(r)},u_{(r)}^*TM \right)\to W^{k,2,\alpha}\left(\Sigma_{(R_{0})},(u^{\circ}_{(r)})^*TM \right).
$$
For any $h\in W^{k,2,\alpha}\left(\Sigma_{(R_{0})},(u_{(R_{0})})^*TM \right)$, denote $v=\exp_{u_{(R_{0})}}(h),$ we have map $\varphi^{*}_{r}v:\Sigma_{(r)}\to M.$ There exists a family of functions $\hat{h}^{\circ}_{(r)}\in W^{k,2,\alpha}\left(\Sigma_{(R_{0})},(u_{(R_{0})})^*TM \right)$ such that ${u}^{\circ}_{(r)}=\exp_{u_{(R_{0})}}\left(\hat{h}^{\circ}_{(r)}\right).$
It is easy to check that
$\hat{h}^{\circ}_{(r)}$ is  a smooth family of functions and for any $l\in \mathbb Z^{+}$,
\begin{equation}
\left\|\hat{h}^{\circ}_{(r)}\right\|_{C^{l}\left(\Sigma_{(R_{0})}\right)}\leq C(r,l),
\end{equation}
for some constant $C(r,l)>0$ depending only on $r$ and $l.$  Denote
$$j^{\circ}_{r}=\left(\varphi_{(r)}^{-1}\right)^{*}j_{r},\;\;\;b_{(R_0)}:=(j_{R_{0}}, u_{(R_{0})}),\;\;\;\;{b}^{\circ}_{(r)}:=(j^{\circ}_{r},u^{\circ}_{(r)}),\;\;\;\;b:=(j_{r}^{\circ}, v).$$

 Let $(s,t)$ be  the holomorphic  coordinates on $\Sigma_{(r)}\setminus \Sigma(R_{0})$ such that
$j_{r}(\frac{\p}{\p s})=\frac{\p}{\p t},j_{r}(\frac{\p}{\p t})=-\frac{\p}{\p s}.$ Denote  $(s^{\circ},t^{\circ})=\varphi_{r}(s,t)$. Then we have
 \begin{equation}\label{eqn_r_bull_cx}
j_{r}^{\circ}\frac{\p}{\p s^{\circ}}=\frac{1}{\varphi'_{r}(s)}\frac{\p}{\p t^{\circ}},\;\;\;\;\;\;j_{r}^{\circ}\frac{\p}{\p t^{\circ}}=
 -{\varphi'_{r}(s)}\frac{\p}{\p s^{\circ}}\;\;\;\;\mbox{ in }\Sigma_{(R_{0})}\setminus \Sigma(R_{0}/4).
 \end{equation}
 Then for any $\eta\in W^{k-1,2,\alpha}(\Sigma_{(R_{0})},v^{*}\widetilde{\mathbf{L}}\otimes \wedge^{0,1}_{j_{r}^{\circ}}T\Sigma_{(R_{0})})$ and $p\in \Sigma_{(R_{0})}$,
$\Psi^{\widetilde{\mathbf L}}_{j_{r}^{\circ},j_{R_{0}}}\eta(p)$ is a smooth family of isomorphisms.
Since $M,u_{(R_{0})}$ and $v$ are smooth,
 $\Phi^{\widetilde{\mathbf L}}_{u_{(R_{0})},v}$ is also a smooth family of isomorphisms. It follows that $P^{\widetilde{\mathbf L}}_{b,b_{(R_{0})}}$ is a smooth family of isomorphisms. In particular, $P^{\widetilde{\mathbf L}}_{b^{\circ}_{(r)},b_{(R_{0})}}$ is smooth with respect to $(r)$.

\v

We have the operator
$$
D^{\widetilde{\mathbf L}}|_{b^{\circ}_{(r)}}: W^{k,2}(\Sigma_{(R_0)},L|_{b^{\circ}_{(r)}})\to W^{k-1,2}(\Sigma_{(R_{0})},L|_{b^{\circ}_{(r)}}\otimes \wedge^{0,1}_{j_{r}^{\circ}}T\Sigma_{(R_{0})}).
$$
Using \eqref{eqn_r_bull_cx} one can easily check that
$$
(\varphi_{r}^{-1})_{*}D^{\widetilde{\mathbf L}}|_{b^{\circ}_{(r)}}=D^{\widetilde{\mathbf L}}|_{b_{(r)}}.
$$
  We define  $Q'^{\widetilde{\mathbf L}}_{b^{\circ}_{(r)}}: W^{k-1,2}(\Sigma_{(R_{0})},L|_{b^{\circ}_{(r)}}\otimes \wedge^{0,1}_{j_{r}^{\circ}}T \Sigma_{(R_{0})})\to W^{k,2}(\Sigma_{(R_{0})},L|_{b^{\circ}_{(r)}})$ by $$Q'^{\widetilde{\mathbf L}}_{b^{\circ}_{(r)}}\eta_{r}^{\circ}=(\varphi^{-1}_{r})^{*}\left(Q'^{\widetilde{\mathbf L}}_{b_{(r)}}(\varphi_{r}^{*}\eta_{r}^{\circ})\right).$$  We define  $Q^{\widetilde{\mathbf L}}_{b^{\circ}_{(r)}}: W^{k-1,2}(\Sigma_{(R_{0})},L|_{b^{\circ}_{(r)}}\otimes \wedge^{0,1}_{j_{r}^{\circ}}T \Sigma_{(R_{0})})\to W^{k,2}(\Sigma_{(R_{0})},L|_{b^{\circ}_{(r)}})$ by
  $$
  Q^{\widetilde{\mathbf L}}_{b^{\circ}_{(r)}}=Q'^{\widetilde{\mathbf L}}_{b^{\circ}_{(r)}}[D^{\widetilde{\mathbf L}}|_{b^{\circ}_{(r)}}Q'^{\widetilde{\mathbf L}}_{b^{\circ}_{(r)}}]^{-1}.
  $$
 We can also define $I^{\widetilde{\mathbf L}}_{b^{\circ}_{(r)}}:Ker D^{\widetilde{\mathbf L}}|_{b_{o}} \to Ker D^{\widetilde{\mathbf L}}|_{b^{\circ}_{(r)}}$ by
$$
I^{\widetilde{\mathbf L}}_{b^{\circ}_{(r)}}(\zeta)=(\varphi^{-1}_{r})^{*}(I^{\widetilde{\mathbf L}}_{(r)}(\zeta)).
$$
  It is easy to see that
  $$
C(k,\alpha,r)^{-1}\|Q^{\widetilde{\mathbf L}}_{b_{(r)}}\|\leq \|Q^{\widetilde{\mathbf L}}_{b^{\circ}_{(r)}}\|\leq C(k,\alpha,r)\|Q^{\widetilde{\mathbf L}}_{b_{(r)}}\|
  $$
  where $C(k,\alpha,r)$ is a constant depending only on $k,\alpha$ and $r.$

 \v
 Denote
 $$D^{\widetilde{\mathbf L}}=P^{\widetilde{\mathbf L}}_{b^{\circ}_{(r)},b_{(R_{0})}}\circ D^{\widetilde{\mathbf L}}|_{b^{\circ}_{(r)}}\circ \left[P^{\widetilde{\mathbf L}}_{b^{\circ}_{(r)},b_{(R_{0})}}\right]^{-1},\;\;\;\;\;\;\;\;\;\;Q^{\widetilde{\mathbf L}}=P^{\widetilde{\mathbf L}}_{b^{\circ}_{(r)},b_{(R_{0})}}\circ Q^{\widetilde{\mathbf L}}_{b^{\circ}_{(r)}}\circ \left[P^{\widetilde{\mathbf L}}_{b^{\circ}_{(r)},b_{(R_{0})}}\right]^{-1},$$
 $$
 Q'^{\widetilde{\mathbf L}}=P^{\widetilde{\mathbf L}}_{b^{\circ}_{(r)},b_{(R_{0})}}\circ Q'^{\widetilde{\mathbf L}}_{b^{\circ}_{(r)}}\circ \left[P^{\widetilde{\mathbf L}}_{b^{\circ}_{(r)},b_{(R_{0})}}\right]^{-1}. $$
For any $\eta\in W^{k-1,2}(\Sigma_{(R_{0})},L|_{b_{(R_{0})}}\otimes \wedge^{0,1}_{j_{R_{0}}}T\Sigma_{(R_{0})})$
 denote  $\eta^{\circ}_{r}= \left[P^{\widetilde{\mathbf L}}_{b^{\circ}_{(r)},b_{(R_{0})}}\right]^{-1}\eta $
and $\eta_{r}=\varphi_{r}^{*}\eta^{\circ}_{r}$. Let $(\eta_{1},\eta_{2})=(\beta_{1;2}(s_{1})\eta_{r}( s_{1}, t_{1}),\beta_{1;2}(s_{2})\eta_{r}(s_{2}, t_{2}) ).$ Denote $(h_{1},h_{2})=Q_{b_o}^{\widetilde{\mathbf L}}(\eta_{1},\eta_{2})$.
Since $(s_{i}^{\circ},t_{i}^{\circ})=\varphi(s_{i},t_{i})$, we have
\begin{align*}
&Q'^{\widetilde{\mathbf L}}_{b^{\circ}_{(r)}}\circ \left[P^{\widetilde{\mathbf L}}_{b^{\circ}_{(r)},b_{(R_{0})}}\right]^{-1}\eta=Q'^{\widetilde{\mathbf L}}_{b^{\circ}_{(r)}}\circ \eta_{r}^{\circ}=(\varphi^{-1}_{r})^{*}\left(Q'^{\widetilde{\mathbf L}}_{b_{(r)}}( \eta_{r})\right)
\\
=&\beta_{1;r}\cdot \varphi_{r}^{-1}(s^{\circ}_{1}) h_{1}\cdot \varphi_{r}^{-1}(s^{\circ}_{1},t^{\circ}_{1}) +\beta_{2;r}\cdot \varphi_{r}^{-1}(s^{\circ}_{1}-2R_{0}) h_{1}\cdot \varphi_{r}^{-1}(s^{\circ}_{1}-2R_{0}, t_{1}^{\circ}-2R_{0}).
\end{align*}
Since $P^{\widetilde{\mathbf L}}_{b^{\circ}_{(r)},b_{(R_{0})}}$ is a smooth, we have
$$
\|\nabla^{l}_{r} Q'^{\widetilde{\mathbf L}}\eta\|_{k,2}\leq C(k,\alpha,r)\|\eta\|_{k-1+l,2}.
$$
 Similarly, we obtain that
\begin{align}\label{eqn_smo_QI}
&\|\nabla^{l}_{r} D^{\widetilde{\mathbf L}}(\zeta)\|_{k-1,2}\leq C(k,\alpha,r)\|\zeta\|_{k+l,2},
\;\;\;\;\|\nabla^{l}_{r} Q^{\widetilde{\mathbf L}}\eta\|_{k,2}\leq C(k,\alpha,r)\|\eta\|_{k-1+l,2},\\
 &\left\| \nabla^{l}_{r} \left(P^{\widetilde{\mathbf L}}_{b^{\circ}_{(r)},b_{(R_{0})}}\circ I^{\widetilde{\mathbf L}}_{b^{\circ}_{(r)}}\right)(\zeta')\right\|_{k,2}\leq C(k,\alpha,r)\|\zeta'\|_{k+l,2,\alpha},\;\;\;\forall\zeta'\in Ker D^{\widetilde{\mathbf L}}_{b_{o}} .
\end{align}
The above estimates can be generalized to gluing several nodes.
\v
We can define $\varphi_{(\mathbf r)}$, $b^{\circ}_{(\mathbf r)}$ and $b_{(\mathbf {R}_{0})}$ as above.
We define a map $$
\mathcal {F}:D_{\mathbf{R}_{0}}^{*}(0)\times\mathbf{A}\times W^{k,2}(\Sigma_{(\mathbf{R}_{0})},u_{(\mathbf{R}_{0})}^{\star}TM)
\times {\mathcal W}^{k,2,\alpha}(\Sigma,\widetilde{\mathbf L}|_{b_{o}})\to W^{k-1,2,\alpha}(\Sigma_{(\mathbf{R}_{0})}, \widetilde{\mathbf L}|_{b_{(\mathbf{R}_{0})}}\otimes\wedge_{j_{\mathbf{R}_{0}}}^{0,1}T\Sigma_{(\mathbf{R}_{0})})
$$
by
$$
\mathcal {F}((\mathbf{r}),\fs,h,\zeta)= P^{\widetilde{\mathbf{L}}}_{b,b_{(\mathbf{R}_0)}}
\circ D^{\widetilde{\mathbf L}}_{b}\circ (P^{\widetilde{\mathbf{L}}}_{b,b^{\circ}_{(\mathbf{r})}})^{-1}I^{\widetilde{\mathbf L}}_{b^{\circ}_{(\mathbf r)}}\zeta,
$$
where $b=(\Sigma_{(\mathbf{R}_{0})},(\mathbf r),
\fs,v),$ $\fs(j_{o},\mathbf{y})=0$ and $v=\exp_{u_{(\mathbf{R}_{0})}}(h)$.
By the same argument as in Lemma \ref{lem_smooth_F} we see that $\mathcal {F}$ is a smooth function. There exists a family smooth function $\hat h_{(\mathbf r)}$ such that  $u^{\circ}_{(\mathbf r)}=\exp_{u_{(\mathbf R_{0})}}(\hat h_{(\mathbf r)}).$
Obviously, when $\|h-\hat h_{(\mathbf r)}\|_{k,2}$ small we have
$$
\mathcal {F}((\mathbf{r}),\fs,h,\zeta)=P^{\widetilde{\mathbf{L}}}_{b,b_{(\mathbf{R}_0)}}(P^{\widetilde{\mathbf{L}}}_{b,b^{\circ}_{(\mathbf r)}})^{-1}(\varphi_{\mathbf r}^{-1})^{*} \left(\mathcal {F}_{(\mathbf r)}(\fs,h',I^{\widetilde{\mathbf L}}_{b_{(\mathbf r)}}(\zeta))\right),
$$
where $h'=(\exp_{u_{(\mathbf r)}}^{-1}\circ (\exp_{u_{(\mathbf R)_{0}}}(h)\circ \varphi_{(\mathbf{r})}).$
  Then by \eqref{glu_solu} and the uniquiness  of the implicit function we have $f^{\widetilde{\mathbf L}}_{a,h,b^{\circ}_{(\mathbf r)}}=(\varphi_{\mathbf r}^{-1})^{*}f^{\widetilde{\mathbf L}}_{a,h',(\mathbf r)} $ such that for any $\zeta\in Ker D^{\widetilde{\mathbf L}}_{b_{o}}$,
$$
D^{\widetilde{\mathbf L}}_{b}\circ(P^{\widetilde{\mathbf{L}}}_{b,b^{\circ}_{(\mathbf r)}})^{-1}\left(I^{\widetilde{\mathbf L}}_{b^{\circ}_{(\mathbf r)}}(\zeta) + Q^{\widetilde{\mathbf L}}_{b^{\circ}_{(\mathbf r)}}\circ f^{\widetilde{\mathbf L}}_{\fs,h,b^{\circ}_{(\mathbf r)}}I^{\widetilde{\mathbf L}}_{b^{\circ}_{(\mathbf r)}}(\zeta)\right)=0
$$
as $|\fs|$ and $\|h\|_{k,2}$ small.
 Since
$$D_{\zeta}\mathcal {F}_{(\mathbf{r})}(\fs,h',0)(\zeta_{1})= \mathcal {F}_{(\mathbf{r})}(\fs,h', \zeta_{1}),\;\;\;D_{\zeta}\mathcal {F}_{(\mathbf{r})}(\fs,h',0)(0)=0,\;\;\;\; $$
we have a explicit formula for $f^{\widetilde{\mathbf L}}_{a,h',(\mathbf r)} $ ( see \eqref{eqn_psi} in the proof of Theorem \ref{smooth_implicit_function_theorem}):
$$ f^{\widetilde{\mathbf L}}_{a,h', (\mathbf r)}\circ I^{\widetilde{\mathbf L}}_{{(\mathbf r)}}(\zeta)= \mathcal {F}_{(\mathbf r)}(0,0, \mathcal {H}^{-1}(I^{\widetilde{\mathbf L}}_{{(\mathbf r)}}\zeta)),$$
where and $\mathcal {H}$ is defined by
\begin{equation} \label{eqn_psi}
\mathcal {H}(x):=x+  Q^{\widetilde{\mathbf L}}_{{(\mathbf r)}} \left(\mathcal {F}_{(\mathbf r)}(\fs,h',x)-  \mathcal {F}_{(\mathbf r)}(0,0,x) \right).\end{equation}
It follows that
$$
f^{\widetilde{\mathbf L}}_{a,h,b^{\circ}_{(\mathbf r)}}\circ I^{\widetilde{\mathbf L}}_{b^{\circ}_{(\mathbf r)}}(\zeta)=P^{\widetilde{\mathbf{L}}}_{b,b^{\circ}_{(\mathbf r)}}(P^{\widetilde{\mathbf{L}}}_{b,b_{(\mathbf{R}_0)}})^{-1}\mathcal {F}\left((\mathbf{r}),0,\hat h_{(\mathbf r)},(I^{\widetilde{\mathbf L}}_{b^{\circ}_{(\mathbf r)}})^{-1}\circ \mathcal {H}_{\circ}^{-1}\circ(I^{\widetilde{\mathbf L}}_{b^{\circ}_{(\mathbf r)}}(\zeta))\right)
$$
where
\begin{equation*}
\mathcal {H}_{\circ}(x):=x+  Q^{\widetilde{\mathbf L}}_{b^{\circ}_{(\mathbf r)}} \left(P^{\widetilde{\mathbf{L}}}_{b,b^{\circ}_{(\mathbf r)}}(P^{\widetilde{\mathbf{L}}}_{b,b_{(\mathbf{R}_0)}})^{-1}\mathcal {F}((\mathbf r),\fs,h,x)-  (P^{\widetilde{\mathbf{L}}}_{b^{\circ}_{(\mathbf r)},b_{(\mathbf{R}_0)}})^{-1} \mathcal {F}((\mathbf r),0,\hat h_{(\mathbf r)},x) \right).
\end{equation*}
Choose $\delta,\rho$ small and $|\mathbf r|$ big enough.
By \eqref{eqn_smo_QI} and $\nabla_{r_{i}}\mathcal {H}_{\circ}^{-1}=-\mathcal {H}^{-1}_{\circ}\circ(\nabla_{r_{i}}\mathcal {H}_{\circ})\circ\mathcal {H}_{\circ}^{-1},$
 one can check that
 $$
 \| \nabla^{l}_{\mathbf{r}} (P^{\widetilde{\mathbf L}}_{b^{\circ}_{(\mathbf r)},b_{(\mathbf R_{0})}}f^{\widetilde{\mathbf L}}_{a,h,b^{\circ}_{(\mathbf r)}}I^{\widetilde{\mathbf L}}_{b^{\circ}_{(\mathbf r)}})(\zeta)\|_{k-1,2}\leq C\|\zeta\|_{k+l,2,\alpha},
 $$ where $\nabla^{l}_{\mathbf{r}}=\nabla^{l_{1}}_{r_{1}}\cdots\nabla^{l_{\mathfrak e}}_{r_{\mathfrak e}}$ with $\sum_{i=1}^{\mathfrak{e}}l_{i}=l.$
Then   we have  for any $\zeta\in Ker D^{\widetilde{\mathbf L}}|_{b_{o}},$
$$\left\| \nabla^{l}_{\mathbf{r}}\left[P^{\widetilde{\mathbf L}}_{b^{\circ}_{(\mathbf r)},b}\left(I^{\widetilde{\mathbf L}}_{b^{\circ}_{(\mathbf r)}}+Q^{\widetilde{\mathbf L}}_{b^{\circ}_{(\mathbf{r})}}\circ f^{\widetilde{\mathbf L}}_{a,h,b^{\circ}_{(\mathbf r)}} \circ I^{\widetilde{\mathbf L}}_{b^{\circ}_{(\mathbf r)}}\right)(\zeta)\right]\right\|_{k,2}\leq C\|\zeta\|_{k+l,2,\alpha}.$$
On the other hand, since $u$  is smooth and $D^{\widetilde{\mathbf L}}|_{b_{o}}\zeta=\bar{\p}_{j_{o},u}\zeta=0$, by the standard elliptic estimate we have
$$
\|\zeta\|_{k+l,2,\alpha}\leq C\|\zeta\|_{k,2,\alpha}.
$$
Hence
$Glu^{\widetilde{\mathbf L}}_{\fs,(\mathbf r),h_{(\mathbf r)}}(e_{\alpha})\circ \varphi_{\mathbf r}^{-1}$ is a smooth family. We have proved
\begin{lemma}\label{smoothness of Glu^L}
There exists positive  constants  $ \mathsf{d},R$ such that
for any  $\zeta\in \ker D^{\widetilde{\mathbf L}}|_{b_{o}}$,   $h\in W^{k,2,\alpha}\left(\Sigma_{(R_{0})},(u_{(R_{0})})^*TM \right)$ with    $$\|\zeta\|_{\mathcal W,k,2,\alpha}\leq \mathsf{d},\;\;\;\;\;\|h-\hat h_{(\mathbf r)}\|< \mathsf{d},\;\;\;\;\;|\mathbf r|\geq R, $$
$( \varphi_{\mathbf r}^{-1})^{*}(Glu^{\widetilde{\mathbf L}}_{\fs,(\mathbf r),h'}(e_{\alpha}) )$ is smooth with respect to $(\mathbf{s}, (\mathbf{r}),h)$ \; for any $e_{\alpha}\in \ker D^{\widetilde{\mathbf L}}|_{b_o}$,  where  $h'=(\exp_{u_{(\mathbf r)}}^{-1}\circ (\exp_{u_{(\mathbf R)_{0}}}(h)\circ \varphi_{(\mathbf{r})})$. In particular $Glu^{\widetilde{\mathbf L}}_{\fs,(\mathbf r),h'}(e_{\alpha})\mid_{\Sigma(R_0)}$
 is smooth.
\end{lemma}

\section{\bf Estimates of derivatives with respect to gluing parameters}\label{diff gluing parameters}

In this section we prove the following theorem.
\begin{theorem}\label{thm_est_mix_deri}
Let $l\in \mathbb Z^+$ be a fixed integer.
 Let $u:\Sigma\to M$ be a $(j,J)$-holomorphic map.	Let $\fc\in (0,1)$ be a fixed constant.    For any $0<\alpha<\frac{1}{100\fc}$, there exists positive  constants  $\mathsf C_{l}, \mathsf{d},R$ such that for any  $\zeta\in \ker D^{\mathbf L}|_{b_{o}}$,   $(\kappa,\xi)\in \ker D \mathcal{S}_{(\kappa_{o},b_{o})}$ with    $$\|\zeta\|_{\mathcal W,k,2,\alpha}\leq \mathsf{d},\;\;\;\;\;\|(\kappa,\xi)\|< \mathsf{d},\;\;\;\;\;|\mathbf r|\geq R, $$ restricting to the compact set $\Sigma(R_0)$, the following estimate hold.
	\begin{equation}\label{eqn_est_L_1st}
	\left\|X_{i}  \left( Glu_{\mathbf{s},h_{(\mathbf r)},(\mathbf{r})}^{\widetilde{\mathbf L}}(\zeta) \right) 	\right\|_{C^{l}(\Sigma(R_{0}))}
	\leq \mathsf C_{l} e^{-(\fc-5\alpha)\tfrac{r_{i}}{4} } ,
	\end{equation}
	\begin{align} \label{eqn_est_L_2nd}
	&\left\|X_{i}X_{j} \left( Glu^{\widetilde{\mathbf L}}_{\mathbf{s},h_{(\mathbf r)},(\mathbf{r})}(\zeta) \right)
	\right\|_{C^{l}(\Sigma(R_{0}))} \leq   \mathsf C_{l} e^{-(\fc-5\alpha)\tfrac{r_{i}+r_{j}}{4} }
	\end{align}
	for any $X_{i}\in \{\frac{\p}{\p r_{i}},\frac{\p}{\p \tau_{i}}\},i=1,\cdots,\mathfrak{e}$,    $\mathbf{s}\in \bigotimes_{i=1}^{\mathfrak{e}} O_i$ and any $1\leq i\neq j\leq \mathfrak{e}$.
\end{theorem}

\v

\subsection{Some operators}\label{some_oper}
\v
 \v
It is important to estimate the derivative of the gluing map with respect to $r$. To this end we need to take the derivative $\frac{\p}{\p r}$ for $ Q^{\widetilde{\mathbf L}}_{b_{(\mathbf{r})}}$ and other operators. Note that both $Q^{\widetilde{\mathbf L}}_{b_{(\mathbf{r})}}$ and $f^{\widetilde{\mathbf L}}_{(r)}$ are global operators, so we need a global estimate. On the other hand,
since the domain $\Sigma_{(r)}$ depends on $r$, in order to make the meaning of the derivative $\frac{\p}{\p r}$ for these operators clear we need transfer all operators defined over $\Sigma_{(r)}$ into a family of operators  defined over $\stackrel{\circ}{\Sigma}_1\cup \stackrel{\circ}{\Sigma}_2$, depending on $(r)$. To simplify notations we will  denote
$$\mathcal{W}^{k,2,\alpha}_{u}=\mathcal{W}^{k,2,\alpha}
(\Sigma,\widetilde{\mathbf L}|_{b_{o}})
,\;\;\; W^{k,2,\alpha}_{u}= W^{k,2,\alpha}(\Sigma,\widetilde{\mathbf L}|_{b_{o}})
,\;\;\;L_{u}^{k-1,2,\alpha}=W^{k-1,2,\alpha}(\Sigma,\widetilde{\mathbf L}|_{b_{o}}\otimes \wedge^{0,1}_{j_o}T^{*}{\Sigma}).
$$
$$
 {\mathcal W}^{k,2,\alpha}_{r,u_{(r)}} ={\mathcal W}^{k,2,\alpha}(\Sigma_{(r)},\widetilde{\mathbf L}|_{b_{(r)}}), \;\;\;\;L^{k-1,2,\alpha}_{r,u_{(r)}}=W^{k-1,2,\alpha}(\Sigma_{(r)}, \widetilde{\mathbf L}|_{b_{(r)}}\otimes \wedge^{0,1}_{j_{o}}T^{*}\Sigma_{(r)}).
$$
\v
We first define three maps
$$H_{r}:L^{k-1,2,\alpha}_{r,u_{(r)}}\to L^{k-1,2,\alpha}_u,\;\;\;\;P_{r}:L^{k-1,2,\alpha}_u\to L^{k-1,2,\alpha}_{r,u_{(r)}},\;\;\;\phi_{r}:\mathcal W^{k,2,\alpha}_u\to W^{k,2,\alpha}_{r,u_{(r)}}$$
as following. Given $\eta\in L^{k-1,2,\alpha}_{r,u_{(r)}}$ define
$$H_{r}\eta=(\beta_{1;2}(s_1)\eta(s_1,t_1),\beta_{2;2}(s_2)\eta(s_2,t_2)),$$
where $\eta(s_{i},t_i)$ is the expression of $\eta$ in terms the coordinates $(s_{i},t_{i}).$
Given
$(\eta_1,\eta_2)\in L^{k-1,2,\alpha}_u$ define
\begin{equation} \label{eqn_s.n.1}
P_{r}(\eta_{1},\eta_{2})=\left\{
\begin{array}{ll}
\eta_{i} & if\;\; p\in  \Sigma(r/2) \\
\beta_{1;2}(s_1)\eta_{1}(s_1,t_1)+\beta_{2;2}(s_{1}-2r) \eta_{2}(s_{1}-2r,t_1-\tau) & if\;\; p\in  \Sigma_{(r)}\setminus \Sigma(r/2)
\end{array}
\right..
\end{equation}
  If no danger of confusion  we will denotes \eqref{eqn_s.n.1} by $P_{r}(\eta_{1},\eta_{2})=\sum\beta_{i;2}\eta_{i}.$
Given
$(\zeta_1+\hat \zeta_0,\;\zeta_2+\hat \zeta_0)\in \mathcal  W^{k,2,\alpha}_u$ with $\supp\; \zeta_{i}\subset \Sigma(3r/2)$, define
\begin{align*}\left.\phi_{r}\left(\zeta_1+\hat \zeta_0, \zeta_{2}+\hat \zeta_0\right)\right|_{\Sigma(r/2)}&=\left.\left(\zeta_{i}+\hat \zeta_0\right)(s_i,t_i)\right|_{\Sigma(r/2)},\\
\left.\phi_r\left(\zeta_1+\hat \zeta_0,\zeta_2+\hat \zeta_0\right)\right|_{ \frac{r}{2} \leq s_{1} \leq \frac{3r}{2}}&=\left.\left(\zeta_1(s_1,t_1)+ \zeta_2(s_1- 2r,t_1-\tau)+\hat \zeta_0\right)\right|_{ \frac{r}{2} \leq s_{1} \leq \frac{3r}{2}}.
\end{align*}
By \eqref{beta_rel.} one can check that
\begin{equation}\label{equ.PH}
P_{r}  H_{r} =Id,\;\;\;\;\;\; H_{r}  P_{r} (\eta_{1},\eta_{2})=(\tilde \xi_{1},\tilde \xi_{2}).
\end{equation}
where
$$\tilde \xi_{1}=\beta_{1;2}\left(\beta_{1;2}\eta_1(s_1,t_1)+\beta_{2;2}\eta_{2}(s_1-2r, t_1-\tau)\right),$$ $$\tilde \xi_{2}=\beta_{2;2}\left(\beta_{1;2}\eta_1(s_2+2r,t_2+\tau)+\beta_{2;2}\eta_{2}(s_2, t_2)\right).
$$
In particular, $H_{r}$ is injective and $P_{r}$ is surjective.
\v
Next we introduce the following three operators
$$\left(Q'^{\widetilde{\mathbf L}}_{b_{(r)}}\right)^{*}:L^{k-1,2,\alpha}_{r,u_{(r)}}\to W^{k,2,\alpha}_u,\;\;\;\;\left(Q^{\widetilde{\mathbf L}}_{b_{(r)}}\right)^{*}:L^{k-1,2,\alpha}_{r,u_{(r)}}\to W^{k,2,\alpha}_u,\;\;\;\left(I^{\widetilde{\mathbf L}}_{(r)}\right)^{*}:\ker D^{\widetilde{\mathbf L}}|_{b_{o}}
\to \mathcal W^{k,2,\alpha}_u .$$
Given
$ \eta \in \lkar$, denote
\begin{equation} (\zeta_1,\zeta_2)=Q^{\widetilde{\mathbf L}}_{b_o}H_r\eta
.\end{equation}
Set
\begin{equation}
\zeta_{r}^{*}=(\beta_{1;r}(s_1)\zeta_{1}(s_1,t_1),\;\;\beta_{2;r}(s_2)\zeta_{2}(s_2,t_2))\in \wka.
\end{equation}
Define
\begin{align}\label{eqn_def_Q*}
&\left(Q'^{\widetilde{\mathbf L}}_{b_{(r)}}\right)^{*}\eta =\zeta_r^*,\;\;\;\; \left(Q^{\widetilde{\mathbf L}}_{b_{(r)}}\right)^*=\left(Q'^{\widetilde{\mathbf L}}_{b_{(r)}}\right)^{*}  \left(D^{\widetilde{\mathbf L}}|_{b_{(r)}} Q'^{\widetilde{\mathbf L}}_{b_{(r)}}\right)^{-1}.
\end{align}
Then we have maps
$$\left(Q'^{\widetilde{\mathbf L}}_{b_{(r)}}\right)^{*}  P_r:L^{k-1,2,\alpha}_u\to W^{k,2,\alpha}_u,\;\;\;\;\;
\left(Q^{\widetilde{\mathbf L}}_{b_{(r)}}\right)^{*}  P_r: L^{k-1,2,\alpha}_u\to W^{k,2,\alpha}_u.$$

\v\n
For any $\zeta+\hat \zeta_{0}\in \ker D^{\widetilde{\mathbf L}}|_{b_{o}},$
where $\zeta=(\zeta_1 ,\zeta_2  )\in W^{k,2,\alpha}_u,$   we set
\begin{equation}
\label{definition_h_ker}
\zeta_{r}^{*}= \left(\zeta_{1}\beta_{1;r}+\hat{\zeta}_{0},\;\; \zeta_{2}\beta_{2;r}+\hat{\zeta}_{0} \right).
\end{equation} Define
\begin{align}\label{def_I*}
&\left(I^{\widetilde{\mathbf L}}_{(r)}\right)^{*}(\zeta+\zeta_{0})=\zeta_{r}^*-\left(Q^{\widetilde{\mathbf L}}_{b_{(r)}}\right)^{*}\circ D^{\widetilde{\mathbf L}}|_{b_{(r)}}\circ\phi_{r}\zeta^*_{r}.
\end{align}
By the definition we have
$$
I^{\widetilde{\mathbf L}}_{(r)}=\phi_{r}\circ \left(I^{\widetilde{\mathbf L}}_{(r)}\right)^{*},\;\;\;\;\;\;Q^{\widetilde{\mathbf L}}_{b_{(\mathbf{r})}}=\phi_{r}\circ \left(Q'^{\widetilde{\mathbf L}}_{b_{(r)}}\right)^{*}.
$$
Define an operator $X:L^{k-1,2,\alpha}_{u}\to L^{k-1,2,\alpha}_{r,u_{(r)}}$ by
$$
X(\eta_{1},\eta_{2})=D^{\widetilde{\mathbf L}}|_{b_{(r)}}Q'^{\widetilde{\mathbf L}}_{b_{(r)}}P_{r}(\eta_{1},\eta_{2})-P_{r}(\eta_{1},\eta_{2}).
$$
Using $E^{\widetilde{\mathbf L}}_{u_{i}}=0,$ one can check that
\begin{align*}\label{def.X}
X(\eta_1,\eta_2)=\sum (\bar{\partial}\beta_{i;r}) h_{i}+\sum  \beta_{i;r}E^{\widetilde{\mathbf L}}_{u_{(r)}}h_{i}
+\left(\sum \beta_{i;r}\beta_{i;2}-1\right)\sum \beta_{i;2}\eta_{i},\end{align*}
where $(h_{1},h_{2})=Q^{\widetilde{\mathbf L}}_{b_{o}}H_{r}P_{r}(\eta_{1},\eta_{2}).$ Obviously, $supp X(\eta_1,\eta_2)\subset \{\frac{r}{2}\leq |s_{i}|\leq \frac{3r}{2}\}.$

\v
Let $b=(a, v)\in \widetilde{\mathbf{O}}_{b_o}(\delta_{o},\rho_{o})$, where $v=\exp_{u_{(r)}}h$.
We define
$$Glu_{a,h,(r)}^{\widetilde{\mathbf L},*}:=\left(I^{\widetilde{\mathbf L}}_{(r)}\right)^{*}+\left(Q^{\widetilde{\mathbf L}}_{b_{(r)}}\right)^{*}\circ f^{\widetilde{\mathbf L}}_{a,h,(r)} \circ I^{\widetilde{\mathbf L}}_{(r)}: \widetilde{\mathbf{F}}\mid_{b_o}\to W^{k,2,\alpha}(\Sigma,\widetilde{\mathbf L}|_{b_o}).$$
This definition can be extended to the gluing several nodes case in a natural way:
$$Glu_{\mathbf{s},h,(\mathbf r)}^{\widetilde{\mathbf L},*}:=\left(I^{\widetilde{\mathbf L}}_{(\mathbf{r})}\right)^{*}+\left(Q^{\widetilde{\mathbf L}}_{b_{(\mathbf{r})}}\right)^{*}\circ f^{\widetilde{\mathbf L}}_{\mathbf{s},h,(\mathbf{r})} \circ I^{\widetilde{\mathbf L}}_{(\mathbf{r})}: \widetilde{\mathbf{F}}\mid_{b_o}\to W^{k,2,\alpha}(\Sigma,\widetilde{\mathbf L}|_{b_o}).$$
It is easy to see that, restricting to $\Sigma(R_0)$, we have $Glu^{\widetilde{\mathbf L},*}_{\mathbf{s},h,(\mathbf{r})}(\zeta)=P^{\mathbf{L}}_{b,b_{(\mathbf{r})}}\circ Glu^{\widetilde{\mathbf L}}_{\mathbf{s},h,(\mathbf{r})}(\zeta)$ for any $\zeta\in D^{\widetilde{\mathbf L}}|_{b_{o}}$.

\subsection{\bf Estimates of the first derivatives}

 Let  $\eta=(\eta_{1},\cdots,\eta_{\iota})\in  L^{k-1,2,\alpha}_u.$
Denote
$$
D_{l}^{i}(R_{0})=\left\{\left.\left(s_{l}^{i},t_{l}^{i}\right)\in\Sigma\;\right| |s_{l}^{i}|\geq R_{0} \right\},\;\;\;\;D^{i}(R_{0})=\cup_{l=1}^{2}D_{l}^{i}(R_{0}).
$$
Denote $h_{(\mathbf r)}=\pi\circ Glu_{\mathbf{s},(\mathbf r)}(\kappa,\xi),\;h^{*}_{(\mathbf r)}=\pi\circ Glu^{*}_{\mathbf{s},(\mathbf r)}(\kappa,\xi)$ and $v_{(\mathbf r)}=\exp_{u_{(\mathbf r)}}(h_{(\mathbf r)}).$
Set
$$
\beta_{1,i;R}(s^{i}_{1})=\beta\left(\frac{1}{2}+\frac{r_{i}-s^{i}_1}{R}\right),\;\;\;\;\beta_{2,i;R}(s^{i}_{2})
=\sqrt{1-\beta^2\left(\frac{1}{2}-\frac{s^{i}_{2}+r_{i}}{R}\right)}.
$$
To simplify notations we denote
$$D:=D^{\widetilde{\mathbf L}}|_{b_{(r)}},\;\;Q:=Q^{\widetilde{\mathbf L}}_{b_{(r)}},\;\;\;I^{*}=\left(I^{\widetilde{\mathbf L}}_{(r)}\right)^{*},\;\;\;\;f=f^{\widetilde{\mathbf L}}_{(r)},\;\;\;\;Q':=  \left(Q^{\widetilde{\mathbf L}}_{b_{(r)}}\right)',$$$$P=P^{\widetilde{\mathbf L}}_{b_{(r)},b},\;\;\;\;\;E=E^{\widetilde{\mathbf L}}_{b_{(r)}},\;\;\;\;  (Q^{\prime})^{*}:=\left(Q^{\widetilde{\mathbf L}}_{b_{(r)}}\right)'^{*},
\;\;Q^{*}:=\left(Q^{\widetilde{\mathbf L}}_{b_{(r)}}\right)^{*}.$$

 The following Lemmas   can be proved by the same method and word-by-word as in \cite{LS-1}, we omit them.

\begin{lemma}\label{lem_Q'_est}
 For any  $(\eta_{1},\eta_{2})\in L^{k-1,2,\alpha}_u,$ the following estimates hold:
\begin{align*}
(1)& \|(Q')^{*}P_{r}(\eta_{1},\eta_{2})|_{\frac{r}{2}\leq |s_{i}|\leq \frac{3r}{2}}\|_{k,2,\alpha}
 \leq \mathsf C\left(e^{-(\fc -\alpha)\frac{r}{4}}\sum\|\eta_{i}|_{|s_{i}|\leq r+1}\|_{k-1,2,\alpha}+ \sum_{i=1}^{2} \left\|\eta_{i}|_{ \frac{r}{4} \leq |s_{i}|\leq  r+1}\right\|_{k-1,2,\alpha}\right),\\
(2)&\left\|\frac{\p}{\p r} ((Q')^{*}P_r)  (\eta_{1},\eta_{2} )\right\|_{k-1,2,\alpha} \leq \mathsf C\left(e^{-(\fc -\alpha)\frac{r}{4}}\sum\|\eta_{i}|_{|s_{i}|\leq r+1}\|_{k-1,2,\alpha}+ \sum_{i=1}^{2} \left\|\eta_{i}|_{ \frac{r}{4} \leq |s_{i}|\leq  r+1}\right\|_{k-1,2,\alpha}\right) \\
(3)&\left\|\frac{\p}{\p r}\left(H_{r}  (DQ') ^{-1} P_{r}\right)(\eta_1,\eta_{2})\right\|_{k-2,2,\alpha} \\
& \leq \mathsf C \left[ e^{-(\fc-\alpha)\frac{r}{4}} \|  (\eta_1,\eta_{2})|_{ |s_{i}|\leq  r+1}\|_{k-1,2,\alpha} + \left\|(\eta_1,\eta_{2})|_{\frac{r}{4}\leq |s_{i}|\leq  r+1}\right\|_{k-1,2,\alpha} \right]. \\
(4)& \left\|H_{r}  (DQ') ^{-1} P_{r}(\eta_1,\eta_{2})|_{\frac{r}{2}  \leq |s_{1}|\leq \frac{3r}{2}}\right\|_{k-1,2,\alpha}  \\
    &\leq \mathsf C\left[ e^{-(\fc-\alpha)\frac{r}{4}} \|  (\eta_1,\eta_{2})|_{ |s_{i}|\leq  r+1}\|_{k-1,2,\alpha} + \left\|(\eta_1,\eta_{2})|_{\frac{r}{4}\leq |s_{i}|\leq  r+1}\right\|_{k-1,2,\alpha} \right],\\
(5)&\left\|\frac{\p}{\p r}(H_{r}P_r)(\eta_{1},\eta_{2} )\right\|_{k-2,2,\alpha} \leq  \mathsf C\sum_{i=1}^{2} \left\|\eta_{i}|_{r-1 \leq |s_{i}|\leq  r+1}\right\|_{\Sigma_{i},k-1,2,\alpha}.
\end{align*}
\end{lemma}
\begin{lemma}\label{lem_est_rI-1} There exists a constant $\mathsf C>0,$ independent of $r,$ such that  for any $h+\hat h_{0}\in \ker D^{\mathbf L}|_{b_{o}}$.
 \begin{equation}
 \label{partial_I_estimate} \left\|\tfrac{\partial}{\partial r} I^*(h+\hat h_{0})
 \right\|_{k-1,2,\alpha}
\leq \mathsf C \|h_{i}|_{\frac{r}{2}\leq |s_i|\leq \frac{3r}{2}}\|_{k  ,2,\alpha}+\mathsf C e^{(\fc-\alpha)\frac{r}{2}} |\hat h_0|.
 \end{equation}
\end{lemma}

\v\v
Denote $\nu_{(r)}=P^{-1} \circ Glu^{\widetilde{\mathbf L}}_{\mathbf{s},h_{(r)},(r)}(\zeta),$ and $\nu^*_{(r)}:=Glu^{\widetilde{\mathbf L},*}_{\mathbf{s}, h_{(r)},(r)}(\zeta)$.  Obviously,
$\nu_{(r)}=\phi_{r}(\nu^*_{(r)})$
\begin{equation}\label{eqn_def_nu}
\nu_{(r)}=I_{r}(\zeta)+Q\circ f\circ I_{(r)}(\zeta),\;\;\;\;\;\nu^{*}_{(r)}=I^{*}_{(r)}(\zeta)+Q^{*}\circ f\circ I_{r}(\zeta).
\end{equation}
Set $h_{r}^{*}= (h_{1}\beta_{1;r}+\hat{h}_{0}, h_{2}\beta_{2;r}+\hat{h}_{0} ).$ Since $u$ (resp. $v_{(r)}$) is a $(j_{o},J)$ (resp. $(j_{\mathbf{s}},J)$) holomorphic map,  we have
\begin{equation}\label{eqn_decay_F}
\sum_{i+j=d}\left|\frac{\p^{i+j} E}{\p s^{i}\p t^{j}}\right|+\sum_{i+j=d}\left|\frac{\p^{i+j} E_{b}^{\widetilde{\mathbf L}}}{\p s^{i}\p t^{j}}\right|\leq C_{d} e^{-\fc|s_{i}|},\;\;\;\;R_{0}\leq |s_{i}|\leq r.
\end{equation}
Taking derivative $D$ on $\nu_{(r)}$ we have
\begin{equation}\label{eqn_f_zeta}
f\circ I_{(r)}(\zeta)=\bar{\p}_{j_{o}} (\nu_{(r)})+E\nu_{(r)}.
\end{equation}
On the other hand, by $D^{\widetilde{\mathbf L}}_{b}(P\nu_{(r)})=0$ we have
\begin{equation}\label{eqn_nu_eq}
\bar{\p}_{j_{o}} (\nu_{(r)})+P^{-1}\left(\nabla_{\bar{\p}_{j_{o}}}(P) (\nu_{(r)})+E^{\widetilde{\mathbf L}}_{b}P\nu_{(r)}\right)=0.
\end{equation}
By the exponential decay of $u_{(r)}$ and $v_{(r)}$ we have
\begin{equation}\label{eqn_P_eq_est}
\left|\nabla_{\bar{\p}_{j_{o}}}(P)  \right|\leq C(|d u_{(r)}| +|d v_{(r)}| )\leq  C_{d} e^{-\fc|s_{i}|},\;\;\;\;R_{0}\leq |s_{i}|\leq r.
\end{equation}
By  \eqref{eqn_decay_F} and \eqref{eqn_P_eq_est} we conclude that $\nu_{(r)}$ satisfies the assumption of Lemma \ref{tube_ker} in Appendix. Then Lemma \ref{tube_ker} gives us
\begin{equation}\label{eqn_dec_nu_r}
\left\|\nu_{(r)}|_{\frac{r}{4}\leq |s_{i}|\leq \frac{7r}{4}}\right\|_{k,2,\alpha}\leq C e^{-(\fc-\alpha)\frac{r}{4}}.
\end{equation}
It follows from \eqref{eqn_f_zeta} and \eqref{eqn_dec_nu_r} that
\begin{lemma}\label{lem_est_f_r}
\begin{equation}\label{decay_exp_J_map}
\left\|H_{r}f\circ I_{(r)}(\zeta)|_{  |s_{i}|\geq \frac{r}{4}}\right\|_{k-1,2,\alpha }\leq   \mathsf Ce^{-(\fc-\alpha)\frac{r}{4}} (1+\|\zeta\|_{k,2,\alpha }),\;\;\;\;\forall \;r\geq 8R_{0}.
\end{equation}
\end{lemma}
\v
Similar Lemma 4.7 in \cite{LS-1}, we have
\begin{lemma}\label{lem_xi^*_app_1}
There exists a constant $\mathsf C>0$ such that for any $\zeta\in \ker D^{\widetilde {\mathbf L}}_{b_{o}}$   we have
$$\left\|\frac{\p}{\p s_{i}}\nu^*_{(r),i}|_{\frac{r}{2}\leq |s_{i}|\leq \frac{3r}{2}}\right\|_{k-1,2,\alpha}\leq \mathsf C e^{-(\fc-5\alpha)\frac{r}{4}}(\|\zeta\|_{k,2,\alpha}+1).$$
\end{lemma}

\v
An estimate similar to Lemma 4.8 in \cite{LS-1} can be proved:

\begin{lemma}
\begin{equation}
\label{local_calculation_4}
\left\|H_{r}\circ D\phi_{r}\left(\frac{\partial}{\partial r} \nu^*_{(r)}\right)\right\|_{k-2,2,\alpha} \leq \mathsf{C}\left( \mathsf {d}\left\|\frac{\partial}{\partial r} \nu^*_{(r)}\right\|_{k-1,2,\alpha} + e^{-(\fc-5\alpha)\frac{r}{4} } \right).
\end{equation}
\end{lemma}
\n \begin{proof} We estimate   $\left\|\beta_{1;2}  D\phi_{r}\left(\frac{\partial}{\partial r} \nu^*_{(r)}\right)\right\|_{k-1,2,\alpha}.$\;  The estimates of  $\left\|\beta_{2;2} D\phi_{r}\left(\frac{\partial}{\partial r} \nu^*_{(r)}\right)\right\|_{k-1,2,\alpha}$ is the same.  As in \cite{LS-1} we construct two smooth family  $\tilde{u}_{(r)},\tilde{h}_{(r)}$, depending on $(r)$, defined over $\Sigma_{1}$ as follows:
\begin{align}\label{defn-1}
\tilde{u}_{(r)}&=\left\{
\begin{array}{ll}
u_{(r)}, & in \;\;\Sigma_{1}(r+1), \\
u_1(q)+\beta(r+2-s_1) (u_{(r)}(s_1,t_1)-u_{1}(q)), & if \ s_{1} \geq r+1
\end{array}
\right.
\\\label{defn-h1}
\tilde{h}_{(r)}&=\left\{
\begin{array}{ll}
h_{(r)}, & in \ \  \Sigma_{1}(r+1), \\
\beta(r+2-s_1) h_{(r)}, & if\ s_{1}\geq r+1
\end{array}
\right..
\end{align} Set $\tilde{v}_{(r)}=\exp_{\tilde{u}_{(r)}}(\tilde{h}_{(r)}) $, $\tilde b=(\mathbf{s},\tilde{v}_{(r)})$ and $\tilde b_{(r)}=(\mathbf{s}_{o},\tilde{u}_{(r)}).$ We can define $\tilde{\nu}_{(r)}$ as the definition of $\tilde{h}_{(r)}.$ So the meaning of $\frac{\p \tilde u_{(r)}}{\p r} ,\frac{\p \tilde v_{(r)}}{\p r} $ and $\nabla_{\frac{\p}{\p r}}   \tilde \nu_{(r)} $ is clear.
Set
$$ \Lambda_{r}:= P^{-1} \circ D^{\widetilde{\mathbf L}}_{b}\circ P(\nu_{(r)}),\;\;\;\;\tilde\Lambda_{r}:=P^{-1}  \circ D^{\widetilde{\mathbf L}}_{\tilde b}\circ P(\tilde \nu_{(r)}).$$  Obviously, $\Lambda_{r}=0$ and $\Lambda_{r}|_{\Sigma(r+1)}=\tilde{\Lambda}_{r}|_{\Sigma(r+1)}.$
We calculate $\frac{\p}{\p r}\left(\beta_{1;2}\Lambda_{r}\right)$:
  \begin{align}\label{eqn_diff_Lam}
  \frac{\p}{\p r}\left(\beta_{1;2}\Lambda_{r}\right)&=  \frac{\p}{\p r}\left(\beta_{1;2}\tilde\Lambda_{r}\right)= \beta_{1;2} P^{-1}  \left[\nabla_{r}\left( D^{\widetilde{\mathbf L}}_{\tilde b}\circ P\right)(\tilde \nu_{(r)})+D^{\widetilde{\mathbf L}}_{\tilde b}\circ P \left(\nabla_{r} \tilde \nu_{(r)}\right)\right].
   \end{align}
   Using Theorem \ref{coordinate_decay-2}, we have
\begin{equation}\label{eqn_diff_para}
\left\|\beta_{1;2}\nabla_{r}\left( D^{\widetilde{\mathbf L}}_{\tilde b}\circ P \right)(\tilde \nu_{(r)})\right\|_{k,2,\alpha}\leq  Ce^{-(\fc-5\alpha)\frac{r}{4}}.
\end{equation}
Restricting in $\Sigma_{1}(r+1)$ we have
$$
\nabla_{r}\tilde \nu_{(r)}=\phi_{r}(\nabla_{r}\nu^{*}_{(r)})-2\nabla_{s_{2}}\nu^{*}_{(r),2},
$$
where $\nu^{*}_{(r)}=\left(\nu^{*}_{(r),1}, \nu^{*}_{(r),2}\right)$. By Lemma \ref{lem_xi^*_app_1} we have
\begin{align*}
&\left\| \nabla_{s_{2}}(P\nu^*_{(r),2})|_{\frac{r}{2}\leq |s_{1}|\leq r+1}\right\|_{k-1,2,\alpha}
\leq     Ce^{-(\fc-\alpha)\frac{r}{4}}\left(\left\|P \nu^*_{(r),2}\right\|_{k,2,\alpha}+1\right).
\end{align*}
Applying the exponential decay of $u_{(r)}$ and $v_{(r)},$  we get
\begin{align}\label{eqn_diff_nu}
&\left\| \nabla_{s_{2}}(\nu^*_{(r),2})|_{\frac{r}{2}\leq |s_{1}|\leq r+1}\right\|_{k-1,2,\alpha}
\leq     Ce^{-(\fc-\alpha)\frac{r}{4}}\left(\left\| \nu^*_{(r),2}\right\|_{k,2,\alpha}+1\right).
\end{align}

Then Lemma follows from    \eqref{eqn_diff_D}, \eqref{eqn_diff_Lam}, \eqref{eqn_diff_para} and \eqref{eqn_diff_nu}.
\end{proof}

\v\n{\bf Proof of \eqref{eqn_est_L_1st}.}
By the definition of $\nu_{(r)}^{*}$ we have
\begin{align}
\label{partial_r_kernel_estimate}
\frac{\partial}{\partial r}  \nu^*_{(r)} = \frac{\partial}{\partial r} I^*( \zeta) + \frac{\partial}{\partial r}\left (Q^* P_{r}\right) H_{r} fI_{(r)}(\zeta)  + Q^* P_{r}\frac{\partial}{\partial r}  (H_{r} fI_{(r)}(\zeta)).
\end{align}
Then multiplying $H_{r} D\phi_{r}$ on both sides of \eqref{partial_r_kernel_estimate} we get
\begin{align*}
H_{r} D\phi_{r}\frac{\partial \nu^*_{(r)}}{\partial r}  = H_{r} D\phi_{r}\frac{\partial I^*_{(r)}}{\partial r} ( \zeta) +  H_{r} D\phi_{r}\frac{\partial}{\partial r}\left (Q^* P_{r}\right) H_{r} f_{(r)}I_{(r)}(\zeta)  + H_{r} P_{r}\frac{\partial}{\partial r}  (H_{r} f_{(r)}I_{(r)}(\zeta)).
\end{align*}
It follows together with \eqref{local_calculation_4} and  Theorem \ref{coordinate_decay-2} that
\begin{align}\label{operator_partial_right_inverse_kernel}
\left\|H_{r}P_{r}\frac{\partial(H_{r}fI_{(r)}(\zeta))}{\partial r}\right\|_{k-2,2,\alpha}
\leq & C\left[\mathsf {d}\left\|\frac{\partial\nu^*_{(r)}}{\partial r} \right\|_{k-1,2,\alpha} +   e^{-(\fc-5\alpha)\tfrac{r}{4} }\right]  + (A) + (B),\end{align}
where
\begin{align*}(A)= &\left\| H_{r}D\phi_{r}\left(\frac{\partial}{\partial r}I^*(\zeta)\right)\right\|_{k-2,2,\alpha},\;\;\;
(B)=&\left\| H_{r} D\phi_{r}\left( \frac{\partial}{\partial r}(Q^*P_{r})\circ H_{r}f(I_{(r)}(\zeta))\right)\right\|_{k-2,2,\alpha}.
\end{align*}
 For any $(h_{1},h_{2})$ with $supp \;h_{i}\subset \Sigma(R_{0})\bigcup\{|s_{i}|\leq \frac{3r}{2}\}$ we have
\begin{equation}\label{eqn_def_Dp}
\beta_{1;2}D \phi_{r}(h_{1},h_{2})=\beta_{1;2} \bar{\p}_{j_{o}}(h_{1}+h_{2})+\beta_{1;2} E(h_{1}+h_{2}).
\end{equation}
Then
\begin{equation}\label{eqn_def_Dp-1}
\|H_{r}D \phi_{r}(h_{1},h_{2})\|_{k-1,2,\alpha}
\leq C\|(h_{1},h_{2})\|_{k,2,\alpha}.
\end{equation}
Taking  the derivation $\frac{\p}{\p r}$ of \eqref{eqn_def_Dp} we obtain
\begin{align}\label{eqn_def_Dp-2}
\left\| \frac{\p}{\p r}(\beta_{1;2}D \phi_{r})(h_{1},h_{2}) \right\|_{k-2,2,\alpha} \leq C\left[\left\| h_{2}|_{\frac{r}{2}\leq s_{1}\leq r+1}\right\|+e^{-(\fc-5\alpha)\tfrac{r}{4} }+\|D \phi_{r}(h_{1},h_{2})|_{r-1\leq s_{1}\leq r+1}\|\right].
\end{align}
Since $H_{r} D\phi_{r}I^*_{(r)}( \zeta)=0,$ we have
$H_{r} D\phi_{r}\frac{\partial I^*_{(r)}}{\partial r} ( \zeta)= \frac{\partial H_{r} D\phi_{r}}{\partial r} I^*_{(r)} ( \zeta).$ Then
\begin{equation}\label{A}
(A)\leq C  e^{-(\mathfrak{c} -5\alpha)\frac{r}{4} } .
\end{equation}
Since
$$ \frac{\partial}{\partial r}\left (Q^* P_{r}\right)=\frac{\partial}{\partial r}\left ((Q')^* P_{r}\right)\circ (H_{r} (DQ')^{-1} P_{r} )+(Q')^* P_{r} \circ \frac{\partial}{\partial r} (H_{r} (DQ')^{-1} P_{r} )$$
by (1), (2), (3), (4) of Lemma \ref{lem_Q'_est} we get
\begin{align}\label{eqn_est_r_Q}
\left\|\frac{\p}{\p r} (Q^{*}P_r)  (\eta_{1},\eta_{2} )\right\|_{k-1,2,\alpha} \leq \mathsf C\left(e^{-(\fc -\alpha)\frac{r}{4}}\sum\|\eta_{i}|_{|s_{i}|\leq r+1}\|_{k-1,2,\alpha}+ \sum_{i=1}^{2} \left\|\eta_{i}|_{ \frac{r}{4} \leq |s_{i}|\leq  r+1}\right\|_{k-1,2,\alpha}\right).
\end{align}
It follows from  Lemma \ref{lem_est_f_r}, \eqref{eqn_def_Dp-1} and \eqref{eqn_est_r_Q} that
\begin{equation}\label{B}
(B) \leq C  e^{-(\mathfrak{c} -5\alpha)\frac{r}{4} }.
\end{equation}
Note that $H_{r}P_{r}\frac{\partial}{\partial r}(H_{r}fI_{(r)}(\zeta))+\frac{\p}{\p r}(H_{r}P_r)H_{r}fI_{(r)}(\zeta)=\frac{\partial}{\partial r}(H_{r}fI_{(r)}(\zeta))$.
Then \eqref{A}, \eqref{B}, (5) of Lemma \ref{lem_Q'_est} together with \eqref{operator_partial_right_inverse_kernel} gives
\begin{equation}\label{estimate_operator_partial_I}
 \left\| \frac{\partial}{\partial r}(H_{r}fI_{(r)}(\zeta))\right\|_{k-1,2,\alpha}\leq C  e^{-(\mathfrak{c} -5\alpha)\frac{r}{4} }  +C\mathsf{d}\left\| \frac{\partial}{\partial r}\left(\nu^*_{(r)}\right)\right\|_{k-1,2,\alpha} .
\end{equation}
Substituting  this into \eqref{partial_r_kernel_estimate},  and using \eqref{eqn_est_r_Q}, Lemma \ref{lem_est_rI-1}, Lemma \ref{lem_est_f_r} we conclude that
\begin{equation}\label{exp_ker_coord_est}
\left\|\frac{\partial}{\partial r}\nu^*_{(r)}\right\|_{k-1,2,\alpha } \leq Ce^{-(\mathfrak{c} -5\alpha)\frac{r}{4} }  \end{equation}
when  $\mathsf{d}$   small. Since  $v_{(\mathbf r)}$  is a $(j_{a},J)$  holomorphic map, by the standard ellptic estimates we have
 \eqref{eqn_est_L_1st}.
 \v
 Repeating the all arguments in this section,  one can prove that there exists a constant $\mathsf C>0 $ such that
\begin{equation}
\left\|\frac{\partial }{\partial \tau}\nu^*_{(r)}
 \right\|_{k-1,2,\alpha}\leq \mathsf C e^{-(\fc-5\alpha)\tfrac{r}{4} }(\mathsf{d}+1)
\end{equation}
 for any $\zeta\in \ker D^{\widetilde{\mathbf L}}_{b_{o}}$.

 \v

\subsection{\bf Estimates of the second derivatives}

We can define   $H_{\mathbf r}$ and $P_{\mathbf r},\cdots$  as   before. Let $\xi=Q^{\widetilde{\mathbf L}}_{b_{o}}\LH \LP\eta$ and $\eta_{l}^{i}=\eta|_{D^{i}_{l}(R_{0})},l=1,2.$ Obviously
\begin{equation}
H_{\mathbf r}P_{\mathbf r}(\eta)|_{D^{i}(R_{0})}=\left(\beta_{1,i;2}(\sum_{l=1}^{2}\beta_{l,i;2}\eta_{l}^{i} ),\beta_{2,i;2}(\sum_{l=1}^{2}\beta_{l,i;2}\eta_{l}^{i} )\right ) .\end{equation}
Set $W_{l}^{i}=\{(s_{l}^{i},t_{l}^{i})|\frac{r_{i}}{4}\leq |s_{l}^{i}|\leq r_{i}+1\}.$ It is easy to see that for any $1\leq i\neq j\leq \mathfrak{e},$   and $l=1,2,$
\begin{align}\label{eqn_r_HP}
 \frac{\p(H_{\mathbf r}P_{\mathbf r})}{\p r_{i} }(\eta)|_{V_{j}}=0,\;\;\;\;\;\frac{\p^2(H_{\mathbf r}P_{\mathbf r})}{\p r_{i} \p r_{j}}(\eta)=0,\;\;\;\;\;\frac{\p^2E}{\p r_{i} \p r_{j}}=0,\;\;\;\;\;\frac{\p^2\beta_{l,\ell;R}}{\p r_{i} \p r_{j}}=0,\\
 \label{eqn_r_beta} supp \frac{\p E}{\p r_{i}}\subset V_{i},\;\;\;\;
 \frac{\p\beta_{l,j,r_{j}}}{\p r_{i}}=\frac{\p\beta_{l,j,2}}{\p r_{i}}=0,\;\;\;supp \frac{\p(H_{\mathbf r}P_{\mathbf r})}{\p r_{i}}\subset V_{i},\;\;\;\;supp \frac{\p\beta_{l,i,r_{i}}}{\p r_{i}}\subset V_{i}.
\end{align}
It follows that
$$
\left\|\frac{\p}{\p r_{i}}\xi\right\|_{k-1,2,\alpha}\leq C\left\|\frac{\p}{\p r_{i}} H_{\mathbf r}P_{\mathbf r}(\eta)\right\|_{k-2,2,\alpha}\leq C\|\eta|_{V_{i}}\|_{k-1,2,\alpha}.
$$
Let $\xi_{l}^{i}=\xi|_{D^{i}_{l}(R_{0})},l=1,2.$ Then $(\xi_{1}^{i},\xi_{2}^{i})$ is the restriction of $\xi$ near the node $q_{i}$.
Since $D_{b_o}^{\widetilde{\mathbf L}}\frac{\p}{\p r_{i}}\xi=\frac{\p}{\p r_{i}}( H_{\mathbf r}P_{\mathbf r}(\eta)) ,$ by Lemma \ref{tube_ker}  and \eqref{eqn_r_beta}  we have for any $j\neq i$
\begin{equation} \label{eqn_est_xi_l}
\sum_{l=1}^2\left\|\frac{\p}{\p r_{i}}\xi|_{W^{j}_{l}}\right\|_{k-1,2,\alpha}\leq C e^{-(\fc-\alpha)\frac{r_{j}}{4}} \left\|\frac{\p}{\p r_{i}} H_{\mathbf r}P_{\mathbf r}(\eta)\right\|_{k-2,2,\alpha}\leq C e^{-(\fc-\alpha)\frac{r_{j}}{4}}\|\eta|_{V_{i}}\|_{k-1,2,\alpha}
\end{equation}

\n
In the following we assume that  $1\leq i\neq j\leq \mathfrak{e}.$   It is easy to see that
\begin{equation}\label{eqn_2nd_xi}
\frac{\p^2 \xi}{\p r_{i} \p r_{j}}|_{D^{\ell}}=0,\;\;\;\;(Q^{\prime})^{*}P_{\mathbf r}(\eta)|_{D^{\ell}}=(\beta_{1,{\ell},r_{\ell}}\xi_{1}^{\ell},\beta_{2,{\ell},r_{\ell}}\xi_{2}^{\ell}),\;\;\;\;\;\;\;\;\forall\; 1\leq \ell \leq \mathfrak{e}.
\end{equation}
 Taking the derivative $\frac{\p}{\p r_{i}}$ and $\frac{\p^2}{\p r_{i}\p r_{j}}$ of $(Q^{\prime})^{*}P_{\mathbf r},$ by \eqref{eqn_r_HP}, \eqref{eqn_r_beta}, \eqref{eqn_2nd_xi} and $(Q^{\prime})^{*} P_{\mathbf r}|_{\Sigma(R_{0})}=\xi,$ we obtain
\begin{align*}
\left.\frac{\p}{\p r_{j}}((Q^{\prime})^{*} P_{\mathbf r}) (\eta)\right|_{D^{i}}&=\left. \left( \beta_{1,{i},r_{i}}\frac{\p \xi_{1}^{i}}{\p r_{j}}, \beta_{2,{i},r_{i}}\frac{\p \xi_{2}^{i}}{\p r_{j}}\right)\right|_{D^{i}},\\
\left.\frac{\p^2}{\p r_{i}\p r_{j}}((Q^{\prime})^{*} P_{\mathbf r}) (\eta)\right|_{D^{\ell}}&=\left.\delta_{\ell,i}\left(\frac{\p\beta_{1,{i},r_{i}}}{\p r_{i}}\frac{\p \xi_{1}^{i}}{\p r_{j}},\frac{\p\beta_{2,{i},r_{i}}}{\p r_{i}}\frac{\p \xi_{2}^{i}}{\p r_{j}}\right)\right|_{D^{\ell}}+\left.\delta_{\ell,j}\left(\frac{\p\beta_{1,{j},r_{j}}}{\p r_{j}}\frac{\p \xi_{1}^{j}}{\p r_{i}},\frac{\p\beta_{2,{j},r_{j}}}{\p r_{j}}\frac{\p \xi_{2}^{j}}{\p r_{i}}\right)\right|_{D^{\ell}},
\end{align*}
and $supp \frac{\p^2}{\p r_{i}\p r_{j}}((Q^{\prime})^{*} P_{\mathbf r}) (\eta)\subset V_{i}\cup V_{j}.$
 By \eqref{eqn_est_xi_l}  we get
 \begin{align} \label{sec_app_right}
&\sum_{l=1}^{2}\left\|\left.\frac{\p}{\p r_{j}} ((Q^{\prime})^{*}P_{\mathbf r})(\eta)\right|_{W_{l}^{i}}\right\|_{k-1,2,\alpha} + \left\|\frac{\p^2}{\p r_{i}\p r_{j}}((Q^{\prime})^{*} P_{\mathbf r}) (\eta)\right\|_{k-2,2,\alpha} \\
\leq &Ce^{-\frac{(\fc-\alpha)r_{j}}{4}} \left\|\eta|_{V_{i}}\right\|_{k-1,2,\alpha}
+Ce^{-\frac{(\fc-\alpha)r_{i}}{4}} \left\|\eta|_{V_{j}}\right\|_{k-1,2,\alpha}.\nonumber
\end{align}
A direct calculation gives us
$$
\LH \circ DQ'\circ \LP|_{D^{i}}=\left(\beta_{1,{i},2} \left(\bar{\p}_{j_{o}}(\sum_{\ell=1}^{2}\beta_{\ell,{i},r_{i}}\xi_{\ell}^{i})+E\sum_{\ell=1}^{2}\beta_{\ell,{i},r_{i}}\xi_{\ell}^{i}\right),\beta_{2,{i},2} \left(\bar{\p}_{j_{o}}(\sum_{\ell=1}^{2}\beta_{\ell,{i},r_{i}}\xi_{\ell}^{i})+E\sum_{\ell=1}^{2}\beta_{\ell,{i},r_{i}}\xi_{\ell}^{i}\right)\right).
$$
It follows from  $ H_{\mathbf r}  DQ' P_{\mathbf r}|_{\Sigma\setminus  \cup_{i}D^{i}(r_{i}/2)}=Id,\;H_{\mathbf r}  DQ' P_{\mathbf r}|_{ \cup_{i}D^{i}(3r_{i}/2)}=0$, \eqref{eqn_r_beta}
  and $\frac{\p E}{\p r_{j}}|_{V_{i}}=0$ that
\begin{equation}
supp\frac{\p}{\p r_{i}}\left(H_{\mathbf r}  DQ' P_{\mathbf r}\right)\subset  \bigcup_{j=1}^{\mathfrak{e}}V_{j},\;\;\;\; supp \frac{\p^2}{\p r_{i}\p r_{j}} \left(H_{\mathbf r}  DQ' P_{\mathbf r}\right)\subset  V_{j}\cup V_{i}.
\end{equation}
Taking the derivative $\frac{\p}{\p r_{i}}$ and $\frac{\p^2}{\p r_{i}\p r_{j}}$ of  $\LH(DQ')\LP,$ using \eqref{eqn_r_HP}, \eqref{eqn_r_beta}, \eqref{eqn_est_xi_l} and \eqref{eqn_2nd_xi}
one can easily check that
\begin{align} \label{sec_dq}
&\sum_{l=1}^{2}\left\|\left.\frac{\p}{\p r_{j}}\left(H_{\mathbf r}  DQ' P_{\mathbf r}\right)(\eta)\right|_{W_{l}^{i}}\right\|_{k-1,2,\alpha} + \left\|\frac{\p^2}{\p r_{i}\p r_{j}}\left(H_{\mathbf r}   DQ'   P_{\mathbf r}\right) (\eta)\right\|_{k-2,2,\alpha} \\
\leq &Ce^{-\frac{(\fc-\alpha)r_{j}}{4}} \left\|\eta|_{V_{i}}\right\|_{k-1,2,\alpha}
+Ce^{-\frac{(\fc-\alpha)r_{i}}{4}} \left\|\eta|_{V_{j}}\right\|_{k-1,2,\alpha}.\nonumber
\end{align}
Note that
$$
\frac{\p}{\p r_{i}}\left(\LH(DQ')P_{\mathbf r}\right)\circ \LH(DQ')^{-1}P_{\mathbf r}+\LH(DQ')P_{\mathbf r}\circ \frac{\p}{\p r_{i}}\left(\LH(DQ')^{-1}P_{\mathbf r}\right)=\frac{\p \LH \LP}{\p r_{i}}.
$$
Multiplying $\LH(DQ')^{-1}P_{\mathbf r}$ on the both sides, by
$$
\frac{\p}{\p r_{i}}\left(\LH(DQ')^{-1}P_{\mathbf r}\right)= \LH \LP\frac{\p}{\p r_{i}}\left(\LH(DQ')^{-1}P_{\mathbf r}\right)+\frac{\p \LH \LP}{\p r_{i}}\LH(DQ')^{-1}P_{\mathbf r},
$$
we have
\begin{align}\label{eqn_r_DQ-1}
\frac{\p}{\p r_{i}}\left(\LH(DQ')^{-1}P_{\mathbf r}\right)=&\LH(DQ')^{-1}P_{\mathbf r}\frac{\p \LH \LP}{\p r_{i}}+\frac{\p \LH \LP}{\p r_{i}}\LH(DQ')^{-1}P_{\mathbf r}\nonumber\\
 &-\LH(DQ')^{-1}P_{\mathbf r}\circ \frac{\p}{\p r_{i}}\left(\LH(DQ')P_{\mathbf r}\right)\circ \LH(DQ')^{-1}P_{\mathbf r}.
\end{align}
Using (4), (5) of Lemma \ref{lem_Q'_est}, \eqref{eqn_r_HP} to the first term,  and  \eqref{eqn_r_HP} to the second term, applying  (4) of  Lemma \ref{lem_Q'_est}  and \eqref{sec_dq} to the last term   we have
\begin{equation}\label{eqn_r_DQ-1_j}
\sum_{l=1}^{2}\left\|\left.\frac{\p}{\p r_{i}}\left(\LH(DQ')^{-1}P_{\mathbf r}\right)\right|_{W_{l}^{j}}\right\|_{k-1,2,\alpha}\leq Ce^{-\frac{(\fc-\alpha)r_{i}}{4}} \left\|\eta|_{V_{j}}\right\|_{k-1,2,\alpha}.
\end{equation}
Taking derivative $\frac{\p}{\p r_{j}}$ of \eqref{eqn_r_DQ-1}, by \eqref{eqn_r_HP} we get
\begin{align*}
\frac{\p^2}{\p r_{i}\p r_{j}}\left(\LH(DQ')^{-1}P_{\mathbf r}\right)=&
\frac{\p}{\p r_{j}}\left(\LH(DQ')P_{\mathbf r}\right)\frac{\p \LH \LP}{\p r_{i}}+\frac{\p \LH \LP}{\p r_{i}}\frac{\p}{\p r_{j}}\left(\LH(DQ')P_{\mathbf r}\right) \\
&-\frac{\p}{\p r_{j}}\left(\LH(DQ')^{-1}P_{\mathbf r}\right)\circ \frac{\p}{\p r_{i}}\left(\LH(DQ')P_{\mathbf r}\right)\circ \LH(DQ')^{-1}P_{\mathbf r}\\
&-\LH(DQ')^{-1}P_{\mathbf r}\circ \frac{\p}{\p r_{i}}\left(\LH(DQ')P_{\mathbf r}\right)\circ \frac{\p}{\p r_{j}}\left(\LH(DQ')^{-1}P_{\mathbf r}\right)\\
&-\LH(DQ')^{-1}P_{\mathbf r}\circ \frac{\p^2}{\p r_{i}\p r_{j}}\left(\LH(DQ')P_{\mathbf r}\right)\circ \LH(DQ')^{-1}P_{\mathbf r}.
\end{align*}
  By (3), (4) of Lemma   \ref{lem_Q'_est},  \eqref{eqn_r_HP}, \eqref{eqn_r_DQ-1_j} and \eqref{sec_dq}  one can check that
\begin{align} \label{sec_dq-1}
&  \left\|\frac{\p^2}{\p r_{i}\p r_{j}}\left(H_{\mathbf r}  (DQ') ^{-1} P_{\mathbf r}\right) (\eta)\right\|_{k-2,2,\alpha} \\
\leq &Ce^{-\frac{(\fc-\alpha)r_{j}}{4}} \left\|\eta|_{V_{i}}\right\|_{k-1,2,\alpha}
+Ce^{-\frac{(\fc-\alpha)r_{i}}{4}} \left\|\eta|_{V_{j}}\right\|_{k-1,2,\alpha}.\nonumber
\end{align}
By  \eqref{sec_app_right}, \eqref{sec_dq-1} and
\begin{align*}
 \frac{\p^{2}Q^{*} \LP}{\p r_{i}\p r_{j}}=& \frac{\p^{2}(Q')^{*} \LP}{\p r_{i}\p r_{j}}\circ H_{\mathbf r}  (DQ') ^{-1} P_{\mathbf r}+(Q')^{*} \LP\frac{\p^2}{\p r_{i}\p r_{j}}\left(H_{\mathbf r}  (DQ') ^{-1} P_{\mathbf r}\right) \\
 &+\frac{\p(Q')^{*} \LP}{ \p r_{j}}\circ \frac{\p}{\p r_{i}}\left(H_{\mathbf r}  (DQ') ^{-1} P_{\mathbf r}\right)+\frac{\p(Q')^{*} \LP}{\p r_{i}}\circ \frac{\p}{\p r_{j}}\left(H_{\mathbf r}  (DQ') ^{-1} P_{\mathbf r}\right),
 \end{align*}
we have
\begin{align} \label{eqn_q_2nd}
&\sum_{l=1}^{2}\left\|\left.\frac{\p}{\p r_{j}}\left( Q^* P_{\mathbf r}\right)(\eta)\right|_{W_{l}^{i}}\right\|_{k-1,2,\alpha} + \left\|\frac{\p^2}{\p r_{i}\p r_{j}}\left(Q^*  P_{\mathbf r}\right) (\eta)\right\|_{k-2,2,\alpha} \\
\leq &Ce^{-\frac{(\fc-\alpha)r_{j}}{4}} \left\|\eta|_{V_{i}}\right\|_{k-1,2,\alpha}
+Ce^{-\frac{(\fc-\alpha)r_{i}}{4}} \left\|\eta|_{V_{j}}\right\|_{k-1,2,\alpha}.\nonumber
\end{align}
Since for any $\zeta+\hat \zeta_{0}\in Ker D^{\widetilde {\mathbf L}}_{b_{o}}$
$$
D\phi_{\mathbf r}(\zeta_{(\mathbf r)}^{*}=\sum (\bar{\partial}\beta_{i;r}) \zeta_{i}+\sum  \beta_{i;r}(E-E^{\widetilde{\mathbf L}}_{u_{i}})(\zeta_{i}+\hat \zeta_{0}),
$$
  we have
$$
supp\;\LH D\phi_{\mathbf r}(\zeta_{(\mathbf r)}^{*})\subset \cup_{i} V_{i},\;\;\;supp\;\frac{\p}{\p r_{i}}\left(\LH D\phi_{\mathbf r}(\zeta_{(\mathbf r)}^{*})\right)\subset  V_{i},\;\;\;\;\frac{\p^2}{\p r_{i}\p r_{j}}\left(\LH D\phi_{\mathbf r}(\zeta_{(\mathbf r)}^{*})\right)=0.
$$
Since  $I_{r}^{*}(\zeta+\hat \zeta_{0})=(Id-Q^{*}\LP\circ \LH D\circ \phi_{\mathbf r} ) (\zeta_{(\mathbf r)}^{*})$,  \eqref{eqn_def_Dp-2} and \eqref{eqn_q_2nd}, we have
\begin{align}\label{eqn_I*_2}
& \sum_{l=1}^{2}\left\|\left.\frac{\p}{\p r_{j}}I^*_{(\mathbf r)}(\zeta+\hat \zeta_{0})\right|_{W_{l}^{i}}
\right\|_{k-1,2,\alpha} +\left\|\frac{\p^2}{\p r_{i}\p r_{j}}I^*_{\mathbf r}(\zeta+\hat \zeta_{0})
\right\|_{k-2,2,\alpha}
\\ \nonumber
\leq &\mathsf C e^{(\fc-\alpha)\frac{r_{i}+r_{j}}{2}} \|\zeta+\hat \zeta_0\|_{\mathcal W,k,2,\alpha}.
\end{align}
Note that, restricting in $V_{i},i\neq j$
$$
\tilde{\nabla}_{\frac{\p}{\p r_{j}}}\tilde \nu_{(\mathbf r)}=\phi_{\mathbf r}\tilde{\nabla}_{\frac{\p}{\p r_{j}}}  \nu^{*}_{(\mathbf r)},\;\;\;\;\;\tilde{\nabla}_{\frac{\p}{\p r_{j}}}\tilde h_{(\mathbf r)}=\phi_{\mathbf r}\tilde{\nabla}_{\frac{\p}{\p r_{j}}}  h^{*}_{(\mathbf r)},\;\;\;\;\;\frac{\p E}{\p r_{j}}=0.
$$
Similar \eqref{local_calculation_4}, by   Theorem \ref{coordinate_decay-2} we can prove that
\begin{align}\label{eqn_diff_rest}
\sum_{l=1}^{2}\left\|\left.H_{\mathbf r}\circ D  \phi_{\mathbf r} \circ \frac{\partial}{\partial r_{j}}   \nu^*_{(\mathbf r)} \right|_{W_{l}^{i}}\right\|_{k-1,2,\alpha}
\leq \mathsf{C}  \mathsf {d}\sum_{l=1}^{2}\left\|\left.\frac{\partial}{\partial r_{j}}  \nu^*_{(\mathbf r)} \right|_{W_{l}^{i}}\right\|_{k-1,2,\alpha}+ \mathsf{C}  e^{-(\fc-5\alpha)\frac{r_{i}+r_{j}}{4}}.
\end{align}
Using \eqref{sec_app_right}, \eqref{sec_dq}, \eqref{eqn_diff_rest} and   the same  argument as in \cite{LS-1}, we have
\begin{equation}\label{eqn_est_sol_rest}
\sum_{l=1}^{2}\left\|\frac{\partial }{\partial r_{i}}\left.\nu^*_{(\mathbf r)}\right|_{W_{l}^{j}} \right\|_{k-1,2,\alpha}+\sum_{l=1}^{2}\left\|\left.  \frac{\p(H_{\mathbf r} fI_{\mathbf r}(\zeta))}{\p r_{i}}\right|_{W_{l}^{j}}\right\|_{k-2,2,\alpha}\leq Ce^{-(\fc-5\alpha)\frac{r_{i}+r_{j}}{4}}.
\end{equation}
By \eqref{eqn_est_sol_rest}, Theorem \ref{coordinate_decay-2}, the Cauchy-Schwarz inequality and the same argument of \eqref{local_calculation_4}, we have
\begin{align}\label{sec_est_diff}
&\left\|H_{\mathbf r} D\phi_{\mathbf r}\circ  \frac{\p^2}{\p r_{i}\p r_{j}} \nu^*_{(\mathbf r)}\right\|_{k-2,2,\alpha}
\leq  \mathsf{C} \left[ \mathsf {d}\left\|\frac{\p^2}{\p r_{i}\p r_{j}}\nu^*_{(\mathbf r)}\right\|_{k-2,2,\alpha}+ e^{-(\fc-5\alpha)\frac{r_{i}+r_{j}}{4}}\right].
\end{align}
Taking the derivative $\frac{\p^2}{\p r_{i}\p r_{j}}$ of  $\nu^*_{(\mathbf r)}$ and multiplying $H_{\mathbf r} D\phi_{\mathbf r}$ on both sides  we get
\begin{align*}
&H_{\mathbf r} D\phi_{\mathbf r}\circ\frac{\p^2(\nu^*_{(\mathbf r)})}{\p r_{i}\p r_{j}} \\
= &H_{\mathbf r} D\phi_{\mathbf r}\circ\frac{\p^2I^*_{\mathbf r}(\zeta) }{\partial r_{i}\p r_{j}} + H_{\mathbf r} D\phi_{\mathbf r}\circ\frac{\p^2(Q^* P_{\mathbf r} )}{\p r_{i}\p r_{j}} \circ H_{\mathbf r} fI_{\mathbf r}(\zeta)  + H_{\mathbf r} P_{\mathbf r} \frac{\p^2(H_{\mathbf r} fI_{\mathbf r}(\zeta))}{\p r_{i}\p r_{j}} \\
&+ H_{\mathbf r} D\phi_{\mathbf r}\circ\frac{\p (Q^* P_{\mathbf r})}{\p r_{i}}\frac{\p(H_{\mathbf r} fI_{\mathbf r}(\zeta))}{\p r_{j}}+H_{\mathbf r} D\phi_{\mathbf r}\circ \frac{\p (Q^* P_{\mathbf r})}{\p r_{j}}\frac{\p(H_{\mathbf r} fI_{\mathbf r}(\zeta))}{\p r_{i}}  .
\end{align*}
By  $\frac{\p(H_{\mathbf r}P_{\mathbf r})}{\p r_{j}}\subset V_{j}$,  \eqref{estimate_operator_partial_I},  \eqref{sec_app_right}, \eqref{sec_dq}, \eqref{eqn_est_sol_rest} and  $$\frac{\partial}{\partial r_{j}}(H_{\mathbf r}f(I_{\mathbf r}(\zeta)))=\frac{\p(H_{\mathbf r}P_{\mathbf r})}{\p r_{j}}\circ H_{\mathbf r}f(I_{\mathbf r}(\zeta))+H_{\mathbf r}P_{\mathbf r}\frac{\partial}{\partial r_{j}}(H_{\mathbf r}f(I_{\mathbf r}(\zeta))),$$ using Lemma \ref{lem_Q'_est} and Theorem \ref{coordinate_decay-2}, we get
\begin{equation}\label{eqn_Qf_1}
\left\|\frac{\p (Q^* P_{\mathbf r})}{\p r_{i}}\frac{\p(H_{\mathbf r} fI_{\mathbf r}(\xi))}{\p r_{j}}\right\|_{k-2,2,\alpha}\leq C e^{-(\fc-5\alpha)\frac{r_{i}+r_{j}}{4}}  .
\end{equation}
Then using Theorem \ref{coordinate_decay-2} and  repeating the proof of  \eqref{estimate_operator_partial_I} we have
\begin{align}\label{eqn_Hf_2}
&\left\|   \frac{\p^2(H_{\mathbf r} fI_{\mathbf r}(\zeta))}{\p r_{i}\p r_{j}}\right\|_{k-2,2,\alpha}\leq C  e^{-(\mathfrak{c} -5\alpha)\frac{r_{i}+r_{j}}{4} }  +C\mathsf{d}\left\| \frac{\partial^2}{\p r_{i}\p r_{j}}\nu^*_{(\mathbf r)}\right\|_{k-2,2,\alpha}
\end{align}
By the definition of $\nu_{(r)}^{*}$ we have
\begin{align*}
 \frac{\p^2\nu^*_{(\mathbf r)}}{\p r_{i}\p r_{j}}
= &\frac{\p^2I^*_{\mathbf r}(\zeta) }{\partial r_{i}\p r_{j}} + \frac{\p^2(Q^* P_{\mathbf r} )}{\p r_{i}\p r_{j}} \circ H_{\mathbf r} fI_{\mathbf r}(\zeta)  + Q^* P_{\mathbf r} \circ \frac{\p^2(H_{\mathbf r} fI_{\mathbf r}(\zeta))}{\p r_{i}\p r_{j}} \\
&+ \frac{\p (Q^* P_{\mathbf r})}{\p r_{i}}\frac{\p(H_{\mathbf r} fI_{\mathbf r}(\zeta))}{\p r_{j}}+  \frac{\p (Q^* P_{\mathbf r})}{\p r_{j}}\frac{\p(H_{\mathbf r} fI_{\mathbf r}(\zeta))}{\p r_{i}}  .
\end{align*}
Applying \eqref{eqn_I*_2} to the first term, \eqref{eqn_q_2nd} to the second term, \eqref{eqn_Hf_2} to the third term, and   \eqref{eqn_Qf_1} to the last two term we can obtain the estimate of \eqref{eqn_est_L_2nd}.
\v\v

\section{Appendix}

\subsection{Implicit function theorem}\label{Implicit function theorem}

 We can generalize Theorem A.3.3 and Proposition A.3.4 in \cite{MS} to the case with parameters by the same method.
\begin{theorem}\label{details_implicit_function_theorem}
Let  $(A,\|\cdot\|_{A})$,  $(X,\|\cdot\|_{X})$ and $(Y,\|\cdot\|_{Y})$ be Banach spaces, $U\subset X$  be open sets and  $V\subset A$, $U\subset X$  be open sets and $F: V\times U\rightarrow Y$ be a continuously differentiable map.   For any $(a,x)\in V\times U$ define
$$
D_{a}F(a,x)(g)=\frac{d}{dt}F(a+tg ,x)|_{t=0},\;\;D_{x}F(a,x)(h)=\frac{d}{dt}F(a ,x+th)|_{t=0},\;\;\;\forall\;g\in A,\; h\in X.
$$
Suppose that   $D_{x}F(a_{o},x_{o})$  is surjective and has a bounded linear right inverse $Q_{(a_{o},x_{o})}:Y\longrightarrow X$ with $\|Q_{(a_{o},x_{o})}\|\leq \mathsf C$ for some constant $\mathsf C>0$.  Choose a positive constant $\delta >0$  such that
\begin{equation}\label{c_of_differential}
\|D_{x}F(a,x)- D_{x}F(a_{o},x_{o})\|\leq \frac{1}{2\mathsf C},\;\;\;\;\forall \;x\in B_{\delta}(x_{o},X),\;a\in B_{\delta}(a_{o},A).
\end{equation}
where $B_{\delta}(a_{o},A)=\{a\in A|\; \|a-a_{o}\|_{A} \leq \delta\},B_{\delta}(x_{o},X)=\{x\in X|\; \|x-x_{o}\|_{X} \leq \delta\}.$ Suppose that $x_1\in X$ and $a\in  B_{\delta}(a_{o},A)$ satisfies
\begin{equation}\label{small_value_of_F}
\|F(a,x_1)\|_{Y}<\frac{\delta}{4\mathsf C},\;\;\;\; \|x_{1}-x_{o}\|_{X}\leq \frac{\delta}{8}.
\end{equation}
Then there exists a unique $x\in X$ such that
\begin{equation}
F(a,x)=0,\;\;\;\;x-x_1\in \im\; Q,\;\;\;\;\|x-x_{o}\|_{X}\leq \delta,\;\;\;\; \|x-x_{1}\|_{X}\leq 2\mathsf C\|F(a,x_{1})\|_{Y}.
\end{equation}
Moreover, if $\|F(a_{o},x_{o})\|_{Y}\leq \frac{\delta}{4\mathsf C},$ there exist a constant $\delta'>0$ and a unique family differential map  $f_{a}:\ker D_{x}F(a_{o},x_{o})\rightarrow Y$ such that for any $(a,x)\in F^{-1}(0)\cap (B_{\delta'}(a_{o},A)\times B_{\delta'}(x_{o},X)),$ we have
 \begin{equation}\label{eqn_implicit_F}
 F(a,x)=0\Longleftrightarrow
 x=x_{o}+\zeta + Q_{(a_{o},x_{o})}\circ f_{a}(\zeta),\;\;\; \zeta \in \ker\;  D_{x}F(a_{o},x_{o})
 \end{equation}
\end{theorem}
 Using Theorem \ref{details_implicit_function_theorem} we can obtain the smoothness of implicit function.
 \begin{theorem}\label{smooth_implicit_function_theorem} Let $F$ satisfies the assumption of Theorem \ref{details_implicit_function_theorem}.
   If $F:V\times U\longrightarrow Y$ is of class $C^{\ell}$, where $\ell$ is a positive integer, then  there exists a constant $\delta'>0$ such that  $F^{-1}(0)|_{B_{\delta'}(a_{o},A)\times  B_{\delta'}(x_{o},X)}$ is $C^{\ell}$ manifold,  and $\xi\rightarrow x_{o} + \xi + Q \circ f_{a}(\xi)$ is a $C^{\ell}$-chart of  $F^{-1}(0)|_{B_{\delta'}(a_{o},A)\times  B_{\delta'}(x_{o},X)}$.
 In particular,
\begin{equation}\label{eqn_impli_est_F}
\|D_{a} \left(x_{o}+\zeta + Q_{(a_{o},x_{o})}\circ f_{a}(\zeta)\right)\|\leq C,
\end{equation}
where $C>0$ is a constant depending only on $C_{1}$, $\mathsf C,\delta'$, $\|f_a\|$ and  $\|D^2_{ax}F(a,x_{o})\|$.
 \end{theorem}
\begin{proof} Since $F(a,x)$ satisfies the assumption of Theorem \ref{details_implicit_function_theorem},  $F^{-1}(0)|_{\{a\}\times  B_{\delta_{1}}(x_{o},X)}$ is a smooth manifold. We only need consider the smoothness of $F^{-1}(0)$ with respect to $a.$

\v
  By the same argument in the proof of  Theorem A.3.3 in \cite{MS}, we have a explicit formula for $f_a$
$$f_a(\zeta)= D_{x}F(a_{o},x_{o})\circ \phi^{-1}_{a}(\zeta+x_{o}) -  D_{x}F(a_{o},x_{o})(x_{o}),$$
where $\phi_{a}$ is defined by
\begin{equation} \label{eqn_psi}
\phi_{a}(x):=x+ Q_{(a_{o},x_{o})}\left(F (a,x)-  D_{x}F{(a_{o},x_{o})}(x-x_{o})\right).\end{equation}
We choose $\delta'$ small such that in $B_{\delta'}(a_{o},A)\times B_{\delta'}(x_{o},X),$
\begin{equation} \label{eqn_psi-1}
| \phi_{a}(x)-I|\leq \frac{1}{2}. \end{equation}
 Then by the smoothness of $F$ and
 $$
 \frac{\p}{\p a}\phi^{-1}_{a}(x)=-\phi^{-1}_{a}\circ\frac{\p\phi_{a}}{\p a}\circ \phi^{-1}_{a}(x)
 $$
 we conclude that
$f_{a}$ is a smooth function of $(a,x).$
It follows that the zero set of $F$ is smooth for $a$ and
\eqref{eqn_impli_est_F} holds.
\end{proof}

\subsection{Exponential decay in tube}
By the same method as in \cite{LS-1}, we can prove the following lemmas
\begin{lemma}\label{tube_ker_c-1}
	Let $\eta\in \lka$ and $h +\hat h_{0}\in \cwk$ be a solution of $D^{\widetilde{\mathbf{L}}}|_{b}(h+\hat h_{0})=\eta$ over $\Sigma\setminus \Sigma(R_{0})$.
	Suppose that, for any $p,q\geq 0,$
	\begin{equation}\label{exp_decay_Ss}
	\left|\frac{\p^{p+q} E^{\widetilde{\mathbf{L}}}_{b_{(r)}} }{\partial s_{1}^ p \partial t_{1} ^q}\right|\leq
	C_{p,q}e^{-\fc |s_i|} ,\;\;\;\;\;\forall \; |s_i|\geq R_0,\;\; l=1,2
	\end{equation} for some constant $C_{p,q}>0.$
	Then  for any $0<\alpha< \frac{\fc}{2} $, there exists a constant $\mathsf{C} >0$ such that for any $R>\max\{R_{0},\bar d\}$ and  $R'>2+R$
	\begin{align}\label{t_ker_1}
	\left\|h \mid_{s_{1}\geq R' }\right\|_{k,2,\alpha} \leq \mathsf{C} \left( (e^{-(\fc-\alpha)(R'-  R)}  +  e^{-( \fc -\alpha) R} ) \left\|h+\hat h_{0}  \right\|_{\mathcal W,k,2,\alpha } + \left\|\eta\mid_{s_{1}\geq R }  \right\|_{k-1,2,\alpha }  \right)
	\end{align}
	In particular, if  $D^{\widetilde{\mathbf{L}}}|_{b}$ has a bounded right inverse $Q_{b}:\lka\to \wka$. Let $h=Q_{b}\eta$ be a solution of $D^{\widetilde{\mathbf{L}}}|_{b}(h) = \eta$ over $(R_0,\infty)\times S^1$.   Then there exists a constant $\mathsf{C}' >0$ independent of $r$ such that
	\begin{equation} \label{f_tube_estimate_curve}
	\left\|h\mid_{s_{1}\geq R'}\right\|_{k,2,\alpha} \leq \mathsf{C}' \left[\left({e^{-(\fc-\alpha)(R'-  R)}}+e^{-(\fc-\alpha) R }\right) \|\eta \|_{k-1,2,\alpha}+ \left\|\eta\mid_{  s_{1} \geq  R }  \right\|_{k-1,2,\alpha}\right].
	\end{equation}
\end{lemma}

\begin{lemma}\label{tube_ker}
Let  $h +\hat h_{0}\in \wkr$ be a solution of $D^{\widetilde{\mathbf L}}|_{b} (h+\hat h_{0})=0$ over $\Sigma_{(r)}\setminus \Sigma(R_{0})$.
Suppose that, for any $p,q\geq 0,$
\begin{equation}\label{exp_decay_Ss}
 \left|\frac{\p^{p+q} E^{\widetilde{\mathbf L}}_{b} }{\partial s_{1}^ p \partial t_{1} ^q}\right|\leq
C_{p,q}e^{-\fc \min(s_{1}, 2lr-s_{1})} ,\;\;\;\;\;\forall \; |s_i|\geq R_0,\;\; l=1,2
\end{equation} for some constant $C_{p,q}>0.$
Then  for any $0<\alpha< \frac{\fc}{2} $, there exists a constant $\mathsf{C} >0$ such that for any $R>\max\{R_{0},\bar d\}$ and  $R'>2+R$
 \begin{align}\label{t_ker_1}
  \left\|h \mid_{R'\leq s_{1}\leq 2lr-R' }\right\|_{k,2,\alpha} \leq \mathsf{C} (e^{-(\fc-\alpha)(R'-  R)}  +  e^{-( \fc -\alpha) R} ) \left\|h+\hat h_{0}  \right\|_{\mathcal W,k,2,\alpha }
  \end{align}
\end{lemma}

\end{document}